\documentclass[11pt]{amsart}

\usepackage{
	amsmath,
	amsfonts,
	amssymb,
	amsthm,
	amscd,
	comment,
	enumitem,
	etoolbox,
	gensymb,
	mathtools,
	mathdots,
	stmaryrd,
}
\usepackage[usenames,dvipsnames]{xcolor}
\usepackage{todonotes}
\usepackage[all]{xy}
\usepackage[hyphens]{url}
\usepackage[hyphenbreaks]{breakurl}
\usepackage{xparse}

\usepackage[utf8]{inputenc}  
\usepackage{bbm}
\usepackage[colorlinks=true, linkcolor=blue, citecolor=blue, urlcolor=blue, breaklinks=true]{hyperref}

\usepackage[capitalize]{cleveref} 
\crefname{defin}{Definition}{Definitions}
\crefname{eg}{Example}{Examples}
\crefname{egs}{Example}{Examples}
\crefname{lem}{Lemma}{Lemmas}
\crefname{theo}{Theorem}{Theorems}
\crefname{equation}{}{}
\crefname{enumi}{}{}

\newcommand{\N}{\mathbb{N}}

\newcommand{\Z}{\mathbb{Z}}
\newcommand{\kk}{\Bbbk}
\newcommand{\one}{\mathbbm{1}}
\newcommand{\CC}{\mathcal{C}}
\newcommand{\DD}{\mathcal{D}}
\newcommand{\B}{\mathcal{B}}

\newcommand{\OB}{\mathcal{OB}}

\newcommand{\isom}{\cong}
\newcommand{\g}{\mathfrak{g}}
\newcommand{\gl}{\mathfrak{gl}}
\newcommand{\so}{\mathfrak{so}}
\renewcommand{\sp}{\mathfrak{sp}}
\newcommand{\osp}{\mathfrak{osp}}
\newcommand{\ev}{\textup{ev}}
\newcommand{\ngle}[1]{\langle #1 \rangle}
\newcommand{\dashmod}{\textup{-mod}}
\newcommand{\kkvec}{\textup{Vec}_\kk}

\DeclareMathOperator{\End}{End}
\DeclareMathOperator{\Hom}{Hom}
\DeclareMathOperator{\ob}{ob}

\DeclareMathOperator{\id}{id}
\DeclareMathOperator{\im}{im}
\DeclareMathOperator{\Curr}{Curr}
\DeclareMathOperator{\Kar}{Kar}

\DeclareMathOperator{\sgn}{sgn}

\newcommand{\bm}{\begin{bmatrix}}
\newcommand{\nm}{\end{bmatrix}}
\newcommand{\summ}{\sum\limits}

\newcommand{\go}{\mathord{\mathsf{I}}}
\newcommand{\goup}{\mathord{\uparrow}}
\newcommand{\godown}{\mathord{\downarrow}}

\usepackage{tikz}
\usetikzlibrary{arrows.meta}
\usetikzlibrary{decorations.markings}
\usetikzlibrary{calc}

\tikzset{anchorbase/.style={>=To,baseline={([yshift=-0.5ex]current bounding box.center)}}}
\tikzset{ % Syntax: \begin{tikzpicture}[centerzero={0,0.2}]
	centerzero/.style={>=To,baseline={([yshift=-0.5ex](#1))}},
	centerzero/.default={0,0}
}
\tikzset{wipe/.style={white,line width=3pt}}

\tikzset{gmod/.style={ForestGreen, thick}}
\tikzset{rmod/.style={red, thick}}
\tikzset{blmod/.style={blue, thick}}
\tikzset{bmod/.style={black, thick}}
\tikzset{gmod>/.style={->, thick, ForestGreen}}
\tikzset{<gmod/.style={<-, thick, ForestGreen}}

\newcommand\triup[3]{\draw[fill=white, draw=black]  ($#1 + (0, 0.125) $) -- ++(0.1445, -0.25) -- ++(-0.289, 0) -- cycle; 
\path #1 node[anchor=#2] {\hspace{2px}$\scriptstyle #3$\hspace*{1px}};}
\newcommand\tridown[3]{\draw[fill=white, draw=black] ($#1 + (0, -0.125) $) -- ++(-0.1445, 0.25) -- ++(0.289, 0) -- cycle; \path #1 node[anchor=#2] {{\hspace{2px}$\scriptstyle #3$\hspace*{2px}}};}
\newcommand{\coupon}[2]{ % \coupon{position}{label}
	\draw (#1) node[inner sep=2pt,draw,fill=white,rounded corners] {$\scriptstyle{#2}$}
}
\newcommand\dotlabel[1]{$\scriptstyle{#1}$}
\newcommand\adot[1]{\filldraw[fill=white, draw=black] #1 circle (1.5pt)}
\newcommand\bdot[1]{\filldraw[fill=black, draw=black] #1 circle (1.5pt)}

\newcommand\bsquare[1]{\filldraw[fill=black, draw=black] ($#1 +(-0.075, -0.075)$) -- ($#1 +(0.075, -0.075)$) -- ($#1 + (0.075, 0.075)$) -- ($#1 + (-0.075, 0.075)$)}
\newcommand\atoken[4][0px]{%  \atoken[vertical spacing]{position}{anchor}{label}
	\filldraw[fill=white, draw=black] (#2) circle (1.5pt) node[anchor=#3] {\raisebox{#1}{\dotlabel{#4}}}
}
\newcommand\btoken[4][0px]{%  \btoken[vertical spacing]{position}{anchor}{label}
	\filldraw[black] (#2) circle (1.5pt) node[anchor=#3] {\raisebox{#1}{\dotlabel{#4}}}
}

\newcommand\bsquaretoken[4][0px]{%  \bsquaretoken[vertical spacing]{position}{anchor}{label}
	\filldraw[fill=black, draw=black] ($(#2) +(-0.075, -0.075)$) -- ($(#2) +(0.075, -0.075)$) -- ($(#2) + (0.075, 0.075)$) -- ($(#2) + (-0.075, 0.075)$) -- ($(#2) +(-0.075, -0.075)$) node[anchor=#3] {\hspace{2px}\raisebox{#1}{\dotlabel{#4}}\hspace*{0px}}
}

\newcommand{\dcap}{
	\begin{tikzpicture}[anchorbase]
	\draw[bmod] (0,0) -- (0,0.125) arc (180:0:0.25) -- (0.5,0);
	\end{tikzpicture}	
}
\newcommand{\dcup}{
	\begin{tikzpicture}[anchorbase]
	\draw[bmod] (0,0.125) -- (0,0) arc (180:360:0.25) -- (0.5,0.125);
	\end{tikzpicture}	
}
\newcommand{\dmult}{
	\begin{tikzpicture}[anchorbase]
	\draw[bmod] (0, 0) -- (0.25, 0.25) -- (0.25, 0.5);
	\draw[bmod] (0.5, 0) -- (0.25, 0.25);
	\end{tikzpicture}	
}

\newcommand{\dmultred}{
	\begin{tikzpicture}[anchorbase]
	\draw[rmod] (0, 0) -- (0.25, 0.25) -- (0.25, 0.5);
	\draw[rmod] (0.5, 0) -- (0.25, 0.25);
	\end{tikzpicture}	
}
\newcommand{\dmodulemult}{
	\begin{tikzpicture}[anchorbase]
	\draw[bmod] (0, 0) -- (0.25, 0.25);
	\draw[gmod] (0.5, 0) -- (0.25, 0.25) -- (0.25, 0.5);
	\end{tikzpicture}	
}

\newcommand{\damapmult}[2][0px]{
	\begin{tikzpicture}[anchorbase]
	\draw[bmod] (0, 0) -- (0.25, 0.25);
	\draw[gmod] (0.5, 0) -- (0.25, 0.25) -- (0.25, 0.5);
	\btoken[#1]{0.25, 0.25}{west}{#2};
	\end{tikzpicture}	
}
\newcommand{\damapmultred}[2][0px]{
	\begin{tikzpicture}[anchorbase]
	\draw[bmod] (0, 0) -- (0.25, 0.25);
	\draw[rmod] (0.5, 0) -- (0.25, 0.25) -- (0.25, 0.5);
	\btoken[#1]{0.25, 0.25}{west}{#2};
	\end{tikzpicture}	
}
\newcommand{\damapmultblue}[2][0px]{
	\begin{tikzpicture}[anchorbase]
	\draw[bmod] (0, 0) -- (0.25, 0.25);
	\draw[blmod] (0.5, 0) -- (0.25, 0.25) -- (0.25, 0.5);
	\btoken[#1]{0.25, 0.25}{west}{#2};
	\end{tikzpicture}	
}
\newcommand{\deamapmult}[3][0px]{
	\begin{tikzpicture}[anchorbase]
	\draw[bmod] (0, 0) -- (0.25, 0.25);
	\draw[gmod] (0.5, 0) -- (0.25, 0.25) -- (0.25, 0.5);
	\btoken[#1]{0.25, 0.25}{west}{#2};
	\node[anchor=north] at (0, 0) {$\scriptstyle #3$};
	\node[anchor=south] at (0.25, 0.5) {$\vphantom{\scriptstyle #3}$};
	\end{tikzpicture}	
}
\newcommand{\dmoduleendo}{
	\begin{tikzpicture}[anchorbase]
	\draw[gmod] (0, 0) -- (0, 0.5);
	\bdot{(0, 0.25)};
	\end{tikzpicture}	
}
\newcommand{\dcross}{
	\begin{tikzpicture}[anchorbase]
	\draw[bmod] (0, 0) -- (0.5, 0.5);
	\draw[bmod] (0.5, 0) -- (0, 0.5);
	\end{tikzpicture}	
}
\newcommand{\dcrossup}{
	\begin{tikzpicture}[anchorbase]
	\draw[->, thick] (0, 0) -- (0.5, 0.5);
	\draw[->, thick] (0.5, 0) -- (0, 0.5);
	\end{tikzpicture}	
}

\newcommand{\dcrossright}{
\begin{tikzpicture}[anchorbase]
\draw[->, thick] (0, 0) -- (0.5, 0.5);
\draw[<-, thick] (0.5, 0) -- (0, 0.5);
\end{tikzpicture}	
}
\newcommand{\dcrossleft}{
	\begin{tikzpicture}[anchorbase]
	\draw[<-, thick] (0, 0) -- (0.5, 0.5);
	\draw[->, thick] (0.5, 0) -- (0, 0.5);
	\end{tikzpicture}	
}
\newcommand{\dcapupdown}{
	\begin{tikzpicture}[anchorbase]
	\draw[->, thick] (0,0) -- (0,0.125) arc (180:0:0.25) -- (0.5,0);
	\end{tikzpicture}	
}
\newcommand{\dcapdownup}{
	\begin{tikzpicture}[anchorbase]
	\draw[<-, thick] (0,0) -- (0,0.125) arc (180:0:0.25) -- (0.5,0);
	\end{tikzpicture}	
}
\newcommand{\dcupupdown}{
	\begin{tikzpicture}[anchorbase]
	\draw[<-, thick] (0,0.125) -- (0,0) arc (180:360:0.25) -- (0.5,0.125);
	\end{tikzpicture}	
}
\newcommand{\dcupdownup}{
	\begin{tikzpicture}[anchorbase]
	\draw[->, thick] (0,0.125) -- (0,0) arc (180:360:0.25) -- (0.5,0.125);
	\end{tikzpicture}	
}

\newtheorem{theo}{Theorem}[section]
\newtheorem{prop}[theo]{Proposition}
\newtheorem{lem}[theo]{Lemma}

\newtheorem{conj}[theo]{Conjecture}

\theoremstyle{definition}
\newtheorem{defin}[theo]{Definition}
\newtheorem{rem}[theo]{Remark}
\newtheorem{eg}[theo]{Example}
\newtheorem{egs}[theo]{Examples}

\numberwithin{equation}{section}
\allowdisplaybreaks

\setenumerate[1]{label=(\alph*)}

\usepackage{environ}
\newcommand{\comsize}{3 in}
\newcommand{\commentary}[1]{\begin{minipage}{\comsize} #1 \end{minipage}}

\NewEnviron{commeqn}[1][2.5 in]{\renewcommand{\comsize}{#1} \begin{align*} \BODY \end{align*}}
\graphicspath{{/}}

\newtoggle{comments}
\newtoggle{details}
\newtoggle{detailsnote}

\iftoggle{comments}{%
	\newcommand{\scomments}[1]{
		\ \\
		{\color{Violet}
			\textbf{SSS:} #1
		}
		\ \\
	}
}{%
	\newcommand{\scomments}[1]{}
}

\iftoggle{details}{%
	\newcommand{\details}[1]{
		\ \\
		{\color{OliveGreen}
			\textbf{Details:} #1
		}
		\\
	}
}{%
	\newcommand{\details}[1]{}
}

\newcommand{\red}[1]{{\mathcolor{red}{#1}}}
\newcommand{\blue}[1]{{\mathcolor{blue}{#1}}}
\newcommand{\green}[1]{{\mathcolor{ForestGreen}{#1}}}

\begin{document}
%===============

\title{Towards interpolating categories for equivariant map algebras}
\author{Saima Samchuck-Schnarch}



\maketitle
\thispagestyle{empty}

\begin{abstract}
	Using the language of string diagrams, we define categorical generalizations of modules for map algebras $\g \otimes A$ and equivariant map algebras $(\g \otimes A)^\Gamma$, where $\g$ is a Lie algebra, $A$ is a commutative associative algebra, and $\Gamma$ is an abelian group acting on $\g$ and $A$. After establishing some properties of these modules, we present several examples of how our definitions can applied in various diagrammatic categories. In particular, we use the oriented Brauer category $\OB$ to construct a candidate interpolating category for the categories of $\gl_n \otimes \kk[t]$-modules.
\end{abstract}

\section{Introduction}

Given a Lie algebra $\g$ and a commutative associative algebra $A$, $\g \otimes A$ naturally becomes a Lie algebra when equipped with the Lie bracket extended from $[x \otimes a, y \otimes b] = [x, y] \otimes ab$. Such Lie algebras, sometimes called \emph{map algebras}, have been studied in some depth for general $A$ (see e.g.\ \cite{feiginLoktev} and \cite{chariMapAlgebras}), and in more detail for certain choices of $A$. Writing $\kk$ for the ground field, taking $A = \kk[t]$ yields \emph{current algebras}, and taking $A = \kk[t, t^{-1}]$ yields \emph{loop algebras}. Both of these families of Lie algebras have applications to theoretical physics, and the representation theory of specific types of current and loop algebras (for instance, where $\g$ is a simple Lie algebra or $\g = \gl_n$) has been studied extensively in the literature. More generally, taking $A = \kk[t_1, t_2, \dotsc, t_k]$ or $A = \kk[t_1, t_2, \dotsc, t_k, t_1^{-1}, t_2^{-1}, \dotsc, t_k^{-1}]$ yield \emph{multicurrent} and \emph{multiloop algebras}, respectively, with connections to extended affine Lie algebras; see the exposition \cite{extendedAffineLieAlgebras} for an overview of this topic. More recently, \emph{equivariant map algebras}, a twisted version of the algebras $\g \otimes A$ that are defined with respect to an action of a group $\Gamma$ on $\g$ and on $A$, were studied in \cite{equivariantMapAlgebras}. These algebras and their modules encompass a wide range of previously-studied examples, including twisted (multi)current and (multi)loop algebras (see e.g.\ \cite{lauMultiloop}), the Onsager algebra and its generalizations, and the tetrahedron algebra. Phrasing theorems and constructions in terms of equivariant map algebras can streamline and unify proofs and results that previously appeared disparate; see e.g.\ the classification of irreducible finite-dimensional representations for equivariant map algebras in \cite{equivariantMapAlgebras}.

In recent years, string diagrammatic methods have proved to be a useful tool for investigating interpolating categories, and more generally for studying the representation theory of various Lie algebras and groups. For instance, the oriented Brauer category $\OB$ interpolates the categories of $\gl_n$-modules, and the unoriented Brauer category $\B$ interpolates both the categories of $\so_n$-modules and $\sp_{2n}$ modules (and more generally, the category of modules for the orthosymplectic Lie superalgebra $\mathfrak{osp}_{m \mid 2n}$). Working in such diagrammatic categories allows one to employ intuitive topological arguments, and often leads to natural generalizations, e.g.\ the Frobenius Brauer supercategories $\OB(A)$ and $\B(A)$ studied in \cite{orientedFrobeniusBrauer} and \cite{diagrammaticsForRealSupergroups}, respectively.

In this paper, we define and study categorical generalizations of equivariant map algebras and their modules, expressed in the language of string diagrams. On one hand, applying the content of this paper to the category of vector spaces provides a new graphical calculus for working with equivariant map algebras. This makes visualizations easier and allows for the aforementioned sorts of topological arguments. On the other hand, our definitions and constructions can be applied to other categories, and they are particularly natural in the context of diagrammatic categories like $\OB$ and $\B$. For example, for each $n \in \N$ one can consider $(L, \kk[t])$-modules in $\OB(n)$, where $L$ is a certain Lie algebra in the oriented Brauer category, and ``$(L, \kk[t])$-modules'' refers to our categorical generalization of modules for a current algebra $L \otimes \kk[t]$. These modules are mapped via the incarnation functor $I_n \colon \OB(n) \to \gl_n\dashmod$ to $\gl_n \otimes \kk[t]$-modules. Thus $(L, \kk[t])$-modules in $\OB$ in some sense capture the behaviour of $\gl_n \otimes \kk[t]$-modules for all $n$ at once. Moreover, one can consider $(L, \kk[t])$-modules in $\OB(\delta)$ even when the dimension parameter $\delta$ is not an integer. Such modules would correspond to $\gl_\delta \otimes \kk[t]$-modules, but of course there is no general linear Lie algebra of degree $\delta \notin \Z$ (for negative integers $n$ one can work with the general linear Lie superalgebras $\gl_{a|b}$, with $b > a$). Hence, in much the same way that $\OB$ allows us to study $\gl_\delta$-modules for $\delta \notin \Z$ despite the fact that $\gl_\delta$ itself does not exist, the definitions in this paper make possible the study of modules for map algebras and equivariant map algebras with non-integer dimension parameters.

We will recall some basic notions related to string diagrams in Section \ref{s:preliminaries}, define our core objects of study and prove some of their properties in Section \ref{s:lieAlgebraObjects}, and then examine some applications to diagrammatic categories in Section \ref{s:diagrammaticCategories}, including a candidate interpolating category $\Curr(\OB)$ for the categories of $\gl_n \otimes \kk[t]$-modules.

We conclude the introduction by listing a few potential directions for future research related to the content of this paper.
\begin{itemize}
	\item Conjecture \ref{conj:currentInterpolation} remains to be investigated further -- see the discussion after Lemma \ref{lem:incarnationKernel} for details on what would need to be shown in order to prove the conjecture. More generally, one could use diagrammatic categories like $\OB$ and $\B$ to produce candidate interpolating categories for the categories of modules for other types of (equivariant) map algebras, e.g.\ multicurrent algebras or twisted loop algebras.
	\item The definition of $(L, A)$-modules in this paper combines a Lie algebra $L$ internal to a category $\CC$ with a $\kk$-algebra $A$ that (in general) is external to $\CC$. This raises a natural question: could this sort of internal-external definition be used to produce categorical generalizations of other algebraic constructions? For instance, if $A$ is a semigroup in some category $\CC$ and $B$ is an associative $\kk$-algebra, one could define $(A, B)$-modules in $\CC$ (intuitively corresponding to $A \otimes B$-modules) using a collection of action morphisms indexed by the elements of $B$, similar to what is done in Definition \ref{def:aMapModule}. Such definitions could potentially be used to produce interpolating categories for various types of modules.
	\item There is a presentation for current algebras $\g \otimes \kk[t]$ known as \emph{Drinfeld's $J$ presentation}, generated by elements $x$ and $J(x)$, with $x$ ranging over $\g$. The isomorphism with ${\g \otimes \kk[t]}$ is given by $x \mapsto x \otimes 1$, $J(x) \mapsto x \otimes t$. This $J$ presentation is based on Drinfeld's similar presentation for Yangians introduced in \cite{drinfeldYangians}; for information on $J$ presentations in English, see e.g.\ \cite[Thm.~12.1.1]{chariPressley} and \cite[\S3.1]{drinfeldJTwistedYangians}. It would be interesting to establish a similar $J$ presentation for current modules in general symmetric monoidal $\kk$-linear categories. The author has investigated such presentations, but existing arguments for showing that Drinfeld's $J$ presentation yields a Lie algebra isomorphic to $\g \otimes \kk[t]$ do not straightforwardly generalize to this new case.
\end{itemize}

\section*{Acknowledgements}

This research was supported by a CGS-D scholarship from the Natural Sciences and Engineering Research Council of Canada (NSERC).

\section{Preliminaries} \label{s:preliminaries}

\subsection{Conventions}

Throughout this paper we work over a field $\kk$ of characteristic 0. We write $\kkvec$ for the category of $\kk$-vector spaces, $\one$ for the unit object of a monoidal category, and $\id_X$ for the identity morphism of an object $X$. If $\Gamma$ is a group, we write $\hat{\Gamma}$ for its character group. Except for Lie algebras, all $\kk$-algebras in this paper are assumed to be associative, but not necessarily unital.

\subsection{Diagrammatics} \label{ss:diagrammatics}
The notions we will discuss in this paper are most naturally expressed via string diagrams. We will recall a few key definitions and properties here. A more detailed introduction to string diagrams with numerous examples can be found in \cite{stringDiagrams}. For a comprehensive treatment of the topic, see e.g.\ \cite[\S2]{monoidalCategoriesAndTopologicalFieldTheory}.

We represent morphisms in monoidal categories via string diagrams, with composition corresponding to vertical stacking and tensor products corresponding to horizontal juxtaposition. Identity morphisms are drawn as plain strands. To reduce visual clutter, we will typically use colour to indicate domains and codomains in string diagrams, as opposed to explicit object labels. To illustrate, suppose that $X$, $\green{Y}$, $\red{Z}$, and $\blue{A}$ are objects, and that $f \colon X \to \green{Y}$, $g \colon \green{Y} \to \red{Z}$, and $h \colon \blue{A} \to \red{Z}$ are morphisms. Then $\id_{X}, f, g \circ f$, and $f \otimes h$ would be respectively drawn as follows:
$$\begin{tikzpicture}[anchorbase]
\draw[bmod] (0, 0) -- (0, 1);
\end{tikzpicture}\ ,
\quad \quad
\begin{tikzpicture}[anchorbase]
\draw[bmod] (0, 0) -- (0, 0.5);
\draw[gmod] (0, 0.5) -- (0, 1);
\coupon{0, 0.5}{f};
\end{tikzpicture},
\quad \quad
\begin{tikzpicture}[anchorbase]
\draw[bmod] (0, 0) -- (0, 0.5);
\draw[gmod] (0, 0.5) -- (0, 1);
\draw[rmod] (0, 1) -- (0, 1.5);
\coupon{0, 0.5}{f};
\coupon{0, 1}{g};
\end{tikzpicture},
\quad \quad
\begin{tikzpicture}[anchorbase]
\draw[bmod] (0, 0) -- (0, 0.5);
\draw[gmod] (0, 0.5) -- (0, 1);
\draw[blmod] (0.75, 0) -- (0.75, 0.5);
\draw[rmod] (0.75, 0.5) -- (0.75, 1);
\coupon{0, 0.5}{f};
\coupon{0.75, 0.5}{h};
\end{tikzpicture}
\ \overset{*}{=}\
\begin{tikzpicture}[anchorbase]
\draw[bmod] (0, 0) -- (0, 0.5);
\draw[gmod] (0, 0.5) -- (0, 1.5);
\draw[blmod] (0.75, 0) -- (0.75, 1);
\draw[rmod] (0.75, 1) -- (0.75, 1.5);
\coupon{0, 0.5}{f};
\coupon{0.75, 1}{h};
\end{tikzpicture}
\ \overset{**}{=}\
\begin{tikzpicture}[anchorbase]
\draw[bmod] (0, 0) -- (0, 1);
\draw[gmod] (0, 1) -- (0, 1.5);
\draw[blmod] (0.75, 0) -- (0.75, 0.5);
\draw[rmod] (0.75, 0.5) -- (0.75, 1.5);
\coupon{0, 1}{f};
\coupon{0.75, 0.5}{h};
\end{tikzpicture}\ .$$
The equalities $*$ and $**$ comprise the \emph{interchange law} for monoidal categories, which tells us that we may freely slide morphisms up and down along strands. We typically avoid traditional morphism names $f, g, h$, etc., and instead introduce morphisms as decorated strands/diagrams. For instance, $\begin{tikzpicture}[anchorbase]
\draw[bmod] (0, 0) -- (0, 0.25);
\draw[gmod] (0, 0.25) -- (0, 0.5);
\adot{(0, 0.25)};
\end{tikzpicture}$ and $\dmodulemult$ could respectively represent a morphism $X \to \green{Y}$ and a morphism $X \otimes \green{Y} \to \green{Y}$.

If $X$ and $\green{Y}$ are objects in a symmetric monoidal category, we draw the symmetric braiding morphism $\sigma_{X, \green{Y}} \colon X \otimes \green{Y} \to \green{Y} \otimes X$ as $\begin{tikzpicture}[anchorbase]
\draw[bmod] (0, 0) -- (0.5, 0.5);
\draw[gmod] (0.5, 0) -- (0, 0.5);
\end{tikzpicture}$. These crossings satisfy the following key properties, for all morphisms $\begin{tikzpicture}[anchorbase]
\draw[bmod] (0, 0) -- (0, 0.25);
\draw[gmod] (0, 0.25) -- (0, 0.5);
\adot{(0, 0.25)};
\end{tikzpicture}$ and $\begin{tikzpicture}[anchorbase]
\draw[rmod] (0, 0) -- (0, 0.25);
\draw[blmod] (0, 0.25) -- (0, 0.5);
\bdot{(0, 0.25)};
\end{tikzpicture}$:
\begin{equation}\label{eq:crossingIdentities} \tag{CROSS}
\begin{tikzpicture}[anchorbase]
\draw[bmod] (0, 0) -- (0, 0.25) to[out=90, in=270] (0.25, 0.5) to[out=90, in=270] (0, 0.75) -- (0, 1);
\draw[gmod] (0.25, 0) -- (0.25, 0.25) to[out=90, in=270] (0, 0.5) to[out=90, in=270] (0.25, 0.75) -- (0.25, 1);
\end{tikzpicture}
\ =\
\begin{tikzpicture}[anchorbase]
\draw[bmod] (0, 0) -- (0, 1);
\draw[gmod] (0.25, 0) -- (0.25, 1);
\end{tikzpicture}\ ,
\quad \quad
\begin{tikzpicture}[anchorbase]
\draw[bmod] (0, 0) -- (0, 0.25) to[out=90, in=270] (0.25, 0.5) to[out=90, in=270] (0.5, 0.75) -- (0.5, 1.25);
\draw[gmod] (0.25, 0) -- (0.25, 0.25) to[out=90, in=270] (0, 0.5) -- (0, 0.75) to[out=90, in=270] (0.25, 1) -- (0.25, 1.25);
\draw[rmod] (0.5, 0) -- (0.5, 0.5) to[out=90, in=270] (0.25, 0.75) to[out=90, in=270] (0, 1) -- (0, 1.25);
\end{tikzpicture}
\ =\
\begin{tikzpicture}[anchorbase, xscale=-1]
\draw[rmod] (0, 0) -- (0, 0.25) to[out=90, in=270] (0.25, 0.5) to[out=90, in=270] (0.5, 0.75) -- (0.5, 1.25);
\draw[gmod] (0.25, 0) -- (0.25, 0.25) to[out=90, in=270] (0, 0.5) -- (0, 0.75) to[out=90, in=270] (0.25, 1) -- (0.25, 1.25);
\draw[bmod] (0.5, 0) -- (0.5, 0.5) to[out=90, in=270] (0.25, 0.75) to[out=90, in=270] (0, 1) -- (0, 1.25);
\end{tikzpicture}\ ,
\quad \quad
\begin{tikzpicture}[anchorbase]
\draw[bmod] (0, 0) -- (0, 0.25);
\draw[gmod] (0, 0.25) -- (0, 0.5) to[out=90, in=270] (0.25, 0.75) -- (0.25, 1);
\draw[rmod] (0.25, 0) -- (0.25, 0.25);
\draw[blmod] (0.25, 0.25) -- (0.25, 0.5) to[out=90, in=270] (0, 0.75) -- (0, 1);
\adot{(0, 0.25)};
\bdot{(0.25, 0.25)};
\end{tikzpicture}
\ =\
\begin{tikzpicture}[anchorbase, yscale=-1]
\draw[blmod] (0, 0) -- (0, 0.25);
\draw[rmod] (0, 0.25) -- (0, 0.5) to[out=90, in=270] (0.25, 0.75) -- (0.25, 1);
\draw[gmod] (0.25, 0) -- (0.25, 0.25);
\draw[bmod] (0.25, 0.25) -- (0.25, 0.5) to[out=90, in=270] (0, 0.75) -- (0, 1);
\bdot{(0, 0.25)};
\adot{(0.25, 0.25)};
\end{tikzpicture}\ .
\end{equation}
Given an object $X$ with (right) dual $X^*$, we will typically draw $X$ as an upwards-oriented strand, $\goup$, and $X^*$ as a downwards-oriented strand, $\godown$. The evaluation and coevaluation morphisms for the dual pair $(X, X^*)$ would then be drawn as $\dcapupdown \colon \goup \otimes \godown \to \one$ and $\dcupdownup \colon \one \to \godown \otimes \goup$ (note that $\one$ is represented in string diagrams by the lack of a strand). Such caps and cups satisfy the \emph{zigzag identities}:
\begin{equation}\label{eq:zigzag} \tag{ZIGZAG}
\begin{tikzpicture}[anchorbase]
\draw[->, black, thick] (0, 0) -- (0, 0.75) arc(180:0:0.25) -- (0.5, 0.5) arc(180:360:0.25) -- (1, 1.25);
\end{tikzpicture}
\ =\
\begin{tikzpicture}[anchorbase]
\draw[->, black, thick] (0, 0) -- (0, 1.25);
\end{tikzpicture}
\ , \quad\quad
\begin{tikzpicture}[anchorbase]
\draw[<-, black, thick] (1, 0) -- (1, 0.75) arc(0:180:0.25) -- (0.5, 0.5) arc(360:180:0.25) -- (0, 1.25);
\end{tikzpicture}
\ =\
\begin{tikzpicture}[anchorbase]
\draw[<-, black, thick] (0, 0) -- (0, 1.25);
\end{tikzpicture}\ .
\end{equation}
In a symmetric monoidal category, right duals can be made into left duals (and vice versa) in a canonical way, with the cup-cap pairs being related by the following \emph{delooping identities}:
\begin{equation}\label{eq:delooping}\tag{DELOOP}
\begin{tikzpicture}[anchorbase]
\draw[<-, thick] (0, 0) -- (0, 0.25) arc(180:0:0.25) -- (0.5, 0);
\end{tikzpicture} 
\ =\ 
\begin{tikzpicture}[anchorbase]
\draw[<-, thick] (0, 0) -- (0, 0.25) to[out=90, in=270] (0.5, 0.75) arc(0:180:0.25) to[out=270, in=90] (0.5, 0.25) -- (0.5, 0);
\end{tikzpicture}\ , 
\quad \quad 
\begin{tikzpicture}[anchorbase, xscale=-1]
\draw[<-, thick] (0, 0) -- (0, 0.25) arc(180:0:0.25) -- (0.5, 0);
\end{tikzpicture} 
\ =\ 
\begin{tikzpicture}[anchorbase, xscale=-1]
\draw[<-, thick] (0, 0) -- (0, 0.25) to[out=90, in=270] (0.5, 0.75) arc(0:180:0.25) to[out=270, in=90] (0.5, 0.25) -- (0.5, 0);
\end{tikzpicture}\ , 
\quad \quad
\begin{tikzpicture}[anchorbase, yscale=-1]
\draw[<-, thick] (0, 0) -- (0, 0.25) arc(180:0:0.25) -- (0.5, 0);
\end{tikzpicture} 
\ =\ 
\begin{tikzpicture}[anchorbase, yscale=-1]
\draw[<-, thick] (0, 0) -- (0, 0.25) to[out=90, in=270] (0.5, 0.75) arc(0:180:0.25) to[out=270, in=90] (0.5, 0.25) -- (0.5, 0);
\end{tikzpicture}\ , 
\quad \quad
\begin{tikzpicture}[anchorbase, xscale=-1, yscale=-1]
\draw[<-, thick] (0, 0) -- (0, 0.25) arc(180:0:0.25) -- (0.5, 0);
\end{tikzpicture} 
\ =\ 
\begin{tikzpicture}[anchorbase, xscale=-1, yscale=-1]
\draw[<-, thick] (0, 0) -- (0, 0.25) to[out=90, in=270] (0.5, 0.75) arc(0:180:0.25) to[out=270, in=90] (0.5, 0.25) -- (0.5, 0);
\end{tikzpicture}\ .
\end{equation}
Throughout the paper we assume that all cups and caps for dual pairs satisfy \eqref{eq:delooping}.

When working in a $\kk$-linear monoidal category that is also additive, we adopt the following diagrammatic shorthands. Let $I$ be some finite indexing set, and suppose $X = \bigoplus\limits_{i \in I} \green{X_i}$ for some objects $\green{X_i}$. For each $i \in I$, we draw the projection $X \to \green{X_i}$ as $\begin{tikzpicture}[anchorbase]
\draw[bmod] (0, 0) -- (0, 0.25);
\draw[gmod] (0, 0.25) -- (0, 0.5);
\triup{(0, 0.25)}{west}{i};
\end{tikzpicture}$ and the inclusion $\green{X_i} \to X$ as $\begin{tikzpicture}[anchorbase]
\draw[gmod] (0, 0) -- (0, 0.25);
\draw[bmod] (0, 0.25) -- (0, 0.5);
\tridown{(0, 0.25)}{west}{i};
\end{tikzpicture}$. The axioms of a biproduct can then be pictured as follows, for $i, j \in I$:
\begin{equation}\label{eq:biproduct1}
\begin{tikzpicture}[anchorbase]
\draw[gmod] (0, 0) -- (0, 0.25);
\draw[bmod] (0, 0.25) -- (0, 0.75);
\draw[gmod] (0, 0.75) -- (0, 1);
\tridown{(0, 0.25)}{west}{i};
\triup{(0, 0.75)}{west}{j};
\end{tikzpicture} = \delta_{ij} = \begin{cases} \id_{\green{X_i}} &\text{if } i = j, \\ 
0 &\text{if } i \neq j, \end{cases}
\end{equation}
\begin{equation} \label{eq:biproduct2}
\summ_{i \in I}
\begin{tikzpicture}[anchorbase]
\draw[bmod] (0, 0) -- (0, 0.25);
\draw[gmod] (0, 0.25) -- (0, 0.75);
\draw[bmod] (0, 0.75) -- (0, 1);
\triup{(0, 0.25)}{west}{i};
\tridown{(0, 0.75)}{west}{i};
\end{tikzpicture} = \id_{X}.
\end{equation}

\section[Lie algebras, (L, A)-modules, and equivariant (Gamma, L, A)-modules]{Lie algebras, $(L, A)$-modules, and equivariant $(\Gamma, L, A)$-modules} \label{s:lieAlgebraObjects}

Throughout this section we work in a symmetric monoidal $\kk$-linear category $\CC$.

\subsection{Core definitions and examples}

In this subsection, we introduce the main objects of study in this paper: ungraded and graded Lie algebras, and various types of modules for them. We will present the definitions in order of increasing generality.

\begin{defin} \label{def:lieObject}
	A \emph{Lie algebra} (in $\CC$) is an object $L$ equipped with a morphism $\dmult \colon L \otimes L \to L$, called the \emph{Lie bracket} for $L$, satisfying the following identities:
	\begin{equation} \label{eq:skew} \tag{SKEW}
	\begin{tikzpicture}[anchorbase]
	\draw[bmod] (0, 0) -- (0.5, 0.5) -- (0.25, 0.75) -- (0.25, 1);
	\draw[bmod] (0.5, 0) -- (0, 0.5) -- (0.25, 0.75);
	\end{tikzpicture}
	\quad = \quad
	- \dmult,
	\end{equation}
	\begin{equation} \label{eq:jacobi} \tag{JACOBI}
	\begin{tikzpicture}[anchorbase]
	\draw[bmod] (0.25, 0) -- (0.5, 0.25);
	\draw[bmod] (0, 0) -- (0, 0.25) -- (0.25, 0.5);
	\draw[bmod] (0.75, 0) -- (0.5, 0.25) -- (0.25, 0.5) -- (0.25, 0.75);
	\end{tikzpicture}
	\ =\
	\begin{tikzpicture}[anchorbase]
	\draw[bmod] (0.5, 0) -- (0.25, 0.25);
	\draw[bmod] (0, 0) -- (0.25, 0.25) -- (0.5, 0.5);
	\draw[bmod] (0.75, 0) -- (0.75, 0.25) -- (0.5, 0.5) -- (0.5, 0.75);
	\end{tikzpicture}
	\ +\
	\begin{tikzpicture}[anchorbase]
	\draw[bmod] (0, -0.25) -- (0.25, 0) -- (0.5, 0.25);
	\draw[bmod] (0.25, -0.25) -- (0, 0) -- (0, 0.25) -- (0.25, 0.5);
	\draw[bmod] (0.75, -0.25) -- (0.75, 0) -- (0.5, 0.25) -- (0.25, 0.5) -- (0.25, 0.75);
	\end{tikzpicture}\ .
	\end{equation}
	If $L$ and $\red{M}$ are two Lie algebras, a morphism from $L$ to $\red{M}$ is a $\CC$-morphism $\begin{tikzpicture}[anchorbase]
	\draw[bmod] (0, 0) -- (0, 0.25);
	\draw[rmod] (0, 0.25) -- (0, 0.5);
	\bdot{(0, 0.25)};
	\end{tikzpicture}$ that satisfies:
	\begin{equation} \label{eq:lieObjectMorphism}
	\begin{tikzpicture}[anchorbase]
	\draw[bmod] (0, 0) -- (0, 0.25);
	\draw[bmod] (0.5, 0) -- (0.5, 0.25);
	\draw[rmod] (0, 0.25) -- (0.25, 0.5) -- (0.25, 0.75);
	\draw[rmod] (0.5, 0.25) -- (0.25, 0.5);
	\bdot{(0, 0.25)};
	\bdot{(0.5, 0.25)};
	\end{tikzpicture}
	\ =\
	\begin{tikzpicture}[anchorbase]
	\draw[bmod] (0, 0) -- (0.25, 0.25) -- (0.25, 0.5);
	\draw[bmod] (0.5, 0) -- (0.25, 0.25);
	\draw[rmod] (0.25, 0.5) -- (0.25, 0.75);
	\bdot{(0.25, 0.5)};
	\end{tikzpicture}\ .
	\end{equation}
	Note that in the above we simply said that $L$ and $\red{M}$ were Lie algebras, leaving the morphisms $\dmult$ and $\dmultred$ implicit. In most cases throughout the paper, we will reference Lie algebras and their modules by just their underlying $\CC$-object.
\end{defin}

\begin{rem} \label{rem:lieAlgebraObject}
	Some authors reserve the term ``Lie algebra'' for the case $\CC = \kkvec$ (see Example \ref{ex:standardExamples}), and would in general refer to the objects of Definition \ref{def:lieObject} as ``Lie algebra objects''. For the sake of brevity, we use the term ``Lie algebra'' no matter the underlying category.
\end{rem}

\begin{defin} \label{def:lieObjectModule}
	Let $L$ be a Lie algebra. An \emph{$L$-module} is an object $\green{V}$ equipped with a morphism $\dmodulemult \colon L \otimes \green{V} \to \green{V}$ satisfying the following identity:
	\begin{equation} \label{eq:lmod} \tag{LMOD}
	\begin{tikzpicture}[anchorbase]
	\draw[bmod] (0.25, 0) -- (0.5, 0.25);
	\draw[bmod] (0, 0) -- (0, 0.25) -- (0.25, 0.5);
	\draw[gmod] (0.75, 0) -- (0.5, 0.25) -- (0.25, 0.5) -- (0.25, 0.75);
	\end{tikzpicture}
	\ =\
	\begin{tikzpicture}[anchorbase]
	\draw[bmod] (0.5, 0) -- (0.25, 0.25);
	\draw[bmod] (0, 0) -- (0.25, 0.25) -- (0.5, 0.5);
	\draw[gmod] (0.75, 0) -- (0.75, 0.25) -- (0.5, 0.5) -- (0.5, 0.75);
	\end{tikzpicture}
	\ +\
	\begin{tikzpicture}[anchorbase]
	\draw[bmod] (0, -0.25) -- (0.25, 0) -- (0.5, 0.25);
	\draw[bmod] (0.25, -0.25) -- (0, 0) -- (0, 0.25) -- (0.25, 0.5);
	\draw[gmod] (0.75, -0.25) -- (0.75, 0) -- (0.5, 0.25) -- (0.25, 0.5) -- (0.25, 0.75);
	\end{tikzpicture}\ .
	\end{equation}
	A morphism between $L$-modules $\green{V} \to \red{W}$ is a $\CC$-morphism $\begin{tikzpicture}[anchorbase]
	\draw[gmod] (0, 0) -- (0, 0.25);
	\draw[rmod] (0, 0.25) -- (0, 0.5);
	\adot{(0, 0.25)};
	\end{tikzpicture}$ satisfying
	\begin{equation} \label{eq:lmodMorphism}
	\begin{tikzpicture}[anchorbase]
	\draw[bmod] (0, 0) -- (0, 0.25) -- (0.25, 0.5);
	\draw[gmod] (0.5, 0) -- (0.5, 0.25);
	\draw[rmod] (0.5, 0.25) -- (0.25, 0.5) -- (0.25, 0.75);
	\adot{(0.5, 0.25)};
	\end{tikzpicture}
	\ =\
	\begin{tikzpicture}[anchorbase]
	\draw[bmod] (0, 0) -- (0.25, 0.25);
	\draw[gmod] (0.5, 0) -- (0.25, 0.25) -- (0.25, 0.5);
	\draw[rmod] (0.25, 0.5) -- (0.25, 0.75);
	\adot{(0.25, 0.5)};
	\end{tikzpicture}\ .
	\end{equation}
	For a given Lie algebra $L$, the collection of $L$-modules and morphisms between them in $\CC$ forms a category that we denote $L\dashmod$.
\end{defin}

\begin{eg} \label{ex:adjointModule}
	Let $\left(L, \dmult\right)$ be a Lie algebra. Then $\left(L, \dmult\right)$ is an $L$-module known as the \emph{adjoint module} for $L$. Note that \eqref{eq:lmod} for the adjoint module is the same identity as \eqref{eq:jacobi}.
\end{eg}

\begin{egs} \label{ex:standardExamples}
	A Lie algebra in $\kkvec$ is simply an ordinary Lie algebra over $\kk$. A Lie algebra in the category of super $\kk$-vector spaces (with morphisms being parity-preserving linear maps) is a Lie superalgebra. For any commutative monoid $M$, a Lie algebra in the category of $M$-graded vector spaces (with morphisms being grade-preserving linear maps) is an $M$-graded Lie algebra over $\kk$. In these three examples, an $L$-module is a representation of a Lie algebra, Lie superalgebra, or graded Lie algebra, respectively, and morphisms of $L$-modules are $L$-equivariant linear maps (that are parity/grade preserving, in the latter two cases).
	
	Given a $\kk$-linear category $\DD$ (not necessarily monoidal), the $\kk$-linear endofunctor category $\End(\DD)$ is strict monoidal but typically not symmetric. In some cases, one can equip a subcategory of $\End(\DD)$ with the structure of a symmetric monoidal category. A \emph{Lie monad over $\DD$} is a Lie algebra in some subcategory of $\End(\DD)$ that is equipped with a symmetric monoidal structure. This definition parallels the fact that an ordinary monad over $\DD$ is a monoid in $\End(\DD)$. Lie monads were previously studied with a slightly more general definition in \cite{LieMonads}.
	
	A Lie algebra in the category of abelian groups is known as a Lie ring.
\end{egs}

From this point until Example \ref{ex:evaluationModule2}, let $L$ be a Lie algebra and $A$ a commutative $\kk$-algebra.

\begin{defin} \label{def:aMapModule}
	An \emph{$(L, A)$-module} is an object $\green{V}$ equipped with a collection of \emph{action morphisms}, $\left\{\damapmult{a} : a \in A\right\},$ satisfying the following two relations for all $a, b \in A$ and $\lambda, \mu \in \kk$:
	\begin{equation} \label{eq:amaplinearity}
	\damapmult{\lambda a + \mu b} \ =\ \lambda \damapmult{a} + \mu \damapmult{b},
	\end{equation}
	\begin{equation} \label{eq:amapcompatibility}
	\begin{tikzpicture}[anchorbase]
	\draw[bmod] (0.25, 0) -- (0.5, 0.25);
	\draw[bmod] (0, 0) -- (0, 0.25) -- (0.25, 0.5);
	\draw[gmod] (0.75, 0) -- (0.5, 0.25) -- (0.25, 0.5) -- (0.25, 0.75);
	\btoken{0.5, 0.25}{west}{b};
	\btoken{0.25, 0.5}{west}{a};
	\end{tikzpicture}
	\ =\
	\begin{tikzpicture}[anchorbase]
	\draw[bmod] (0.5, 0) -- (0.25, 0.25);
	\draw[bmod] (0, 0) -- (0.25, 0.25) -- (0.5, 0.5);
	\draw[gmod] (0.75, 0) -- (0.75, 0.25) -- (0.5, 0.5) -- (0.5, 0.75);
	\btoken[2px]{0.5, 0.5}{west}{ab};
	\end{tikzpicture}
	\ +\
	\begin{tikzpicture}[anchorbase]
	\draw[bmod] (0, -0.25) -- (0.25, 0) -- (0.5, 0.25);
	\draw[bmod] (0.25, -0.25) -- (0, 0) -- (0, 0.25) -- (0.25, 0.5);
	\draw[gmod] (0.75, -0.25) -- (0.75, 0) -- (0.5, 0.25) -- (0.25, 0.5) -- (0.25, 0.75);
	\btoken{0.5, 0.25}{west}{a};
	\btoken[2px]{0.25, 0.5}{west}{b};
	\end{tikzpicture}\ .
	\end{equation}
	A morphism of $(L, A)$-modules $\green{V} \to \red{W}$ is a $\CC$-morphism $\begin{tikzpicture}[anchorbase]
	\draw[gmod] (0, 0) -- (0, 0.25);
	\draw[rmod] (0, 0.25) -- (0, 0.5);
	\adot{(0, 0.25)};
	\end{tikzpicture}$ satisfying
	\begin{equation} \label{eq:amapMorphism}
	\begin{tikzpicture}[anchorbase]
	\draw[bmod] (0, -0.25) -- (0, 0) -- (0.25, 0.25);
	\draw[gmod] (0.5, -0.25) -- (0.5, 0);
	\draw[rmod] (0.5, 0) -- (0.25, 0.25) -- (0.25, 0.5);
	\btoken{0.25, 0.25}{west}{a};
	\adot{(0.5, 0)};
	\end{tikzpicture}
	\ =\
	\begin{tikzpicture}[anchorbase]
	\draw[bmod] (0, 0) -- (0.25, 0.25);
	\draw[gmod] (0.5, 0) -- (0.25, 0.25) -- (0.25, 0.5);
	\draw[rmod] (0.25, 0.5) -- (0.25, 0.75);
	\btoken{0.25, 0.25}{west}{a};
	\adot{(0.25, 0.5)};
	\end{tikzpicture}
	\end{equation}
	for each $a \in A$. We write $(L, A)\dashmod$ for the category of $(L, A)$-modules. The linearity identity \eqref{eq:amaplinearity} is usually obvious, so we will not explicitly mention it when checking that constructions introduced later in this paper satisfy the definition of an $(L, A)$-module. Moreover, in light of \eqref{eq:amaplinearity}, it suffices to specify action morphisms for an $(L, A)$-module with $a$ ranging over a basis of $A$, and similarly it suffices to impose the identities \eqref{eq:amapcompatibility} and \eqref{eq:amapMorphism} with $a$ and $b$ ranging over a basis of $A$. The same comments apply to the equivariant $(\Gamma, L, A)$-modules we will introduce in Definition \ref{def:equivariantModule}.
\end{defin}

Working in $\CC = \kkvec$, $(L, A)$-modules are equivalent to $L \otimes A$-modules, where $L \otimes A$ is equipped with the Lie bracket extended from $[x \otimes a, y \otimes b] = [x, y] \otimes ab$. Given an $(L, A)$-module $\green{V}$, the corresponding $L \otimes A$-action on $\green{V}$ is given by $(x \otimes a) \cdot \green{v} = \damapmult{a}(x \otimes \green{v})$ for $x \in L$, $a \in A$, and $\green{v} \in \green{V}$. Conversely, if $\green{V}$ is an $L \otimes A$-module, then the corresponding $(L, A)$-module has action morphisms given by $\damapmult{a}(x \otimes \green{v}) = (x \otimes a) \cdot \green{v}$. Definition \ref{def:aMapModule} allows us to study the analogues of these modules in categories where we cannot directly construct an object to play the role of $L \otimes A$. For instance, in Section \ref{s:diagrammaticCategories} we will discuss $(L, \kk[t])$-modules in the oriented Brauer category.

\begin{lem} \label{lem:amapGivesLmodule}
	Suppose $e \in A$ is idempotent and that $\green{V}$ is an $(L, A)$-module. Then $\left(\green{V}, \damapmult{e}\right)$ is an $L$-module. As a special case, if $A$ is unital, then $\left(\green{V}, \damapmult{1}\right)$ is an $L$-module that we call the \emph{underlying $L$-module} of $\green{V}$.
\end{lem}

\begin{proof}
	Using \eqref{eq:amapcompatibility} and the assumption $e^2 = e$, we have $\begin{tikzpicture}[anchorbase]
	\draw[bmod] (0.25, 0) -- (0.5, 0.25);
	\draw[bmod] (0, 0) -- (0, 0.25) -- (0.25, 0.5);
	\draw[gmod] (0.75, 0) -- (0.5, 0.25) -- (0.25, 0.5) -- (0.25, 0.75);
	\btoken{0.5, 0.25}{west}{e};
	\btoken{0.25, 0.5}{west}{e};
	\end{tikzpicture}
	\ =\
	\begin{tikzpicture}[anchorbase]
	\draw[bmod] (0.5, 0) -- (0.25, 0.25);
	\draw[bmod] (0, 0) -- (0.25, 0.25) -- (0.5, 0.5);
	\draw[gmod] (0.75, 0) -- (0.75, 0.25) -- (0.5, 0.5) -- (0.5, 0.75);
	\btoken[2px]{0.5, 0.5}{west}{e};
	\end{tikzpicture}
	\ +\
	\begin{tikzpicture}[anchorbase]
	\draw[bmod] (0, -0.25) -- (0.25, 0) -- (0.5, 0.25);
	\draw[bmod] (0.25, -0.25) -- (0, 0) -- (0, 0.25) -- (0.25, 0.5);
	\draw[gmod] (0.75, -0.25) -- (0.75, 0) -- (0.5, 0.25) -- (0.25, 0.5) -- (0.25, 0.75);
	\btoken{0.5, 0.25}{west}{e};
	\btoken[2px]{0.25, 0.5}{west}{e};
	\end{tikzpicture}\ ,$ which shows that \eqref{eq:lmod} is satisfied.
\end{proof}

\begin{eg} \label{ex:kMapLModule}
	Consider $\kk$ as a one-dimensional $\kk$-algebra. Then $(L, \kk)$-modules are equivalent to $L$-modules. Given an $(L, \kk)$-module $\green{V}$, the associated $L$-module is $\left(\green{V}, \damapmult{1}\right)$, as in Lemma \ref{lem:amapGivesLmodule}. Conversely, if $\green{V}$ is an $L$-module, then setting $\damapmult{a} = a\ \dmodulemult$ for each $a \in \kk$ yields the associated $(L, \kk)$-module.
\end{eg}

\begin{eg} \label{ex:trivialModule}
	Let $\green{V} \in \ob(\CC)$, and for each $a \in A$, define $\damapmult{a}$ to be the zero map. With these action morphisms, $\green{V}$ is an $(L, A)$-module. This is known as the \emph{trivial $(L, A)$-module structure} on $\green{V}$. If $\green{V}, \red{W} \in \ob(\CC)$ and we equip both with the trivial $(L, A)$-module structure, then a morphism of $(L, A)$-modules $\green{V} \to \red{W}$ is simply a $\CC$-morphism $\green{V} \to \red{W}$. Hence the full subcategory of $(L, A)\dashmod$ consisting of trivial $(L, A)$-modules is isomorphic to $\CC$.
\end{eg}

\begin{eg}[Evaluation modules] \label{ex:evaluationModule}
	Let $\green{V}$ be an $L$-module, and $\mathfrak{m} \subseteq A$ a maximal ideal of $A$ such that $A/\mathfrak{m} \isom \kk$ (as $\kk$-algebras). Let $\ev_\mathfrak{m} \colon A \to \kk$ denote the composition of the canonical projection $A \to A/\mathfrak{m}$ and the isomorphism $A/\mathfrak{m} \isom \kk$. We call $\ev_\mathfrak{m}$ \emph{evaluation at $\mathfrak{m}$}. This name references the case where $A = \kk[t]$ and $\mathfrak{m} = \ngle{t - k}$ for some $k \in \kk$; then $\ev_{\mathfrak{m}}$ is given by evaluation at $k$, i.e.\ $f(t) \mapsto f(k)$.
	
	For each $a \in A$, define $\damapmult{a} = \ev_\mathfrak{m}(a) \dmodulemult$. With these action morphisms, $\green{V}$ is an $(L, A)$-module. We call such modules \emph{evaluation modules}. To see that \eqref{eq:amapcompatibility} holds, let $a, b \in A$. Then:
	\begin{multline*}
	\begin{tikzpicture}[anchorbase]
	\draw[bmod] (0.25, 0) -- (0.5, 0.25);
	\draw[bmod] (0, 0) -- (0, 0.25) -- (0.25, 0.5);
	\draw[gmod] (0.75, 0) -- (0.5, 0.25) -- (0.25, 0.5) -- (0.25, 0.75);
	\btoken{0.5, 0.25}{west}{b};
	\btoken{0.25, 0.5}{west}{a};
	\end{tikzpicture} = \ev_\mathfrak{m}(a)\ev_\mathfrak{m}(b)\ \begin{tikzpicture}[anchorbase]
	\draw[bmod] (0.25, 0) -- (0.5, 0.25);
	\draw[bmod] (0, 0) -- (0, 0.25) -- (0.25, 0.5);
	\draw[gmod] (0.75, 0) -- (0.5, 0.25) -- (0.25, 0.5) -- (0.25, 0.75);
	\end{tikzpicture} \overset{\eqref{eq:lmod}}{=}\ev_\mathfrak{m}(a)\ev_\mathfrak{m}(b)\ \begin{tikzpicture}[anchorbase]
	\draw[bmod] (0.5, 0) -- (0.25, 0.25);
	\draw[bmod] (0, 0) -- (0.25, 0.25) -- (0.5, 0.5);
	\draw[gmod] (0.75, 0) -- (0.75, 0.25) -- (0.5, 0.5) -- (0.5, 0.75);
	\end{tikzpicture}
	\ +\
	\ev_\mathfrak{m}(a)\ev_\mathfrak{m}(b)\ \begin{tikzpicture}[anchorbase]
	\draw[bmod] (0, -0.25) -- (0.25, 0) -- (0.5, 0.25);
	\draw[bmod] (0.25, -0.25) -- (0, 0) -- (0, 0.25) -- (0.25, 0.5);
	\draw[gmod] (0.75, -0.25) -- (0.75, 0) -- (0.5, 0.25) -- (0.25, 0.5) -- (0.25, 0.75);
	\end{tikzpicture} \\
	=
	\ev_\mathfrak{m}(ab)\ \begin{tikzpicture}[anchorbase]
	\draw[bmod] (0.5, 0) -- (0.25, 0.25);
	\draw[bmod] (0, 0) -- (0.25, 0.25) -- (0.5, 0.5);
	\draw[gmod] (0.75, 0) -- (0.75, 0.25) -- (0.5, 0.5) -- (0.5, 0.75);
	\end{tikzpicture}
	\ +\
	\ev_\mathfrak{m}(a)\ev_\mathfrak{m}(b)\ \begin{tikzpicture}[anchorbase]
	\draw[bmod] (0, -0.25) -- (0.25, 0) -- (0.5, 0.25);
	\draw[bmod] (0.25, -0.25) -- (0, 0) -- (0, 0.25) -- (0.25, 0.5);
	\draw[gmod] (0.75, -0.25) -- (0.75, 0) -- (0.5, 0.25) -- (0.25, 0.5) -- (0.25, 0.75);
	\end{tikzpicture} 
	= \begin{tikzpicture}[anchorbase]
	\draw[bmod] (0.5, 0) -- (0.25, 0.25);
	\draw[bmod] (0, 0) -- (0.25, 0.25) -- (0.5, 0.5);
	\draw[gmod] (0.75, 0) -- (0.75, 0.25) -- (0.5, 0.5) -- (0.5, 0.75);
	\btoken[2px]{0.5, 0.5}{west}{ab};
	\end{tikzpicture}
	\ +\
	\begin{tikzpicture}[anchorbase]
	\draw[bmod] (0, -0.25) -- (0.25, 0) -- (0.5, 0.25);
	\draw[bmod] (0.25, -0.25) -- (0, 0) -- (0, 0.25) -- (0.25, 0.5);
	\draw[gmod] (0.75, -0.25) -- (0.75, 0) -- (0.5, 0.25) -- (0.25, 0.5) -- (0.25, 0.75);
	\btoken{0.5, 0.25}{west}{a};
	\btoken[2px]{0.25, 0.5}{west}{b};
	\end{tikzpicture},
	\end{multline*}
	as desired.
\end{eg}

\begin{defin}
	An $(L, \kk[t])$-module is called a \emph{current $L$-module}. An $(L, \kk[t, t^{-1}])$-module is called a \emph{loop $L$-module}.
\end{defin}

\begin{eg}[Induced modules] \label{ex:evaluationModule2}
	Let $\green{V}$ be an $L$-module and $\phi \colon A \to \End_{L\dashmod}(\green{V})$ a homomorphism of $\kk$-algebras. For each $a \in A$, set $\begin{tikzpicture}[anchorbase]
	\draw[gmod] (0, 0) -- (0, 0.5);
	\btoken{0, 0.25}{west}{a};
	\end{tikzpicture} = \begin{tikzpicture}[anchorbase]
	\draw[gmod] (0, 0) -- (0, 1);
	\coupon{0, 0.5}{\phi(a)};
	\end{tikzpicture}$ and define $\damapmult{a} = \begin{tikzpicture}[anchorbase]
	\draw[bmod] (0, 0) -- (0.25, 0.25);
	\draw[gmod] (0.5, 0) -- (0.25, 0.25) -- (0.25, 0.75);
	\btoken{0.25, 0.5}{west}{a};
	\end{tikzpicture}$. Equipped with these action morphisms, $\green{V}$ becomes an $(L, A)$-module that we call the \emph{$(L, A)$-module induced by $\phi$}. To see that \eqref{eq:amapcompatibility} holds, let $a, b \in A$. Since $\phi$ is an algebra homomorphism and $A$ is commutative, $\begin{tikzpicture}[anchorbase]
	\draw[gmod] (0, 0) -- (0, 0.75);
	\btoken{0, 0.25}{west}{b};
	\btoken{0, 0.5}{west}{a};
	\end{tikzpicture} = \begin{tikzpicture}[anchorbase]
	\draw[gmod] (0, 0) -- (0, 0.5);
	\btoken{0, 0.25}{west}{ab};
	\end{tikzpicture} = \begin{tikzpicture}[anchorbase]
	\draw[gmod] (0, 0) -- (0, 0.75);
	\btoken{0, 0.25}{west}{a};
	\btoken{0, 0.5}{west}{b};
	\end{tikzpicture}$. We compute:
	\begin{multline*}
	\begin{tikzpicture}[anchorbase]
	\draw[bmod] (0.25, 0) -- (0.5, 0.25);
	\draw[bmod] (0, 0) -- (0, 0.25) -- (0.25, 0.5);
	\draw[gmod] (0.75, 0) -- (0.5, 0.25) -- (0.25, 0.5) -- (0.25, 0.75);
	\btoken{0.5, 0.25}{west}{b};
	\btoken{0.25, 0.5}{west}{a};
	\end{tikzpicture} 
	=
	\begin{tikzpicture}[anchorbase]
	\draw[bmod] (0.25, 0) -- (0.5, 0.25);
	\draw[bmod] (0, 0) -- (0, 0.25) -- (0.25, 0.5);
	\draw[gmod] (0.75, 0) -- (0.5, 0.25) -- (0.25, 0.5) -- (0.25, 1);
	\btoken{0.375, 0.375}{west}{b};
	\btoken{0.25, 0.75}{west}{a};
	\end{tikzpicture}
	\overset{\eqref{eq:lmodMorphism}}{=}
	\begin{tikzpicture}[anchorbase]
	\draw[bmod] (0.25, 0) -- (0.5, 0.25);
	\draw[bmod] (0, 0) -- (0, 0.25) -- (0.25, 0.5);
	\draw[gmod] (0.75, 0) -- (0.5, 0.25) -- (0.25, 0.5) -- (0.25, 1);
	\btoken{0.25, 0.75}{west}{ab};
	\end{tikzpicture}
	\overset{\eqref{eq:lmod}}{=}
	\begin{tikzpicture}[anchorbase]
	\draw[bmod] (0.5, 0) -- (0.25, 0.25);
	\draw[bmod] (0, 0) -- (0.25, 0.25) -- (0.5, 0.5);
	\draw[gmod] (0.75, 0) -- (0.75, 0.25) -- (0.5, 0.5) -- (0.5, 1);
	\btoken{0.5, 0.75}{west}{ab};
	\end{tikzpicture}
	\ +\
	\begin{tikzpicture}[anchorbase]
	\draw[bmod] (0, -0.25) -- (0.25, 0) -- (0.5, 0.25);
	\draw[bmod] (0.25, -0.25) -- (0, 0) -- (0, 0.25) -- (0.25, 0.5);
	\draw[gmod] (0.75, -0.25) -- (0.75, 0) -- (0.5, 0.25) -- (0.25, 0.5) -- (0.25, 1);
	\btoken{0.25, 0.75}{west}{ab};
	\end{tikzpicture}
	\overset{\eqref{eq:lmodMorphism}}{=}
	\begin{tikzpicture}[anchorbase]
	\draw[bmod] (0.5, 0) -- (0.25, 0.25);
	\draw[bmod] (0, 0) -- (0.25, 0.25) -- (0.5, 0.5);
	\draw[gmod] (0.75, 0) -- (0.75, 0.25) -- (0.5, 0.5) -- (0.5, 1);
	\btoken{0.5, 0.75}{west}{ab};
	\end{tikzpicture}
	\ +\
	\begin{tikzpicture}[anchorbase]
	\draw[bmod] (0, -0.25) -- (0.25, 0) -- (0.5, 0.25);
	\draw[bmod] (0.25, -0.25) -- (0, 0) -- (0, 0.25) -- (0.25, 0.5);
	\draw[gmod] (0.75, -0.25) -- (0.75, 0) -- (0.5, 0.25) -- (0.25, 0.5) -- (0.25, 1);
	\btoken{0.25, 0.75}{west}{b};
	\btoken{0.375, 0.375}{west}{a};
	\end{tikzpicture}
	=
	\begin{tikzpicture}[anchorbase]
	\draw[bmod] (0.5, 0) -- (0.25, 0.25);
	\draw[bmod] (0, 0) -- (0.25, 0.25) -- (0.5, 0.5);
	\draw[gmod] (0.75, 0) -- (0.75, 0.25) -- (0.5, 0.5) -- (0.5, 0.75);
	\btoken[2px]{0.5, 0.5}{west}{ab};
	\end{tikzpicture}
	\ +\
	\begin{tikzpicture}[anchorbase]
	\draw[bmod] (0, -0.25) -- (0.25, 0) -- (0.5, 0.25);
	\draw[bmod] (0.25, -0.25) -- (0, 0) -- (0, 0.25) -- (0.25, 0.5);
	\draw[gmod] (0.75, -0.25) -- (0.75, 0) -- (0.5, 0.25) -- (0.25, 0.5) -- (0.25, 0.75);
	\btoken{0.5, 0.25}{west}{a};
	\btoken{0.25, 0.5}{west}{b};
	\end{tikzpicture}\ ,
	\end{multline*}
	as desired.
	
	In the case $A = \kk[t]$, a homomorphism $\phi \colon \kk[t] \to \End_{L\dashmod}(\green{V})$ is determined by specifying the image of $t$. For $\dmoduleendo \in \End_{L\dashmod}(\green{V})$, we define the \emph{current $L$-module induced by $\dmoduleendo$} to be the the $(L, \kk[t])$-module induced by the homomorphism $\phi$ satisfying $\phi(t) = \dmoduleendo$. If $\dmoduleendo = k\ \begin{tikzpicture}[anchorbase]
	\draw[gmod] (0, 0) -- (0, 0.5);
	\end{tikzpicture}$ for some $k \in \kk$, we call the current module induced by $\dmoduleendo$ the \emph{evaluation (current) module of $\green{V}$ at $k$}. One can similarly define the \emph{loop $L$-module induced by $\dmoduleendo$} if $\dmoduleendo$ is invertible, and the \emph{evaluation (loop) module of $\green{V}$ at $k$} when $k \neq 0$. Current and loop evaluation modules are special cases of the evaluation $(L, A)$-modules described in Example \ref{ex:evaluationModule}. The former is the case $A = \kk[t], \mathfrak{m} = \ngle{t - k}$, and the latter is the case $A = \kk[t, t^{-1}], \mathfrak{m} = \ngle{t - k}$.
	
	Looking concretely at the category $\CC = \kkvec$, the current $L$-module induced by an endomorphism $f \in \End_{L\dashmod}(\green{V})$ has $L \otimes \kk[t]$-action given by $x \otimes t^n \cdot \green{v} = f^n(x \cdot \green{v})$ for all $x \in L, \green{v} \in \green{V}$, and $n \in \N$.
\end{eg}

\begin{defin} \label{def:gradedLieAlgebraObject}
	Let $\Gamma$ be an abelian group. A \emph{$\Gamma$-graded Lie algebra} (in $\CC$) is a collection of objects $\left\{L_g  : g \in \Gamma\right\}$ together with \emph{graded Lie bracket morphisms} $\left\{\begin{tikzpicture}[anchorbase]
	\draw[bmod] (0, 0) -- (0.25, 0.25) -- (0.25, 0.5);
	\draw[bmod] (0.5, 0) -- (0.25, 0.25);
	\node[anchor=north] at (0, 0) {$\scriptstyle g$};
	\node[anchor=north] at (0.5, 0) {$\scriptstyle h$};
	\node[anchor=south] at (0.25, 0.5) {$\scriptstyle g + h$};
	\end{tikzpicture} : g, h \in \Gamma \right\}$ that satisfy the following identities for all $g, h, i \in \Gamma$. Note that here and throughout the paper, we label graded strands with $g$, $h$, etc.\ rather than $L_g, L_h$, etc.\ to avoid visual clutter.
	\begin{equation} \label{eq:gradedSkewSymmetry} \tag{G-SKEW}
	\begin{tikzpicture}[anchorbase]
	\draw[bmod] (0, 0) -- (0.5, 0.5) -- (0.25, 0.75) -- (0.25, 1);
	\draw[bmod] (0.5, 0) -- (0, 0.5) -- (0.25, 0.75);
	\node[anchor=north] at (0, 0) {$\scriptstyle g$};
	\node[anchor=north] at (0.5, 0) {$\scriptstyle h$};
	\node[anchor=south] at (0.25, 1) {$\scriptstyle g + h$};
	\end{tikzpicture}
	\quad = \quad
	- \begin{tikzpicture}[anchorbase]
	\draw[bmod] (0, 0) -- (0.25, 0.25) -- (0.25, 0.5);
	\draw[bmod] (0.5, 0) -- (0.25, 0.25);
	\node[anchor=north] at (0, 0) {$\scriptstyle g$};
	\node[anchor=north] at (0.5, 0) {$\scriptstyle h$};
	\node[anchor=south] at (0.25, 0.5) {$\scriptstyle g + h$};
	\end{tikzpicture},
	\end{equation}
	\begin{equation} \label{eq:gradedJacobi} \tag{G-JACOBI}
	\begin{tikzpicture}[anchorbase]
	\draw[bmod] (0.25, 0) -- (0.5, 0.25);
	\draw[bmod] (0, 0) -- (0, 0.25) -- (0.25, 0.5);
	\draw[bmod] (0.75, 0) -- (0.5, 0.25) -- (0.25, 0.5) -- (0.25, 0.75);
	\node[anchor=north] at (0, 0) {$\scriptstyle g$};
	\node[anchor=north] at (0.25, 0) {$\scriptstyle h$};
	\node[anchor=north] at (0.75, 0) {$\scriptstyle i$};
	\node[anchor=south] at (0.25, 0.75) {$\scriptstyle g + h + i$};
	\end{tikzpicture}
	\ =\
	\begin{tikzpicture}[anchorbase]
	\draw[bmod] (0.5, 0) -- (0.25, 0.25);
	\draw[bmod] (0, 0) -- (0.25, 0.25) -- (0.5, 0.5);
	\draw[bmod] (0.75, 0) -- (0.75, 0.25) -- (0.5, 0.5) -- (0.5, 0.75);
	\node[anchor=north] at (0, 0) {$\scriptstyle g$};
	\node[anchor=north] at (0.5, 0) {$\scriptstyle h$};
	\node[anchor=north] at (0.75, 0) {$\scriptstyle i$};
	\node[anchor=south] at (0.5, 0.75) {$\scriptstyle g + h + i$};
	\end{tikzpicture}
	\ +\
	\begin{tikzpicture}[anchorbase]
	\draw[bmod] (0, -0.25) -- (0.25, 0) -- (0.5, 0.25);
	\draw[bmod] (0.25, -0.25) -- (0, 0) -- (0, 0.25) -- (0.25, 0.5);
	\draw[bmod] (0.75, -0.25) -- (0.75, 0) -- (0.5, 0.25) -- (0.25, 0.5) -- (0.25, 0.75);
	\node[anchor=north] at (0, -0.25) {$\scriptstyle g$};
	\node[anchor=north] at (0.25, -0.25) {$\scriptstyle h$};
	\node[anchor=north] at (0.75, -0.25) {$\scriptstyle i$};
	\node[anchor=south] at (0.25, 0.75) {$\scriptstyle g + h + i$};
	\end{tikzpicture}\ .
	\end{equation}
	If $L$ and $\red{M}$ are two $\Gamma$-graded Lie algebras in $\CC$ (for the same group $\Gamma$), a morphism from $L$ to $\red{M}$ is a collection of $\CC$-morphisms $\left\{\begin{tikzpicture}[anchorbase]
	\draw[bmod] (0, 0) -- (0, 0.25);
	\draw[rmod] (0, 0.25) -- (0, 0.5);
	\bdot{(0, 0.25)};
	\node at (0, -0.25) {$\scriptstyle g$};
	\node at (0, 0.75) {$\scriptstyle g$};
	\end{tikzpicture} : g \in \Gamma \right\}$ satisfying the following identity for all $g, h \in \Gamma$:
	\begin{equation} \label{eq:gradedLieObjectMorphism}
	\begin{tikzpicture}[anchorbase]
	\draw[bmod] (0, 0) -- (0, 0.25);
	\draw[bmod] (0.5, 0) -- (0.5, 0.25);
	\draw[rmod] (0, 0.25) -- (0.25, 0.5) -- (0.25, 0.75);
	\draw[rmod] (0.5, 0.25) -- (0.25, 0.5);
	\bdot{(0, 0.25)};
	\bdot{(0.5, 0.25)};
	\node[anchor=north] at (0, 0) {$\scriptstyle g$};
	\node[anchor=north] at (0.5, 0) {$\scriptstyle h$};
	\node[anchor=south] at (0.25, 0.75) {$\scriptstyle g + h$};
	\end{tikzpicture}
	\ =\
	\begin{tikzpicture}[anchorbase]
	\draw[bmod] (0, 0) -- (0.25, 0.25) -- (0.25, 0.5);
	\draw[bmod] (0.5, 0) -- (0.25, 0.25);
	\draw[rmod] (0.25, 0.5) -- (0.25, 0.75);
	\bdot{(0.25, 0.5)};
	\node[anchor=north] at (0, 0) {$\scriptstyle g$};
	\node[anchor=north] at (0.5, 0) {$\scriptstyle h$};
	\node[anchor=south] at (0.25, 0.75) {$\scriptstyle g + h$};
	\end{tikzpicture}\ .
	\end{equation}
\end{defin}

The following is an alternative definition for graded Lie algebras that can be used in the context of additive categories and finite abelian groups. This additive formulation more directly corresponds to the usual definition of graded Lie algebras in $\kkvec$.

\begin{defin} \label{def:additiveGradedLieAlgebraObject}
	Suppose $\CC$ is additive (in addition to being symmetric monoidal and $\kk$-linear), and let $\Gamma$ be a finite abelian group. An \emph{additive $\Gamma$-graded Lie algebra} is a Lie algebra (as specified in Definition \ref{def:lieObject}) $L$ such that $L = \bigoplus\limits_{g \in \Gamma} \green{L_g}$ for some objects $\green{L_g} \in \ob(\CC)$, and $\begin{tikzpicture}[anchorbase]
	\draw[gmod] (0, 0) -- (0, 0.25);
	\draw[bmod] (0, 0.25) -- (0.25, 0.5);
	\draw[gmod] (0.5, 0) -- (0.5, 0.25);
	\draw[bmod] (0.5, 0.25) -- (0.25, 0.5) -- (0.25, 0.75);
	\draw[gmod] (0.25, 0.75) -- (0.25, 1);
	\triup{(0, 0.25)}{east}{g};
	\triup{(0.5, 0.25)}{west}{h};
	\tridown{(0.25, 0.75)}{west}{i};
	\end{tikzpicture} = 0$ whenever $i \neq g + h$ (recall that these triangles denote inclusion and projection maps; see Section \ref{ss:diagrammatics} for details).
\end{defin}

If $\CC$ is additive and $\Gamma$ is a finite abelian group, given a $\Gamma$-graded Lie algebra $\{\green{L_g} : g \in \Gamma \}$ in the sense of Definition \ref{def:gradedLieAlgebraObject}, the corresponding additive $\Gamma$-graded Lie algebra is $L := \bigoplus\limits_{g \in \Gamma} \green{L_g}$, with Lie bracket defined by $\dmult = \summ_{g, h \in \Gamma}
\begin{tikzpicture}[anchorbase]
\draw[bmod] (0, 0) -- (0, 0.25);
\draw[gmod] (0, 0.25) -- (0.25, 0.5);
\draw[bmod] (0.5, 0) -- (0.5, 0.25);
\draw[gmod] (0.5, 0.25) -- (0.25, 0.5) -- (0.25, 0.75);
\draw[bmod] (0.25, 0.75) -- (0.25, 1);
\tridown{(0, 0.25)}{east}{g};
\tridown{(0.5, 0.25)}{west}{h};
\triup{(0.25, 0.75)}{west}{g + h};
\end{tikzpicture}$. Conversely, if $L = \bigoplus\limits_{g \in \Gamma} \green{L_g}$ is an additive $\Gamma$-graded Lie algebra, the corresponding (non-additive) $\Gamma$-graded Lie algebra is $\{\green{L_g} : g \in \Gamma \}$, with graded Lie brackets given by $\begin{tikzpicture}[anchorbase]
\draw[gmod] (0, 0) -- (0.25, 0.25) -- (0.25, 0.5);
\draw[gmod] (0.5, 0) -- (0.25, 0.25);
\node[anchor=north] at (0, 0) {$\scriptstyle g$};
\node[anchor=north] at (0.5, 0) {$\scriptstyle h$};
\node[anchor=south] at (0.25, 0.5) {$\scriptstyle g + h$};
\end{tikzpicture} = \begin{tikzpicture}[anchorbase]
\draw[gmod] (0, 0) -- (0, 0.25);
\draw[bmod] (0, 0.25) -- (0.25, 0.5);
\draw[gmod] (0.5, 0) -- (0.5, 0.25);
\draw[bmod] (0.5, 0.25) -- (0.25, 0.5) -- (0.25, 0.75);
\draw[gmod] (0.25, 0.75) -- (0.25, 1);
\triup{(0, 0.25)}{east}{g};
\triup{(0.5, 0.25)}{west}{h};
\tridown{(0.25, 0.75)}{west}{g + h};
\end{tikzpicture}$ for $g, h \in \Gamma$.

The following definition introduces modules for graded Lie algebras. Note that these modules are not themselves graded; the idea is that we decompose the action of $L$ on $\green{V}$ into several action morphisms, according to the $\Gamma$-grading on $L$. This decomposition will be used to construct equivariant evaluation modules in Proposition \ref{prop:equivariantEvaluationModule}. Equivariant evaluation modules are a categorical generalization of modules for equivariant map algebras $(\g \otimes A)^\Gamma$, which we will discuss in more detail shortly.

\begin{defin} \label{def:gradedLModule}
	Let $\Gamma$ be an abelian group and $L$ a $\Gamma$-graded Lie algebra. An \emph{$L$-module} is an object $\green{V}$ equipped with a family of morphisms $\left\{\begin{tikzpicture}[anchorbase]
	\draw[bmod] (0, 0) -- (0.25, 0.25);
	\draw[gmod] (0.5, 0) -- (0.25, 0.25) -- (0.25, 0.5);
	\node[anchor=north] at (0, 0) {$\scriptstyle g$};
	%Phantom top label to get correct vertical positioning.
	\node[anchor=south] at (0.25, 0.5) {$\vphantom{\scriptstyle a}$};
	\end{tikzpicture} : g \in \Gamma \right\}$ satisfying the following identity for all $g, h \in \Gamma$:
	\begin{equation} \label{eq:gradedLMod}
	\begin{tikzpicture}[anchorbase]
	\draw[bmod] (0.25, 0) -- (0.5, 0.25);
	\draw[bmod] (0, 0) -- (0, 0.25) -- (0.25, 0.5);
	\draw[gmod] (0.75, 0) -- (0.5, 0.25) -- (0.25, 0.5) -- (0.25, 0.75);
	\node[anchor=north] at (0, 0) {$\scriptstyle g$};
	\node[anchor=north] at (0.25, 0) {$\scriptstyle h$};
	%Phantom top label to get correct vertical positioning.
	\node[anchor=south] at (0.25, 0.75) {$\vphantom{\scriptstyle a}$};
	\end{tikzpicture}
	\ =\
	\begin{tikzpicture}[anchorbase]
	\draw[bmod] (0.5, 0) -- (0.25, 0.25);
	\draw[bmod] (0, 0) -- (0.25, 0.25) -- (0.5, 0.5);
	\draw[gmod] (0.75, 0) -- (0.75, 0.25) -- (0.5, 0.5) -- (0.5, 0.75);
	\node[anchor=north] at (0, 0) {$\scriptstyle g$};
	\node[anchor=north] at (0.5, 0) {$\scriptstyle h$};
	%Phantom top label to get correct vertical positioning.
	\node[anchor=south] at (0.5, 0.75) {$\vphantom{\scriptstyle a}$};
	\end{tikzpicture}
	\ +\
	\begin{tikzpicture}[anchorbase]
	\draw[bmod] (0, -0.25) -- (0.25, 0) -- (0.5, 0.25);
	\draw[bmod] (0.25, -0.25) -- (0, 0) -- (0, 0.25) -- (0.25, 0.5);
	\draw[gmod] (0.75, -0.25) -- (0.75, 0) -- (0.5, 0.25) -- (0.25, 0.5) -- (0.25, 0.75);
	\node[anchor=north] at (0, -0.25) {$\scriptstyle g$};
	\node[anchor=north] at (0.25, -0.25) {$\scriptstyle h$};
	%Phantom top label to get correct vertical positioning.
	\node[anchor=south] at (0.25, 0.75) {$\vphantom{\scriptstyle a}$};
	\end{tikzpicture}\ .
	\end{equation}
\end{defin}

\begin{eg} \label{ex:gradedLieObject}
	Let $\g$ be a Lie algebra in $\kkvec$ and $\Gamma$ a finite abelian group acting on $\g$ via automorphisms. (For instance, if $\alpha \colon \g \to \g$ is an automorphism of order $n \in \N$, then one could take $\Gamma = \Z_n$, with a generator of $\Z_n$ acting as $\alpha$.) We get the decomposition $\g = \bigoplus\limits_{\chi \in \hat{\Gamma}} \g_{\chi}$, where $\g_{\chi}$ denotes the $\chi$-isotypic subspace of $\g$. Since $\Gamma$ acts via automorphisms, we have $[\g_{\chi}, \g_{\chi'}] \subseteq \g_{\chi \chi'}$ for all characters $\chi, \chi' \in \hat{\Gamma}$. Defining the graded Lie bracket morphisms $\begin{tikzpicture}[anchorbase]
	\draw[bmod] (0, 0) -- (0.25, 0.25) -- (0.25, 0.5);
	\draw[bmod] (0.5, 0) -- (0.25, 0.25);
	\node[anchor=north] at (0, 0) {$\scriptstyle \chi$};
	\node[anchor=north] at (0.5, 0) {$\scriptstyle \chi'$};
	\node[anchor=south] at (0.25, 0.5) {$\scriptstyle \chi\chi'$};
	\end{tikzpicture} \colon \g_{\chi} \otimes \g_{\chi'} \to \g_{\chi \chi'}$ to be the restrictions of the Lie bracket for $\g$ with the indicated domain, we thus have that $\g$ is a $\hat{\Gamma}$-graded Lie algebra in $\kkvec$. A module for this graded Lie algebra is simply a module $\green{V}$ for $\g$ considered as an ungraded Lie algebra, with the module action decomposed into the morphisms $\begin{tikzpicture}[anchorbase]
	\draw[bmod] (0, 0) -- (0.25, 0.25);
	\draw[gmod] (0.5, 0) -- (0.25, 0.25) -- (0.25, 0.5);
	\node[anchor=north] at (0, 0) {$\scriptstyle \chi$};
	%Phantom top label to get correct vertical positioning.
	\node[anchor=south] at (0.25, 0.5) {$\vphantom{\scriptstyle a}$};
	\end{tikzpicture} \colon \g_\chi \otimes \green{V} \to \green{V}$ obtained by restricting the overall action $\g \otimes \green{V} \to \green{V}$.
\end{eg}

Throughout the rest of Section \ref{s:lieAlgebraObjects}, let $\Gamma$ be an abelian group, $L$ a $\hat{\Gamma}$-graded Lie algebra, and $A$ a commutative $\kk$-algebra. For use in examples, let $\g$ be a Lie algebra in $\kkvec$. Suppose that $\Gamma$ acts on $A$ and $\g$ via automorphisms, inducing the decompositions $A = \bigoplus\limits_{\chi \in \hat{\Gamma}} A_\chi$ and $\g = \bigoplus\limits_{\chi \in \hat{\Gamma}} \g_\chi$, where $A_\chi$ denotes the $\chi$-isotypic subspace of $A$, and similarly for $\g$.

\begin{defin} \label{def:equivariantModule}
	An \emph{equivariant $(\Gamma, L, A)$-module} is an object $\green{V}$ equipped with a family of \emph{action morphisms}, $\left\{ \deamapmult{a}{f} : f \in \hat{\Gamma}, a \in A_{f^{-1}} \right\}$, satisfying the following identity for all $f \in \hat{\Gamma}$, $a, b \in A_{f^{-1}}$, and $\lambda, \mu \in \kk$:
	\begin{equation} \label{eq:equivariantAMapLinearity}
	\deamapmult{\lambda a + \mu b}{f} \ =\ \lambda \deamapmult{a}{f} + \mu \deamapmult{b}{f},
	\end{equation}
	and the following identity for all $f, g \in \hat{\Gamma}$, $a \in A_{f^{-1}}$, and $b \in A_{g^{-1}}$:
	\begin{equation} \label{eq:equivariantAMapCompatibility}
	\begin{tikzpicture}[anchorbase]
	\draw[bmod] (0.25, 0) -- (0.5, 0.25);
	\draw[bmod] (0, 0) -- (0, 0.25) -- (0.25, 0.5);
	\draw[gmod] (0.75, 0) -- (0.5, 0.25) -- (0.25, 0.5) -- (0.25, 0.75);
	\btoken{0.5, 0.25}{west}{b};
	\btoken{0.25, 0.5}{west}{a};
	\node[anchor=north] at (0, 0) {$\scriptstyle f$};
	\node[anchor=north] at (0.25, 0) {$\scriptstyle g$};
	%Phantom top label to get correct vertical positioning.
	\node[anchor=south] at (0.25, 0.75) {$\vphantom{\scriptstyle f}$};
	\end{tikzpicture}
	\ =\
	\begin{tikzpicture}[anchorbase]
	\draw[bmod] (0.5, 0) -- (0.25, 0.25);
	\draw[bmod] (0, 0) -- (0.25, 0.25) -- (0.5, 0.5);
	\draw[gmod] (0.75, 0) -- (0.75, 0.25) -- (0.5, 0.5) -- (0.5, 0.75);
	\btoken[2px]{0.5, 0.5}{west}{ab};
	\node[anchor=north] at (0, 0) {$\scriptstyle f$};
	\node[anchor=north] at (0.5, 0) {$\scriptstyle g$};
	%Phantom top label to get correct vertical positioning.
	\node[anchor=south] at (0.5, 0.75) {$\vphantom{\scriptstyle f}$};
	\end{tikzpicture}
	\ +\
	\begin{tikzpicture}[anchorbase]
	\draw[bmod] (0, -0.25) -- (0.25, 0) -- (0.5, 0.25);
	\draw[bmod] (0.25, -0.25) -- (0, 0) -- (0, 0.25) -- (0.25, 0.5);
	\draw[gmod] (0.75, -0.25) -- (0.75, 0) -- (0.5, 0.25) -- (0.25, 0.5) -- (0.25, 0.75);
	\btoken{0.5, 0.25}{west}{a};
	\btoken[2px]{0.25, 0.5}{west}{b};
	\node[anchor=north] at (0, -0.25) {$\scriptstyle f$};
	\node[anchor=north] at (0.25, -0.25) {$\scriptstyle g$};
	%Phantom top label to get correct vertical positioning.
	\node[anchor=south] at (0.25, 0.75) {$\vphantom{\scriptstyle f}$};
	\end{tikzpicture}\ .
	\end{equation}
	If $\green{V}$ and $\red{W}$ are equivariant $(\Gamma, L, A)$-modules (with respect to the same group $\Gamma$ and the same action of $\Gamma$ on $A$), a morphism from $\green{V} \to \red{W}$ is a $\CC$-morphism $\begin{tikzpicture}[anchorbase]
	\draw[gmod] (0, 0) -- (0, 0.25);
	\draw[rmod] (0, 0.25) -- (0, 0.5);
	\adot{(0, 0.25)};
	\end{tikzpicture}$ satisfying the following identity for all $f \in \hat{\Gamma}$ and $a \in A_{f^{-1}}$:
	\begin{equation} \label{eq:equivariantAMapMorphism}
	\begin{tikzpicture}[anchorbase]
	\draw[bmod] (0, -0.25) -- (0, 0) -- (0.25, 0.25);
	\draw[gmod] (0.5, -0.25) -- (0.5, 0);
	\draw[rmod] (0.5, 0) -- (0.25, 0.25) -- (0.25, 0.5);
	\btoken{0.25, 0.25}{west}{a};
	\adot{(0.5, 0)};
	\node[anchor=north] at (0, -0.25) {$\scriptstyle f$};
	%Phantom top label to get correct vertical positioning.
	\node[anchor=south] at (0.25, 0.5) {$\vphantom{\scriptstyle f}$};
	\end{tikzpicture}
	\ =\
	\begin{tikzpicture}[anchorbase]
	\draw[bmod] (0, 0) -- (0.25, 0.25);
	\draw[gmod] (0.5, 0) -- (0.25, 0.25) -- (0.25, 0.5);
	\draw[rmod] (0.25, 0.5) -- (0.25, 0.75);
	\btoken{0.25, 0.25}{west}{a};
	\adot{(0.25, 0.5)};
	\node[anchor=north] at (0, 0) {$\scriptstyle f$};
	%Phantom top label to get correct vertical positioning.
	\node[anchor=south] at (0.25, 0.75) {$\vphantom{\scriptstyle f}$};
	\end{tikzpicture}\ .
	\end{equation}
	We write $(\Gamma, L, A)\dashmod$ for the category of $(\Gamma, L, A)$-modules. Note that this module category is defined with respect to a fixed choice of $\Gamma$-action on $A$; we omit the action from our notation for the sake of brevity.
\end{defin}

\begin{eg} \label{ex:equivariantMapAlgebra}

Let $(\g \otimes A)^\Gamma = \{\alpha \in \g \otimes A \mid g \cdot \alpha = \alpha \text{ for all } g \in \Gamma \}$ be the Lie subalgebra of $\Gamma$-fixed points of $\g \otimes A$, where $\Gamma$ acts diagonally on simple tensors: $g \cdot (x \otimes a) = (g \cdot x) \otimes (g \cdot a)$. Such a Lie algebra $(\g \otimes A)^\Gamma$ is known as an \emph{equivariant map algebra}. It is straightforward to see that $(\g \otimes A)^\Gamma = \bigoplus\limits_{f \in \hat{\Gamma}}  \g_f \otimes A_{f^{-1}}$.
\details{We have $\g \otimes A = \bigoplus\limits_{f \in \hat{\Gamma}} \bigoplus\limits_{g \in \hat{\Gamma}} \g_f \otimes A_{g^{-1}}$. For all $x \otimes a \in \g_f \otimes A_{g^{-1}}$ and $h \in \Gamma$ we have $h \cdot (x \otimes a) = f(h)g^{-1}(h)x \otimes a = (fg^{-1}(h)x \otimes a$, so each subspace $\g_f \otimes A_{g^{-1}}$ is invariant under the action of $\Gamma$. Hence $(\g \otimes A)^\Gamma$ is spanned by the elements of each $\g_f \otimes A_{g^{-1}}$ that are fixed by $\Gamma$. By the previous calculation, a simple tensor $x \otimes a \in \g_f \otimes A_{g^{-1}}$ is fixed by $\Gamma$ if and only if $fg^{-1} = \id$, i.e.\ $f = g$. Extending by linearity yields the desired result.}
	
Modules for $(\g \otimes A)^\Gamma$ are equivalent to equivariant $(\Gamma, \g, A)$-modules. Given a module $\green{V}$ for $(\g \otimes A)^\Gamma$, the associated equivariant $(\Gamma, \g, A)$-module is defined as follows: for each $f \in \hat{\Gamma}$, $x \in \g_f$, $\green{v} \in \green{V}$, and $a \in A_{f^{-1}}$, set $\deamapmult{a}{f}(x \otimes \green{v}) = (x \otimes a) \cdot \green{v}$, where $\cdot$ denotes the action of $(\g \otimes A)^\Gamma$ on $\green{V}$; note that we have $x \otimes a \in (\g_f \otimes A_{f^{-1}}) \subseteq (\g \otimes A)^\Gamma$, so these action morphisms are well-defined. To see that \eqref{eq:equivariantAMapCompatibility} holds, let $f, g \in \hat{\Gamma}, a \in A_{f^{-1}}, b \in A_{g^{-1}}$, $x \in \g_{f}, y \in \g_{g}$, and $\green{v} \in \green{V}$. Then:
	\begin{multline*}
	\begin{tikzpicture}[anchorbase]
	\draw[bmod] (0.25, 0) -- (0.5, 0.25);
	\draw[bmod] (0, 0) -- (0, 0.25) -- (0.25, 0.5);
	\draw[gmod] (0.75, 0) -- (0.5, 0.25) -- (0.25, 0.5) -- (0.25, 0.75);
	\btoken{0.5, 0.25}{west}{b};
	\btoken{0.25, 0.5}{west}{a};
	\node[anchor=north] at (0, 0) {$\scriptstyle f$};
	\node[anchor=north] at (0.25, 0) {$\scriptstyle g$};
	%Phantom top label to get correct vertical positioning.
	\node[anchor=south] at (0.25, 0.75) {$\vphantom{\scriptstyle f}$};
	\end{tikzpicture}(x \otimes y \otimes \green{v}) 
	= 
	(x \otimes a) \cdot ((y \otimes b) \cdot \green{v})
	= 
	[x \otimes a, y \otimes b] \cdot \green{v} + (y \otimes b) \cdot ((x \otimes a) \cdot \green{v})\\
	=
	[x, y] \otimes ab \cdot \green{v} + (y \otimes b) \cdot ((x \otimes a) \cdot \green{v})
	= 
	\begin{tikzpicture}[anchorbase]
	\draw[bmod] (0.5, 0) -- (0.25, 0.25);
	\draw[bmod] (0, 0) -- (0.25, 0.25) -- (0.5, 0.5);
	\draw[gmod] (0.75, 0) -- (0.75, 0.25) -- (0.5, 0.5) -- (0.5, 0.75);
	\btoken[2px]{0.5, 0.5}{west}{ab};
	\node[anchor=north] at (0, 0) {$\scriptstyle f$};
	\node[anchor=north] at (0.5, 0) {$\scriptstyle g$};
	%Phantom top label to get correct vertical positioning.
	\node[anchor=south] at (0.5, 0.75) {$\vphantom{\scriptstyle f}$};
	\end{tikzpicture}(x \otimes y \otimes \green{v})
	\ +\
	\begin{tikzpicture}[anchorbase]
	\draw[bmod] (0, -0.25) -- (0.25, 0) -- (0.5, 0.25);
	\draw[bmod] (0.25, -0.25) -- (0, 0) -- (0, 0.25) -- (0.25, 0.5);
	\draw[gmod] (0.75, -0.25) -- (0.75, 0) -- (0.5, 0.25) -- (0.25, 0.5) -- (0.25, 0.75);
	\btoken{0.5, 0.25}{west}{a};
	\btoken[2px]{0.25, 0.5}{west}{b};
	\node[anchor=north] at (0, -0.25) {$\scriptstyle f$};
	\node[anchor=north] at (0.25, -0.25) {$\scriptstyle g$};
	%Phantom top label to get correct vertical positioning.
	\node[anchor=south] at (0.25, 0.75) {$\vphantom{\scriptstyle f}$};
	\end{tikzpicture}(x \otimes y \otimes \green{v}),
	\end{multline*}
as desired.

\end{eg}

\begin{rem} \label{rem:generalizationsExplained}
	It is straightforward to see that ungraded Lie algebras are equivalent to $\{1\}$-graded Lie algebras, and that $(L, A)$-modules are equivalent to equivariant $(\{1\}, \tilde{L}, A)$-modules, where $\tilde{L}$ denotes the $\{1\}$-graded Lie algebra corresponding to an ungraded Lie algebra $L$.
	\details{Given an ungraded Lie algebra $L$, the associated $\{1\}$-graded Lie algebra consists of the singleton object set $\{L_1\} = \{L\}$ and the singleton graded Lie bracket set $\left\{\begin{tikzpicture}[anchorbase]
		\draw[bmod] (0, 0) -- (0.25, 0.25) -- (0.25, 0.5);
		\draw[bmod] (0.5, 0) -- (0.25, 0.25);
		\node[anchor=north] at (0, 0) {$\scriptstyle 1$};
		\node[anchor=north] at (0.5, 0) {$\scriptstyle 1$};
		\node[anchor=south] at (0.25, 0.5) {$\scriptstyle 1$};
		\end{tikzpicture}\right\} = \left\{\dmult \right\}$. If $\green{V}$ is an $(L, A)$-module, the associated equivariant $(\{1\}, \tilde{L}, A)$-module is $\green{V}$ equipped with the action morphisms $\deamapmult{a}{1} = \damapmult{a}$.}
Due to these equivalences and what was outlined in Example \ref{ex:kMapLModule}, any result or construction for equivariant $(\Gamma, L, A)$-modules can be easily specialized to the cases of $(L, A)$-modules and $L$-modules. For instance, the definition of dual modules in Lemma \ref{lem:dualModule} is only explicitly stated for $(\Gamma, L, A)$-modules, but it applies to $(L, A)$-modules and $L$-modules as well.
\end{rem}

\begin{defin} \label{def:idealInvariantSubgroup}
	Let $\mathfrak{m} \subseteq A$ be an ideal. We write $\Gamma_\mathfrak{m} = \{g \in \Gamma \mid g \cdot \mathfrak{m} = \mathfrak{m} \}$ for the subgroup of elements that leave $\mathfrak{m}$ invariant.
\end{defin}

\begin{lem} \label{lem:twistedCharacterEvaluation}
	Let $\mathfrak{m} \subseteq A$ be a maximal ideal such that $A/\mathfrak{m} \isom \kk$. Suppose that $f \in \hat{\Gamma}$ is a character that is nontrivial when restricted to $\Gamma_\mathfrak{m}$. Then the canonical projection $A_f \to A/\mathfrak{m}$ is the zero map.
\end{lem}

\begin{proof}
	By assumption on $f$, there exists some $g \in \Gamma_\mathfrak{m}$ such that $f(g) \neq 1$. Note that the action of $\Gamma$ on $A$ induces an action of $\Gamma_\mathfrak{m}$ on $A/\mathfrak{m}$ (namely, given by $g \cdot (a + \mathfrak{m}) = (g \cdot a) + \mathfrak{m}$). Since $A/\mathfrak{m} \isom \kk$ is one-dimensional, $g$ acts as a scalar on this quotient. Since $\Gamma$ acts on $A$ via automorphisms, $g \cdot 1 = 1$. We have $1 \notin \mathfrak{m}$ since $\mathfrak{m}$ is maximal, and hence $g$ fixes the nonzero element $1 + \mathfrak{m}$ of $A/\mathfrak{m}$. Hence $g$ acts as the identity on $A/\mathfrak{m}$. For all $a \in A_f$ we thus have:
	\begin{equation*}
	a + \mathfrak{m} = g \cdot (a + \mathfrak{m}) = (g \cdot a) + \mathfrak{m} = f(g)a + \mathfrak{m},
	\end{equation*}
	where we use the fact that $a \in A_f$ in the last step. Thus $(1 - f(g))a \in \mathfrak{m}$. Since $f(g) \neq 1$, we find that $a \in \mathfrak{m}$. Hence $A_f \subseteq \mathfrak{m}$, and the projection $A_f \to A/\mathfrak{m}$ is zero.
\end{proof}

\begin{prop} \label{prop:equivariantEvaluationModule}
	Let $\mathfrak{m} \subseteq A$ be a maximal ideal such that $A/\mathfrak{m} \isom \kk$, and write $\ev_\mathfrak{m} \colon A \to \kk$ for the associated evaluation map. Let $(\hat{\Gamma})^\mathfrak{m} \subseteq \hat{\Gamma}$ denote the subgroup consisting of characters that are trivial on $\Gamma_\mathfrak{m}$. Let $L_\mathfrak{m}$ denote the $(\hat{\Gamma})^\mathfrak{m}$-graded Lie subalgebra of $L$. Explicitly, the underlying set of objects for $L_\mathfrak{m}$ is the subset $\{L_g : g \in (\hat{\Gamma})^{\mathfrak{m}}\}$ of the object set for $L$, and its set of graded Lie brackets is $\left\{\begin{tikzpicture}[anchorbase]
	\draw[bmod] (0, 0) -- (0.25, 0.25) -- (0.25, 0.5);
	\draw[bmod] (0.5, 0) -- (0.25, 0.25);
	\node[anchor=north] at (0, 0) {$\scriptstyle g$};
	\node[anchor=north] at (0.5, 0) {$\scriptstyle h$};
	\node[anchor=south] at (0.25, 0.5) {$\scriptstyle g + h$};
	\end{tikzpicture} : g, h \in (\hat{\Gamma})^\mathfrak{m} \right\}$, i.e.\ the subset of brackets for $L$ that are labelled by elements of $(\hat{\Gamma})^\mathfrak{m}$. Suppose that $\green{V}$ is a module for $L_\mathfrak{m}$. For each $f \in (\hat{\Gamma})^\mathfrak{m}$ and $a \in A_{f^{-1}}$, define $\deamapmult{a}{f} = \ev_\mathfrak{m}(a)\begin{tikzpicture}[anchorbase]
	\draw[bmod] (0, 0) -- (0.25, 0.25);
	\draw[gmod] (0.5, 0) -- (0.25, 0.25) -- (0.25, 0.5);
	\node[anchor=north] at (0, 0) {$\scriptstyle f$};
	%Phantom top label to get correct vertical positioning.
	\node[anchor=south] at (0.25, 0.5) {$\vphantom{\scriptstyle f}$};
	\end{tikzpicture}$, and for each $f \in \hat{\Gamma} \setminus (\hat{\Gamma})^{\mathfrak{m}}$ and $a \in A_{f^{-1}}$, define $\deamapmult{a}{f} = 0$. With these action morphisms, $\green{V}$ is an equivariant $(\Gamma, L, A)$-module. Such a module is called an \emph{equivariant evaluation module (at $\mathfrak{m}$)}.
\end{prop}

\begin{proof}
	To see that \eqref{eq:equivariantAMapCompatibility} holds, let $f, g \in \hat{\Gamma}$ and $a \in A_{f^{-1}}, b \in A_{g^{-1}}$. If $f, g \in (\hat{\Gamma})^\mathfrak{m}$, note that $fg \in (\hat{\Gamma})^\mathfrak{m}$ as well, and:
	\begin{multline*}
	\begin{tikzpicture}[anchorbase]
	\draw[bmod] (0.25, 0) -- (0.5, 0.25);
	\draw[bmod] (0, 0) -- (0, 0.25) -- (0.25, 0.5);
	\draw[gmod] (0.75, 0) -- (0.5, 0.25) -- (0.25, 0.5) -- (0.25, 0.75);
	\btoken{0.5, 0.25}{west}{b};
	\btoken{0.25, 0.5}{west}{a};
	\node[anchor=north] at (0, 0) {$\scriptstyle f$};
	\node[anchor=north] at (0.25, 0) {$\scriptstyle g$};
	%Phantom top label to get correct vertical positioning.
	\node[anchor=south] at (0.25, 0.75) {$\vphantom{\scriptstyle f}$};
	\end{tikzpicture} 
	= 
	\ev_\mathfrak{m}(a)\ev_\mathfrak{m}(b)\begin{tikzpicture}[anchorbase]
	\draw[bmod] (0.25, 0) -- (0.5, 0.25);
	\draw[bmod] (0, 0) -- (0, 0.25) -- (0.25, 0.5);
	\draw[gmod] (0.75, 0) -- (0.5, 0.25) -- (0.25, 0.5) -- (0.25, 0.75);
	\node[anchor=north] at (0, 0) {$\scriptstyle f$};
	\node[anchor=north] at (0.25, 0) {$\scriptstyle g$};
	%Phantom top label to get correct vertical positioning.
	\node[anchor=south] at (0.25, 0.75) {$\vphantom{\scriptstyle f}$};
	\end{tikzpicture}
	\overset{\eqref{eq:gradedLMod}}{=} 
	\ev_\mathfrak{m}(ab)\begin{tikzpicture}[anchorbase]
	\draw[bmod] (0.5, 0) -- (0.25, 0.25);
	\draw[bmod] (0, 0) -- (0.25, 0.25) -- (0.5, 0.5);
	\draw[gmod] (0.75, 0) -- (0.75, 0.25) -- (0.5, 0.5) -- (0.5, 0.75);
	\node[anchor=north] at (0, 0) {$\scriptstyle f$};
	\node[anchor=north] at (0.5, 0) {$\scriptstyle g$};
	%Phantom top label to get correct vertical positioning.
	\node[anchor=south] at (0.5, 0.75) {$\vphantom{\scriptstyle f}$};
	\end{tikzpicture}
	\ +\
	\ev_\mathfrak{m}(a)\ev_\mathfrak{m}(b) \begin{tikzpicture}[anchorbase]
	\draw[bmod] (0, -0.25) -- (0.25, 0) -- (0.5, 0.25);
	\draw[bmod] (0.25, -0.25) -- (0, 0) -- (0, 0.25) -- (0.25, 0.5);
	\draw[gmod] (0.75, -0.25) -- (0.75, 0) -- (0.5, 0.25) -- (0.25, 0.5) -- (0.25, 0.75);
	\node[anchor=north] at (0, -0.25) {$\scriptstyle f$};
	\node[anchor=north] at (0.25, -0.25) {$\scriptstyle g$};
	%Phantom top label to get correct vertical positioning.
	\node[anchor=south] at (0.25, 0.75) {$\vphantom{\scriptstyle f}$};
	\end{tikzpicture}
	= 
	\begin{tikzpicture}[anchorbase]
	\draw[bmod] (0.5, 0) -- (0.25, 0.25);
	\draw[bmod] (0, 0) -- (0.25, 0.25) -- (0.5, 0.5);
	\draw[gmod] (0.75, 0) -- (0.75, 0.25) -- (0.5, 0.5) -- (0.5, 0.75);
	\btoken[2px]{0.5, 0.5}{west}{ab};
	\node[anchor=north] at (0, 0) {$\scriptstyle f$};
	\node[anchor=north] at (0.5, 0) {$\scriptstyle g$};
	%Phantom top label to get correct vertical positioning.
	\node[anchor=south] at (0.5, 0.75) {$\vphantom{\scriptstyle f}$};
	\end{tikzpicture}
	\ +\
	\begin{tikzpicture}[anchorbase]
	\draw[bmod] (0, -0.25) -- (0.25, 0) -- (0.5, 0.25);
	\draw[bmod] (0.25, -0.25) -- (0, 0) -- (0, 0.25) -- (0.25, 0.5);
	\draw[gmod] (0.75, -0.25) -- (0.75, 0) -- (0.5, 0.25) -- (0.25, 0.5) -- (0.25, 0.75);
	\btoken{0.5, 0.25}{west}{a};
	\btoken[2px]{0.25, 0.5}{west}{b};
	\node[anchor=north] at (0, -0.25) {$\scriptstyle f$};
	\node[anchor=north] at (0.25, -0.25) {$\scriptstyle g$};
	%Phantom top label to get correct vertical positioning.
	\node[anchor=south] at (0.25, 0.75) {$\vphantom{\scriptstyle f}$};
	\end{tikzpicture}\ .
	\end{multline*}
	If $f \in (\hat{\Gamma})^\mathfrak{m}$ but $g \notin (\hat{\Gamma})^\mathfrak{m}$, then $fg \notin (\hat{\Gamma})^\mathfrak{m}$, so $\deamapmult{b}{g}$ and $\deamapmult{ab}{fg}$ are both zero, and \eqref{eq:equivariantAMapCompatibility} holds trivially. The case $f \notin (\hat{\Gamma})^\mathfrak{m}, g \in (\hat{\Gamma})^\mathfrak{m}$ is analogous. Finally, suppose that $f, g \notin (\hat{\Gamma})^\mathfrak{m}$. Thus $\deamapmult{a}{f}$ and $\deamapmult{b}{g}$ are zero, and it suffices to show that $\deamapmult{ab}{fg} = 0$. If $fg \notin (\hat{\Gamma})^\mathfrak{m}$, this holds by definition. If $fg \in (\hat{\Gamma})^\mathfrak{m}$, then $\deamapmult{ab}{fg} = \ev_\mathfrak{m}(ab)\begin{tikzpicture}[anchorbase]
	\draw[bmod] (0, 0) -- (0.25, 0.25);
	\draw[gmod] (0.5, 0) -- (0.25, 0.25) -- (0.25, 0.5);
	\node[anchor=north] at (0, 0) {$\scriptstyle fg$};
	%Phantom top label to get correct vertical positioning.
	\node[anchor=south] at (0.25, 0.5) {$\vphantom{\scriptstyle f}$};
	\end{tikzpicture}$. Using Lemma \ref{lem:twistedCharacterEvaluation}, we have $\ev_\mathfrak{m}(a) = \ev_\mathfrak{m}(b) = 0$, and hence $\ev_\mathfrak{m}(ab) = \ev_\mathfrak{m}(a)\ev_\mathfrak{m}(b) = 0$, as desired.
\end{proof}

Equivariant evaluation modules -- also called twisted evaluation modules -- have previously been studied in the category $\CC = \kkvec$; see e.g.\ \cite{equivariantMapAlgebras}\footnote{Note that the term ``evaluation module'' in \cite{equivariantMapAlgebras} refers to what we would call a tensor product of evaluation modules in the current paper -- see \cite[Rem.~4.8]{equivariantMapAlgebras}.}. In this context, equivariant evaluation modules were defined in the following way: consider an equivariant map algebra $(\g \otimes A)^\Gamma$, as in Example \ref{ex:equivariantMapAlgebra}. Let $\mathfrak{m}$ be a maximal ideal of $A$, and let $\g_\mathfrak{m} = \{x \in \g \mid g \cdot x = x \text{ for all } g \in \Gamma_\mathfrak{m} \}$ denote the Lie subalgebra of $\g$ fixed by $\Gamma_{\mathfrak{m}}$. A direct calculation shows that the image of $$\xymatrixcolsep{0.7in}\xymatrix{(\g \otimes A)^\Gamma \ar[r]^-{\id_\g \otimes \ev_\mathfrak{m}} & \g \otimes \kk \ar[r]^-{\isom} & \g }$$ lies in $\g_\mathfrak{m}$. We write $\widehat{\ev_\mathfrak{m}} \colon (\g \otimes A)^\Gamma \to \g_\mathfrak{m}$ for this composite map. Suppose $\rho \colon \g_\mathfrak{m} \to \End(V)$ is a representation of $\g_\mathfrak{m}$, where $V$ is some vector space. The associated equivariant evaluation module is then given by the composition $\xymatrix{(\g \otimes A)^\Gamma \ar[r]^-{\widehat{\ev_\mathfrak{m}}} & \g_\mathfrak{m} \ar[r]^-{\rho} & \End(V) }.$
	
To see that this definition is equivalent to the one in Proposition \ref{prop:equivariantEvaluationModule}, note that an element $x \in \g$ is fixed by all elements of $\Gamma_\mathfrak{m}$ if and only if it is a linear combination of elements $x_\chi \in \g_{\chi}$, where $\chi$ ranges over characters that are trivial on $\Gamma_\mathfrak{m}$, i.e.\ $\chi \in (\hat{\Gamma})^\mathfrak{m}$. Hence $\g_\mathfrak{m} = \bigoplus\limits_{\chi \in (\hat{\Gamma})^\mathfrak{m}} \g_\chi$, which corresponds to the $(\hat{\Gamma})^\mathfrak{m}$-graded Lie subalgebra $L_\mathfrak{m}$ used in Proposition \ref{prop:equivariantEvaluationModule}. In light of Lemma \ref{lem:twistedCharacterEvaluation}, for any $\chi \in \hat{\Gamma} \setminus (\hat{\Gamma})^\mathfrak{m}$, $x \in \g_\chi$, and $a \in A_{\chi^{-1}}$, we have $\ev_\mathfrak{m}(a) = 0$, and hence $x \otimes a \in (\g \otimes A)^\Gamma$ acts on $V$ as 0. If $\chi \in (\hat{\Gamma})^\mathfrak{m}$, then $x \otimes a$ acts as $\ev_\mathfrak{m}(a)x$ (via the assumed action of $\g_\mathfrak{m}$ on $V$). This matches the definition of the action morphisms in Proposition \ref{prop:equivariantEvaluationModule}.

\subsection{Further properties, examples, and constructions}

\begin{prop} \label{prop:symmetricMonoidalModules}
	The category $(\Gamma, L, A)\dashmod$ has a symmetric monoidal structure given as follows. If $\green{V}$, $\red{W}$ are two equivariant $(\Gamma, L, A)$-modules, then their tensor product in $(\Gamma, L, A)\dashmod$ is $\green{V} \otimes_\CC \red{W}$, with action morphisms given by $\begin{tikzpicture}[anchorbase]
	\draw[bmod] (0, 0) -- (0.25, 0.25);
	\draw[gmod] (0.25, 0) -- (0.25, 0.25) -- (0, 0.5);
	\draw[rmod] (0.5, 0) -- (0.25, 0.25) -- (0.5, 0.5);
	\btoken{0.25, 0.25}{west}{a};
	\node[anchor=north] at (0, 0) {$\scriptstyle f$};
	%Phantom top label to get correct vertical positioning.
	\node[anchor=south] at (0.5, 0.5) {$\vphantom{\scriptstyle f}$};
	\end{tikzpicture} := \begin{tikzpicture}[anchorbase]
	\draw[bmod] (0, 0) -- (0.25, 0.25);
	\draw[gmod] (0.5, 0) -- (0.25, 0.25) -- (0.25, 0.5);
	\draw[rmod] (0.75, 0) -- (0.75, 0.5);
	\btoken{0.25, 0.25}{east}{a};
	\node[anchor=north] at (0, 0) {$\scriptstyle f$};
	%Phantom top label to get correct vertical positioning.
	\node[anchor=south] at (0.75, 0.5) {$\vphantom{\scriptstyle f}$};
	\end{tikzpicture}
	\ +\
	\begin{tikzpicture}[anchorbase]
	\draw[bmod] (0, 0) -- (0.5, 0.5);
	\draw[gmod] (0.5, 0) -- (0, 0.5) -- (0, 0.75);
	\draw[rmod] (1, 0) -- (0.5, 0.5) -- (0.5, 0.75);
	\btoken{0.5, 0.5}{west}{a};
	\node[anchor=north] at (0, 0) {$\scriptstyle f$};
	%Phantom top label to get correct vertical positioning.
	\node[anchor=south] at (0.5, 0.75) {$\vphantom{\scriptstyle f}$};
	\end{tikzpicture}$ for each $f \in \hat{\Gamma}$ and $a \in A_{f^{-1}}$. The tensor product on morphisms and the symmetric braiding morphisms are inherited directly from $\CC$. The monoidal unit for $(\Gamma, L, A)\dashmod$ is the monoidal unit $\one$ from $\CC$, equipped with the trivial module structure.
\end{prop}

\begin{proof}
	The right-hand side of \eqref{eq:equivariantAMapCompatibility} for $\green{V} \otimes \red{W}$ is:
	\begin{multline*}
	\begin{tikzpicture}[anchorbase]
	\draw[bmod] (0, 0) -- (0.25, 0.25) -- (0.5, 0.5);
	\draw[bmod] (0.5, 0) -- (0.25, 0.25);
	\draw[gmod] (0.75, 0) -- (0.75, 0.25) -- (0.5, 0.5) -- (0.5, 0.75);
	\draw[rmod] (1, 0) -- (1, 0.75);
	\node[anchor=north] at (0, 0) {$\scriptstyle f$};
	\node[anchor=north] at (0.5, 0) {$\scriptstyle g$};
	%Phantom top label to get correct vertical positioning.
	\node[anchor=south] at (1, 0.75) {$\vphantom{\scriptstyle f}$};
	\btoken{0.5, 0.5}{east}{ab};
	\end{tikzpicture}
	\ +\
	\begin{tikzpicture}[anchorbase]
	\draw[bmod] (0, 0) -- (0.25, 0.25) -- (0.75, 0.75);
	\draw[bmod] (0.5, 0) -- (0.25, 0.25);
	\draw[gmod] (0.75, 0) -- (0.25, 0.5) -- (0.25, 1);
	\draw[rmod] (1, 0) -- (1, 0.5) -- (0.75, 0.75) -- (0.75, 1);
	\btoken{0.75, 0.75}{west}{ab};
	\node[anchor=north] at (0, 0) {$\scriptstyle f$};
	\node[anchor=north] at (0.5, 0) {$\scriptstyle g$};
	%Phantom top label to get correct vertical positioning.
	\node[anchor=south] at (0.75, 1) {$\vphantom{\scriptstyle f}$};
	\end{tikzpicture}
	\ +\
	\begin{tikzpicture}[anchorbase]
	\draw[bmod] (0.25, -0.25) -- (0, 0) -- (0, 0.25) -- (0.25, 0.5);
	\draw[bmod] (0, -0.25) -- (0.25, 0) -- (0.5, 0.25);
	\draw[gmod] (0.75, -0.25) -- (0.75, 0) -- (0.5, 0.25) -- (0.25, 0.5) -- (0.25, 0.75);
	\draw[rmod] (1, -0.25) -- (1, 0.75);
	\btoken{0.25, 0.5}{west}{b};
	\btoken{0.5, 0.25}{west}{a};
	\node[anchor=north] at (0, -0.25) {$\scriptstyle f$};
	\node[anchor=north] at (0.25, -0.25) {$\scriptstyle g$};
	%Phantom top label to get correct vertical positioning.
	\node[anchor=south] at (1, 0.75) {$\vphantom{\scriptstyle f}$};
	\end{tikzpicture}
	\ +\
	\begin{tikzpicture}[anchorbase]
	\draw[bmod] (0.25, -0.25) -- (0, 0) -- (0.75, 0.75);
	\draw[bmod] (0, -0.25) -- (0.25, 0) -- (0.5, 0.25);
	\draw[gmod] (0.75, -0.25) -- (0.75, 0) -- (0.5, 0.25) -- (0.5, 1);
	\draw[rmod] (1, -0.25) -- (1, 0.5) -- (0.75, 0.75) -- (0.75, 1);
	\btoken{0.75, 0.75}{west}{b};
	\btoken{0.5, 0.25}{west}{a};
	\node[anchor=north] at (0, -0.25) {$\scriptstyle f$};
	\node[anchor=north] at (0.25, -0.25) {$\scriptstyle g$};
	%Phantom top label to get correct vertical positioning.
	\node[anchor=south] at (0.75, 1) {$\vphantom{\scriptstyle f}$};
	\end{tikzpicture}
	\ +\
	\begin{tikzpicture}[anchorbase]
	\draw[bmod] (0.25, -0.25) -- (0, 0) -- (0, 0.5) -- (0.25, 0.75);
	\draw[bmod] (0, -0.25) -- (0.25, 0) -- (0.75, 0.5);
	\draw[gmod] (0.5, -0.25) -- (0.5, 0.5) -- (0.25, 0.75) -- (0.25, 1);
	\draw[rmod] (1, -0.25) -- (1, 0.25) -- (0.75, 0.5) -- (0.75, 1);
	\btoken{0.25, 0.75}{west}{b};
	\btoken{0.75, 0.5}{west}{a};
	\node[anchor=north] at (0, -0.25) {$\scriptstyle f$};
	\node[anchor=north] at (0.25, -0.25) {$\scriptstyle g$};
	%Phantom top label to get correct vertical positioning.
	\node[anchor=south] at (0.75, 1) {$\vphantom{\scriptstyle f}$};
	\end{tikzpicture}
	\ +\
	\begin{tikzpicture}[anchorbase]
	\draw[bmod] (0.25, -0.25) -- (0, 0) -- (0, 0.25) -- (0.5, 0.75);
	\draw[bmod] (0, -0.25) -- (0.25, 0) -- (0.75, 0.5);
	\draw[gmod] (0.75, -0.25) -- (0.75, 0) -- (0, 0.75) -- (0, 1);
	\draw[rmod] (1, -0.25) -- (1, 0.25) -- (0.75, 0.5) -- (0.5, 0.75) -- (0.5, 1);
	\btoken{0.5, 0.75}{west}{b};
	\btoken{0.75, 0.5}{west}{a};
	\node[anchor=north] at (0, -0.25) {$\scriptstyle f$};
	\node[anchor=north] at (0.25, -0.25) {$\scriptstyle g$};
	%Phantom top label to get correct vertical positioning.
	\node[anchor=south] at (0.5, 1) {$\vphantom{\scriptstyle f}$};
	\end{tikzpicture}\\
	\overset{\eqref{eq:crossingIdentities}}{=} 
	\left(\begin{tikzpicture}[anchorbase]
	\draw[bmod] (0, 0) -- (0.25, 0.25) -- (0.5, 0.5);
	\draw[bmod] (0.5, 0) -- (0.25, 0.25);
	\draw[gmod] (0.75, 0) -- (0.75, 0.25) -- (0.5, 0.5) -- (0.5, 0.75);
	\draw[rmod] (1, 0) -- (1, 0.75);
	\btoken{0.5, 0.5}{east}{ab};
	\node[anchor=north] at (0, 0) {$\scriptstyle f$};
	\node[anchor=north] at (0.5, 0) {$\scriptstyle g$};
	%Phantom top label to get correct vertical positioning.
	\node[anchor=south] at (1, 0.75) {$\vphantom{\scriptstyle f}$};
	\end{tikzpicture}
	\ +\
	\begin{tikzpicture}[anchorbase]
	\draw[bmod] (0.25, -0.25) -- (0, 0) -- (0, 0.25) -- (0.25, 0.5);
	\draw[bmod] (0, -0.25) -- (0.25, 0) -- (0.5, 0.25);
	\draw[gmod] (0.75, -0.25) -- (0.75, 0) -- (0.5, 0.25) -- (0.25, 0.5) -- (0.25, 0.75);
	\draw[rmod] (1, -0.25) -- (1, 0.75);
	\btoken{0.25, 0.5}{west}{b};
	\btoken{0.5, 0.25}{west}{a};
	\node[anchor=north] at (0, -0.25) {$\scriptstyle f$};
	\node[anchor=north] at (0.25, -0.25) {$\scriptstyle g$};
	%Phantom top label to get correct vertical positioning.
	\node[anchor=south] at (1, 0.75) {$\vphantom{\scriptstyle f}$};
	\end{tikzpicture}\ 
	\right)
	\ +\
	\left(
	\begin{tikzpicture}[anchorbase]
	\draw[bmod] (0, 0) -- (0.5, 0.5) -- (0.75, 0.75) -- (1, 1);
	\draw[bmod] (0.5, 0) -- (1, 0.5) -- (0.75, 0.75);
	\draw[gmod] (0.75, 0) -- (0.25, 0.5) -- (0.25, 1.25);
	\draw[rmod] (1.25, 0) -- (1.25, 0.75) -- (1, 1) -- (1, 1.25);
	\btoken{1, 1}{west}{ab};
	\node[anchor=north] at (0, 0) {$\scriptstyle f$};
	\node[anchor=north] at (0.5, 0) {$\scriptstyle g$};
	%Phantom top label to get correct vertical positioning.
	\node[anchor=south] at (1, 1.25) {$\vphantom{\scriptstyle f}$};
	\end{tikzpicture}
	\ +\
	\begin{tikzpicture}[anchorbase]
	\draw[bmod] (0, 0) -- (0.25, 0.25) -- (0.5, 0.5) -- (0.75, 0.75);
	\draw[bmod] (0.25, 0) -- (0.5, 0.25) -- (0.25, 0.5) -- (0.25, 0.75) -- (0.5, 1);
	\draw[gmod] (0.5, 0) -- (0, 0.5) -- (0, 1.25);
	\draw[rmod] (1, 0) -- (1, 0.5) -- (0.75, 0.75) -- (0.5, 1) -- (0.5, 1.25);
	\btoken{0.5, 1}{west}{b};
	\btoken{0.75, 0.75}{west}{a};
	\node[anchor=north] at (0, 0) {$\scriptstyle f$};
	\node[anchor=north] at (0.25, 0) {$\scriptstyle g$};
	%Phantom top label to get correct vertical positioning.
	\node[anchor=south] at (0.5, 1.25) {$\vphantom{\scriptstyle f}$};
	\end{tikzpicture}
	\right)
	\ +\
	\begin{tikzpicture}[anchorbase]
	\draw[bmod] (0, 0) -- (0.25, 0.25) -- (0, 0.5) -- (0.25, 0.75);
	\draw[bmod] (0.25, 0) -- (0, 0.25) -- (0.25, 0.5) -- (0.75, 1);
	\draw[gmod] (0.5, 0) -- (0.5, 0.5) -- (0.25, 0.75) -- (0.25, 1.25);
	\draw[rmod] (1, 0) -- (1, 0.75) -- (0.75, 1) -- (0.75, 1.25);
	\btoken{0.75, 1}{east}{b};
	\btoken{0.25, 0.75}{east}{a};
	\node[anchor=north] at (0, 0) {$\scriptstyle f$};
	\node[anchor=north] at (0.25, 0) {$\scriptstyle g$};
	%Phantom top label to get correct vertical positioning.
	\node[anchor=south] at (0.75, 1.25) {$\vphantom{\scriptstyle f}$};
	\end{tikzpicture}
	\ +\
	\begin{tikzpicture}[anchorbase]
	\draw[bmod] (0, 0) -- (0.75, 0.75);
	\draw[bmod] (0.25, 0) -- (0.5, 0.25);
	\draw[gmod] (0.75, 0) -- (0.5, 0.25) -- (0.5, 1);
	\draw[rmod] (1, 0) -- (1, 0.5) -- (0.75, 0.75) -- (0.75, 1);
	\btoken{0.75, 0.75}{west}{a};
	\btoken{0.5, 0.25}{west}{b};
	\node[anchor=north] at (0, 0) {$\scriptstyle f$};
	\node[anchor=north] at (0.25, 0) {$\scriptstyle g$};
	%Phantom top label to get correct vertical positioning.
	\node[anchor=south] at (0.75, 1) {$\vphantom{\scriptstyle f}$};
	\end{tikzpicture}\\
	\overset{\eqref{eq:crossingIdentities}}{=}
	\left(\begin{tikzpicture}[anchorbase]
	\draw[bmod] (0, 0) -- (0.25, 0.25) -- (0.5, 0.5);
	\draw[bmod] (0.5, 0) -- (0.25, 0.25);
	\draw[gmod] (0.75, 0) -- (0.75, 0.25) -- (0.5, 0.5) -- (0.5, 0.75);
	\draw[rmod] (1, 0) -- (1, 0.75);
	\btoken{0.5, 0.5}{east}{ab};
	\node[anchor=north] at (0, 0) {$\scriptstyle f$};
	\node[anchor=north] at (0.5, 0) {$\scriptstyle g$};
	%Phantom top label to get correct vertical positioning.
	\node[anchor=south] at (1, 0.75) {$\vphantom{\scriptstyle f}$};
	\end{tikzpicture}
	\ +\
	\begin{tikzpicture}[anchorbase]
	\draw[bmod] (0.25, -0.25) -- (0, 0) -- (0, 0.25) -- (0.25, 0.5);
	\draw[bmod] (0, -0.25) -- (0.25, 0) -- (0.5, 0.25);
	\draw[gmod] (0.75, -0.25) -- (0.75, 0) -- (0.5, 0.25) -- (0.25, 0.5) -- (0.25, 0.75);
	\draw[rmod] (1, -0.25) -- (1, 0.75);
	\btoken{0.25, 0.5}{west}{b};
	\btoken{0.5, 0.25}{west}{a};
	\node[anchor=north] at (0, -0.25) {$\scriptstyle f$};
	\node[anchor=north] at (0.25, -0.25) {$\scriptstyle g$};
	%Phantom top label to get correct vertical positioning.
	\node[anchor=south] at (1, 0.75) {$\vphantom{\scriptstyle f}$};
	\end{tikzpicture}
	\right)
	\ +\
	\left(
	\begin{tikzpicture}[anchorbase]
	\draw[bmod] (0, 0) -- (0.5, 0.5) -- (0.75, 0.75) -- (1, 1);
	\draw[bmod] (0.5, 0) -- (1, 0.5) -- (0.75, 0.75);
	\draw[gmod] (0.75, 0) -- (0.25, 0.5) -- (0.25, 1.25);
	\draw[rmod] (1.25, 0) -- (1.25, 0.75) -- (1, 1) -- (1, 1.25);
	\btoken{1, 1}{west}{ab};
	\node[anchor=north] at (0, 0) {$\scriptstyle f$};
	\node[anchor=north] at (0.5, 0) {$\scriptstyle g$};
	%Phantom top label to get correct vertical positioning.
	\node[anchor=south] at (1, 1.25) {$\vphantom{\scriptstyle f}$};
	\end{tikzpicture}
	\ +\
	\begin{tikzpicture}[anchorbase]
	\draw[bmod] (0, 0) -- (0.25, 0.25) -- (0.5, 0.5) -- (0.75, 0.75);
	\draw[bmod] (0.25, 0) -- (0.5, 0.25) -- (0.25, 0.5) -- (0.25, 0.75) -- (0.5, 1);
	\draw[gmod] (0.5, 0) -- (0, 0.5) -- (0, 1.25);
	\draw[rmod] (1, 0) -- (1, 0.5) -- (0.75, 0.75) -- (0.5, 1) -- (0.5, 1.25);
	\btoken{0.5, 1}{west}{b};
	\btoken{0.75, 0.75}{west}{a};
	\node[anchor=north] at (0, 0) {$\scriptstyle f$};
	\node[anchor=north] at (0.25, 0) {$\scriptstyle g$};
	%Phantom top label to get correct vertical positioning.
	\node[anchor=south] at (0.5, 1.25) {$\vphantom{\scriptstyle f}$};
	\end{tikzpicture}
	\right)
	\ +\
	\begin{tikzpicture}[anchorbase]
	\draw[bmod] (0, 0) -- (0, 0.5) -- (0.25, 0.75);
	\draw[bmod] (0.25, 0) -- (0.75, 0.5);
	\draw[gmod] (0.5, 0) -- (0.5, 0.5) -- (0.25, 0.75) -- (0.25, 1);
	\draw[rmod] (1, 0) -- (1, 0.25) -- (0.75, 0.5) -- (0.75, 1);
	\btoken{0.25, 0.75}{east}{a};
	\btoken{0.75, 0.5}{west}{b};
	\node[anchor=north] at (0, 0) {$\scriptstyle f$};
	\node[anchor=north] at (0.25, 0) {$\scriptstyle g$};
	%Phantom top label to get correct vertical positioning.
	\node[anchor=south] at (0.75, 1) {$\vphantom{\scriptstyle f}$};
	\end{tikzpicture}
	\ +\
	\begin{tikzpicture}[anchorbase]
	\draw[bmod] (0, 0) -- (0.75, 0.75);
	\draw[bmod] (0.25, 0) -- (0.5, 0.25);
	\draw[gmod] (0.75, 0) -- (0.5, 0.25) -- (0.5, 1);
	\draw[rmod] (1, 0) -- (1, 0.5) -- (0.75, 0.75) -- (0.75, 1);
	\btoken{0.75, 0.75}{west}{a};
	\btoken{0.5, 0.25}{west}{b};
	\node[anchor=north] at (0, 0) {$\scriptstyle f$};
	\node[anchor=north] at (0.25, 0) {$\scriptstyle g$};
	%Phantom top label to get correct vertical positioning.
	\node[anchor=south] at (0.75, 1) {$\vphantom{\scriptstyle f}$};
	\end{tikzpicture}\\
	\overset{\substack{\eqref{eq:equivariantAMapCompatibility}\text{ for}\\ \text{$\green{V}$ and $\red{W}$}\\ \vspace{-5px}}}{=}
	\begin{tikzpicture}[anchorbase]
	\draw[bmod] (0, 0) -- (0, 0.25) -- (0.25, 0.5);
	\draw[bmod] (0.25, 0) -- (0.5, 0.25);
	\draw[gmod] (0.75, 0) -- (0.5, 0.25) -- (0.25, 0.5) -- (0.25, 0.75);
	\draw[rmod] (1, 0) -- (1, 0.75);
	\btoken{0.25, 0.5}{west}{a};
	\btoken{0.5, 0.25}{west}{b};
	\node[anchor=north] at (0, 0) {$\scriptstyle f$};
	\node[anchor=north] at (0.25, 0) {$\scriptstyle g$};
	%Phantom top label to get correct vertical positioning.
	\node[anchor=south] at (1, 0.75) {$\vphantom{\scriptstyle f}$};
	\end{tikzpicture}
	\ +\
	\begin{tikzpicture}[anchorbase]
	\draw[bmod] (0, 0) -- (0, 0.25) -- (0.5, 0.75);
	\draw[bmod] (0.25, 0) -- (0.75, 0.5);
	\draw[gmod] (0.75, 0) -- (0, 0.75) -- (0, 1);
	\draw[rmod] (1, 0) -- (1, 0.25) -- (0.75, 0.5) -- (0.5, 0.75) -- (0.5, 1);
	\btoken{0.5, 0.75}{west}{a};
	\btoken{0.75, 0.5}{west}{b};
	\node[anchor=north] at (0, 0) {$\scriptstyle f$};
	\node[anchor=north] at (0.25, 0) {$\scriptstyle g$};
	%Phantom top label to get correct vertical positioning.
	\node[anchor=south] at (0.5, 1) {$\vphantom{\scriptstyle f}$};
	\end{tikzpicture}
	\ +\
	\begin{tikzpicture}[anchorbase]
	\draw[bmod] (0, 0) -- (0, 0.5) -- (0.25, 0.75);
	\draw[bmod] (0.25, 0) -- (0.75, 0.5);
	\draw[gmod] (0.5, 0) -- (0.5, 0.5) -- (0.25, 0.75) -- (0.25, 1);
	\draw[rmod] (1, 0) -- (1, 0.25) -- (0.75, 0.5) -- (0.75, 1);
	\btoken{0.25, 0.75}{west}{a};
	\btoken{0.75, 0.5}{west}{b};
	\node[anchor=north] at (0, 0) {$\scriptstyle f$};
	\node[anchor=north] at (0.25, 0) {$\scriptstyle g$};
	%Phantom top label to get correct vertical positioning.
	\node[anchor=south] at (0.75, 1) {$\vphantom{\scriptstyle f}$};
	\end{tikzpicture}
	\ +\
	\begin{tikzpicture}[anchorbase]
	\draw[bmod] (0, 0) -- (0.75, 0.75);
	\draw[bmod] (0.25, 0) -- (0.5, 0.25);
	\draw[gmod] (0.75, 0) -- (0.5, 0.25) -- (0.5, 1);
	\draw[rmod] (1, 0) -- (1, 0.5) -- (0.75, 0.75) -- (0.75, 1);
	\btoken{0.75, 0.75}{west}{a};
	\btoken{0.5, 0.25}{west}{b};
	\node[anchor=north] at (0, 0) {$\scriptstyle f$};
	\node[anchor=north] at (0.25, 0) {$\scriptstyle g$};
	%Phantom top label to get correct vertical positioning.
	\node[anchor=south] at (0.75, 1) {$\vphantom{\scriptstyle f}$};
	\end{tikzpicture}, 
	\end{multline*}
	which is the left-hand side of \eqref{eq:equivariantAMapCompatibility} for $\green{V} \otimes \red{W}$, as desired. A direct calculation shows that the symmetric braidings from $\CC$ are morphisms of equivariant $(\Gamma, L, A)$-modules, and similarly for $\CC$-tensor products of equivariant $(\Gamma, L, A)$-module morphisms. It is similarly straightforward to confirm that the monoidal unit from $\CC$ equipped with the trivial module structure is the monoidal unit for $(\Gamma, L, A)\dashmod$.
\end{proof}

\begin{lem} \label{lem:dualModule}
	Let $\begin{tikzpicture}[anchorbase]
	\draw[gmod>] (0, 0) -- (0, 0.5);
	\end{tikzpicture}$ be an equivariant $(\Gamma, L, A)$-module whose underlying $\CC$-object has a dual $\begin{tikzpicture}[anchorbase]
	\draw[<gmod] (0, 0) -- (0, 0.5);
	\end{tikzpicture} \in \ob(\CC)$. For each $f \in \hat{\Gamma}$ and $a \in A_{f^{-1}}$, define $\begin{tikzpicture}[anchorbase]
	\draw[bmod] (0, 0) -- (0.25, 0.25);
	\draw[<gmod] (0.5, 0) -- (0.25, 0.25) -- (0.25, 0.5);
	\btoken{0.25, 0.25}{west}{a};
	\node[anchor=north] at (0, 0) {$\scriptstyle f$};
	%Phantom top label to get correct vertical positioning.
	\node[anchor=south] at (0.25, 0.5) {$\vphantom{\scriptstyle f}$};
	\end{tikzpicture} = -\ \begin{tikzpicture}[anchorbase]
	\draw[<gmod] (0.25, -0.5) to[out=90, in=270]  (0, 0.25) -- (0, 0.5) arc(180:0:0.25);
	\draw[bmod] (0, -0.5) to[out=90, in=225]  (0.25, 0.25) -- (0.5, 0.5);
	\draw[gmod] (0.5, 0.5) -- (0.75, 0.25) -- (0.75, 0) arc(180:360:0.125) -- (1, 1);
	\btoken{0.5, 0.5}{west}{a};
	\node[anchor=north] at (0, -0.5) {$\scriptstyle f$};
	%Phantom top label to get correct vertical positioning.
	\node[anchor=south] at (0.5, 0.5) {$\vphantom{\scriptstyle f}$};
	\end{tikzpicture}\ .$ Equipped with these action morphisms, $\begin{tikzpicture}[anchorbase]
	\draw[<gmod] (0, 0) -- (0, 0.5);
	\end{tikzpicture}$ is the dual module to $\begin{tikzpicture}[anchorbase]
	\draw[gmod>] (0, 0) -- (0, 0.5);
	\end{tikzpicture}$ (in $(\Gamma, L, A)\dashmod$).
 \end{lem}

\begin{proof}
	Let $f, g \in \hat{\Gamma}$, $a \in A_{f^{-1}}$, and $b \in A_{g^{-1}}$. Then:
	\begin{equation} \label{eq:dualModuleCalculation}
	\begin{tikzpicture}[anchorbase]
	\draw[bmod] (0.25, 0) -- (0.5, 0.25);
	\draw[bmod] (0, 0) -- (0, 0.25) -- (0.25, 0.5);
	\draw[<gmod] (0.75, 0) -- (0.5, 0.25) -- (0.25, 0.5) -- (0.25, 0.75);
	\btoken[1px]{0.5, 0.25}{west}{b};
	\btoken{0.25, 0.5}{west}{a};
	\node[anchor=north] at (0, 0) {$\scriptstyle f$};
	\node[anchor=north] at (0.25, 0) {$\scriptstyle g$};
	%Phantom top label to get correct vertical positioning.
	\node[anchor=south] at (0.25, 0.75) {$\vphantom{\scriptstyle f}$};
	\end{tikzpicture}
	=
	\begin{tikzpicture}[anchorbase]
	\draw[<gmod] (0.5, -1.25) -- (0.5, -1) to[out=90, in=270] (0.25, -0.75) -- (0.25, -0.375) arc(180:0:0.25) -- (0.75, -0.375) -- (1, -0.625) -- (1, -0.75) arc(180:360:0.125) -- (1.25, -0.75) -- (1.25, -0.5) to[out=90, in=270]  (0, 0.5) -- (0, 0.75) arc(180:0:0.25) -- (0.5, 0.5);
	\draw[bmod] (0, -1.25) -- (0, -0.5) to[out=90, in=225]  (0.25, 0.25) -- (0.5, 0.5);
	\draw[gmod] (0.5, 0.5) -- (0.75, 0.25) -- (0.75, 0.125) arc(180:360:0.125) -- (1, 1) -- (1, 1.25);
	\draw[bmod] (0.25, -1.25) -- (0.25, -1) to[out=90, in=270] (0.5, -0.75) -- (0.5, -0.625) -- (0.75, -0.375);
	\btoken{0.5, 0.5}{west}{a};
	\node[anchor=north] at (0, -1.25) {$\scriptstyle f$};
	\btoken{0.75, -0.375}{west}{b};
	\node[anchor=north] at (0.25, -1.25) {$\scriptstyle g$};
	%Phantom top label to get correct vertical positioning.
	\node[anchor=south] at (1, 1.25) {$\vphantom{\scriptstyle f}$};
	\end{tikzpicture}
	\overset{\eqref{eq:crossingIdentities}}{=}
	\begin{tikzpicture}[anchorbase]
	\draw[bmod] (0, 0) -- (0, 0.25) to[out=90, in=270] (1.25, 0.75) -- (1.25, 1.25) -- (1.5, 1.5);
	\draw[bmod] (0.25, 0) -- (0.25, 0.25) to[out=90, in=270] (0, 0.5) to[out=90, in=270] (0, 0.75) to[out=90, in=270] (0.25, 1) -- (0.25, 1.25) -- (0.5, 1.5);
	\draw[<gmod] (0.5, 0) -- (0.5, 0.5) to[out=90, in=270] (0, 1) -- (0, 1.5) arc(180:0:0.25) -- (0.75, 1.25) -- (0.75, 1) arc(180:360:0.125) -- (1, 1.5) arc(180:0:0.25) -- (1.75, 1.25) -- (1.75, 1) arc(180:360:0.125) -- (2, 2);
	\btoken{0.5, 1.5}{west}{b};
	\node[anchor=north] at (0, 0) {$\scriptstyle f$};
	\btoken{1.5, 1.5}{west}{a};
	\node[anchor=north] at (0.25, 0) {$\scriptstyle g$};
	%Phantom top label to get correct vertical positioning.
	\node[anchor=south] at (2, 2) {$\vphantom{\scriptstyle f}$};
	\end{tikzpicture}
	\overset{\eqref{eq:zigzag}}{=}
	\begin{tikzpicture}[anchorbase]
	\draw[bmod] (0, 0) -- (0, 0.25) to[out=90, in=270] (0.5, 0.75) -- (0.5, 1.25) -- (0.75, 1.5);
	\draw[bmod] (0.25, 0) -- (0.25, 0.25) to[out=90, in=270] (0, 0.5) to[out=90, in=270] (0, 0.75) to[out=90, in=270] (0.25, 1) -- (0.25, 1.75) -- (0.5, 2);
	\draw[<gmod] (0.5, 0) -- (0.5, 0.5) to[out=90, in=270] (0, 1) -- (0, 2) arc(180:0:0.25) -- (0.75, 1.75) -- (0.75, 1.5) -- (1, 1.25) -- (1, 1) arc(180:360:0.125) -- (1.25, 2.5);
	\btoken{0.5, 2}{west}{b};
	\node[anchor=north] at (0, 0) {$\scriptstyle f$};
	\btoken{0.75, 1.5}{west}{a};
	\node[anchor=north] at (0.25, 0) {$\scriptstyle g$};
	%Phantom top label to get correct vertical positioning.
	\node[anchor=south] at (1.25, 2.5) {$\vphantom{\scriptstyle f}$};
	\end{tikzpicture}\ ,
	\end{equation}
	and so
	\begin{multline*}
	\begin{tikzpicture}[anchorbase]
	\draw[bmod] (0.25, 0) -- (0.5, 0.25);
	\draw[bmod] (0, 0) -- (0, 0.25) -- (0.25, 0.5);
	\draw[<gmod] (0.75, 0) -- (0.5, 0.25) -- (0.25, 0.5) -- (0.25, 0.75);
	\btoken[1 px]{0.5, 0.25}{west}{b};
	\btoken{0.25, 0.5}{west}{a};
	\node[anchor=north] at (0, 0) {$\scriptstyle f$};
	\node[anchor=north] at (0.25, 0) {$\scriptstyle g$};
	%Phantom top label to get correct vertical positioning.
	\node[anchor=south] at (0.25, 0.75) {$\vphantom{\scriptstyle f}$};
	\end{tikzpicture}
	\underset{\eqref{eq:equivariantAMapCompatibility}}{\overset{\eqref{eq:dualModuleCalculation}}{=}}
	\begin{tikzpicture}[anchorbase]
	\draw[bmod] (0, 0) -- (0, 0.25) to[out=90, in=315] (0.75, 1) -- (0.5, 1.25) -- (0.5, 1.5) -- (0.75, 1.75);
	\draw[bmod] (0.25, 0) -- (0.25, 0.25) to[out=90, in=270] (0, 0.5) to[out=90, in=270] (0, 0.75) to[out=90, in=225] (0.25, 1) -- (0.5, 1.25);
	\draw[<gmod] (0.5, 0) -- (0.5, 0.5) to[out=90, in=270] (0, 1) -- (0, 1.75) arc(180:0:0.375) -- (1, 1.5) -- (1, 1.25) arc(180:360:0.125) -- (1.25, 2.375);
	\node[anchor=north] at (0, 0) {$\scriptstyle f$};
	\btoken{0.75, 1.75}{west}{ab};
	\node[anchor=north] at (0.25, 0) {$\scriptstyle g$};
	%Phantom top label to get correct vertical positioning.
	\node[anchor=south] at (1.25, 2.375) {$\vphantom{\scriptstyle f}$};
	\end{tikzpicture}
	+
	\begin{tikzpicture}[anchorbase]
	\draw[bmod] (0, 0) -- (0, 0.25) to[out=90, in=270] (0.5, 0.75) -- (0.5, 1) to[out=90, in=270] (0.25, 1.25) -- (0.25, 1.75) -- (0.5, 2);
	\draw[bmod] (0.25, 0) -- (0.25, 0.25) to[out=90, in=270] (0, 0.5) to[out=90, in=270] (0, 0.75) to[out=90, in=225] (0.25, 1) -- (0.75, 1.5);
	\draw[<gmod] (0.5, 0) -- (0.5, 0.5) to[out=90, in=270] (0, 1) -- (0, 2) arc(180:0:0.25) -- (0.75, 1.75) -- (0.75, 1.5) -- (1, 1.25) -- (1, 1) arc(180:360:0.125) -- (1.25, 2.5);
	\btoken{0.5, 2}{west}{a};
	\node[anchor=north] at (0, 0) {$\scriptstyle f$};
	\btoken{0.75, 1.5}{west}{b};
	\node[anchor=north] at (0.25, 0) {$\scriptstyle g$};
	%Phantom top label to get correct vertical positioning.
	\node[anchor=south] at (1.25, 2.5) {$\vphantom{\scriptstyle f}$};
	\end{tikzpicture}
	\overset{\eqref{eq:crossingIdentities}}{=}
	\begin{tikzpicture}[anchorbase]
	\draw[bmod] (0, 0) -- (0, 0.25) to[out=90, in=315] (0.5, 0.75) -- (0.25, 1);
	\draw[bmod] (0.5, 0) -- (0.5, 0.25) to[out=90, in=225] (0, 0.75) -- (0.25, 1) -- (0.25, 1.25) -- (0.25, 1.5) -- (0.5, 1.75);
	\draw[<gmod] (0.75, 0) -- (0.75, 0.75) to[out=90, in=270] (0, 1.75) arc(180:0:0.25) -- (0.75, 1.5) -- (0.75, 1.25)  arc(180:360:0.125) -- (1, 2.25);
	\node[anchor=north] at (0, 0) {$\scriptstyle f$};
	\btoken{0.5, 1.75}{west}{ab};
	\node[anchor=north] at (0.5, 0) {$\scriptstyle g$};
	%Phantom top label to get correct vertical positioning.
	\node[anchor=south] at (1, 2.25) {$\vphantom{\scriptstyle f}$};
	\end{tikzpicture}
	+
	\begin{tikzpicture}[anchorbase]
	\draw[bmod] (0.25, -0.25) -- (0.25, 0) to[out=90, in=270] (0, 0.25) to[out=90, in=270] (0.5, 0.75) -- (0.5, 1.25) -- (0.75, 1.5);
	\draw[bmod] (0, -0.25) -- (0, 0) to[out=90, in=270] (0.25, 0.25) to[out=90, in=270] (0, 0.5) to[out=90, in=270] (0, 0.75) to[out=90, in=270] (0.25, 1) -- (0.25, 1.75) -- (0.5, 2);
	\draw[<gmod] (0.5, -0.25) -- (0.5, 0.5) to[out=90, in=270] (0, 1) -- (0, 2) arc(180:0:0.25) -- (0.75, 1.75) -- (0.75, 1.5) -- (1, 1.25) -- (1, 1) arc(180:360:0.125) -- (1.25, 2.5);
	\btoken{0.5, 2}{west}{a};
	\node[anchor=north] at (0, -0.25) {$\scriptstyle f$};
	\btoken{0.75, 1.5}{west}{b};
	\node[anchor=north] at (0.25, -0.25) {$\scriptstyle g$};
	%Phantom top label to get correct vertical positioning.
	\node[anchor=south] at (1.25, 2.5) {$\vphantom{\scriptstyle f}$};
	\end{tikzpicture}\\
	\overset{\eqref{eq:gradedSkewSymmetry}}{=}
	\ -\
	\begin{tikzpicture}[anchorbase]
	\draw[bmod] (0.5, 0.75) -- (0.25, 1);
	\draw[bmod] (0, 0.75) -- (0.25, 1) -- (0.25, 1.25) -- (0.25, 1.5) -- (0.5, 1.75);
	\draw[<gmod] (0.75, 0.75) to[out=90, in=270] (0, 1.75) arc(180:0:0.25) -- (0.75, 1.5) -- (0.75, 1.25)  arc(180:360:0.125) -- (1, 2.25);
	\node[anchor=north] at (0, 0.75) {$\scriptstyle f$};
	\btoken{0.5, 1.75}{west}{ab};
	\node[anchor=north] at (0.5, 0.75) {$\scriptstyle g$};
	%Phantom top label to get correct vertical positioning.
	\node[anchor=south] at (1, 2.25) {$\vphantom{\scriptstyle f}$};
	\end{tikzpicture}
	+
	\begin{tikzpicture}[anchorbase]
	\draw[bmod] (0.25, -0.25) -- (0.25, 0) to[out=90, in=270] (0, 0.25) to[out=90, in=270] (0.5, 0.75) -- (0.5, 1.25) -- (0.75, 1.5);
	\draw[bmod] (0, -0.25) -- (0, 0) to[out=90, in=270] (0.25, 0.25) to[out=90, in=270] (0, 0.5) to[out=90, in=270] (0, 0.75) to[out=90, in=270] (0.25, 1) -- (0.25, 1.75) -- (0.5, 2);
	\draw[<gmod] (0.5, -0.25) -- (0.5, 0.5) to[out=90, in=270] (0, 1) -- (0, 2) arc(180:0:0.25) -- (0.75, 1.75) -- (0.75, 1.5) -- (1, 1.25) -- (1, 1) arc(180:360:0.125) -- (1.25, 2.5);
	\btoken{0.5, 2}{west}{a};
	\node[anchor=north] at (0, -0.25) {$\scriptstyle f$};
	\btoken{0.75, 1.5}{west}{b};
	\node[anchor=north] at (0.25, -0.25) {$\scriptstyle g$};
	%Phantom top label to get correct vertical positioning.
	\node[anchor=south] at (1.25, 2.5) {$\vphantom{\scriptstyle f}$};
	\end{tikzpicture}
	\overset{\eqref{eq:dualModuleCalculation}}{=}
	\begin{tikzpicture}[anchorbase]
	\draw[bmod] (0.5, 0) -- (0.25, 0.25);
	\draw[bmod] (0, 0) -- (0.25, 0.25) -- (0.5, 0.5);
	\draw[<gmod] (0.75, 0) -- (0.75, 0.25) -- (0.5, 0.5) -- (0.5, 0.75);
	\btoken[2px]{0.5, 0.5}{west}{ab};
	\node[anchor=north] at (0, 0) {$\scriptstyle f$};
	\node[anchor=north] at (0.5, 0) {$\scriptstyle g$};
	%Phantom top label to get correct vertical positioning.
	\node[anchor=south] at (0.5, 0.75) {$\vphantom{\scriptstyle f}$};
	\end{tikzpicture}
	\ +\
	\begin{tikzpicture}[anchorbase]
	\draw[bmod] (0, -0.25) -- (0.25, 0) -- (0.5, 0.25);
	\draw[bmod] (0.25, -0.25) -- (0, 0) -- (0, 0.25) -- (0.25, 0.5);
	\draw[<gmod] (0.75, -0.25) -- (0.75, 0) -- (0.5, 0.25) -- (0.25, 0.5) -- (0.25, 0.75);
	\btoken{0.5, 0.25}{west}{a};
	\btoken[2px]{0.25, 0.5}{west}{b};
	\node[anchor=north] at (0, -0.25) {$\scriptstyle f$};
	\node[anchor=north] at (0.25, -0.25) {$\scriptstyle g$};
	%Phantom top label to get correct vertical positioning.
	\node[anchor=south] at (0.25, 0.75) {$\vphantom{\scriptstyle f}$};
	\end{tikzpicture}\ ,
	\end{multline*}
	showing that the action morphisms for $\begin{tikzpicture}[anchorbase]
	\draw[<gmod] (0, 0) -- (0, 0.5);
	\end{tikzpicture}$ satisfy \eqref{eq:equivariantAMapCompatibility}. To see that $\begin{tikzpicture}[anchorbase]
	\draw[gmod>] (0, 0) -- (0, 0.5);
	\end{tikzpicture}$ and $\begin{tikzpicture}[anchorbase]
	\draw[<gmod] (0, 0) -- (0, 0.5);
	\end{tikzpicture}$ are dual modules, it suffices to show that $\begin{tikzpicture}[anchorbase]
	\draw[gmod>] (0, 0) -- (0, 0.25) arc(180:0:0.25) -- (0.5, 0);
	\end{tikzpicture}$ and $\begin{tikzpicture}[anchorbase, yscale=-1]
	\draw[gmod>] (0, 0) -- (0, 0.25) arc(180:0:0.25) -- (0.5, 0);
	\end{tikzpicture}$ (that is, the cap and cup for the underlying dual pair of objects in $\CC$) are morphisms of equivariant $(\Gamma, L, A)$-modules. We will prove the claim for the cap; the argument for the cup is analogous. Let $f \in \hat{\Gamma}$ and $a \in A_{f^{-1}}$. We compute:
	\begin{multline*}
	\begin{tikzpicture}[anchorbase]
	\draw[bmod] (0, 0) -- (0.5, 0.5);
	\draw[gmod>] (0.5, 0) -- (0, 0.5) -- (0, 0.75) arc(180:0:0.25) -- (0.5, 0.5) -- (1, 0);
	\btoken{0.5, 0.5}{west}{a};
	\node[anchor=north] at (0, 0) {$\scriptstyle f$};
	%Phantom top label to get correct vertical positioning.
	\node[anchor=south] at (0.25, 0.75) {$\vphantom{\scriptstyle f}$};
	\end{tikzpicture}
	=
	-\
	\begin{tikzpicture}[anchorbase]
	\draw[bmod] (0, 0) -- (0.75, 0.75);
	\draw[<gmod] (0.75, 0) to[out=90, in=270] (0.25, 0.5) -- (0.25, 0.75) arc(180:0:0.25) to[out=315, in=90] (1, 0.5) arc(180:360:0.125) -- (1.25, 1) arc(0:180:0.625) -- (0, 0.375) to[out=270, in=90] (0.375, 0);
	\btoken{0.75, 0.75}{west}{a};
	\node[anchor=north] at (0, 0) {$\scriptstyle f$};
	\end{tikzpicture}
	\overset{\eqref{eq:crossingIdentities}}{=}
	-\ 
	\begin{tikzpicture}[anchorbase]
	\draw[bmod] (0, 0) to[out=90, in=270] (0.375, 0.375) -- (0.375, 0.5)  to[out=90, in=270] (0, 1) -- (0, 1.25) to[out=90, in=225] (0.25, 1.5);
	\draw[gmod>] (0.25, 0) to[out=90, in=270] (0, 0.25) -- (0, 0.75) arc(180:0:0.5) arc(360:180:0.25) to[out=90, in=315] (0.25, 1.5) arc(0:180:0.25) -- (-0.25, 0.75) to[out=270, in=89] (0.5, 0);
	\btoken{0.25, 1.5}{west}{a};
	\node[anchor=north] at (0, 0) {$\scriptstyle f$};
	\end{tikzpicture}
	\underset{\eqref{eq:delooping}}{\overset{\eqref{eq:zigzag}}{=}}
	-\ 
	\begin{tikzpicture}[anchorbase]
	\draw[bmod] (0, 0) to[out=90, in=270] (0.75, 0.75) to[out=90, in=270] (0.25, 1.25) to[out=90, in=270] (0.5, 1.5);
	\draw[gmod>] (0.5, 0) to[out=90, in=270] (0, 0.5) to[out=90, in=270] (0.75, 1.25) to[out=90, in=270] (0.5, 1.5) -- (0.5, 1.75) arc(0:180:0.25) -- (0, 1) to[out=270, in=90] (1, 0);
	\btoken{0.5, 1.5}{west}{a};
	\node[anchor=north] at (0, 0) {$\scriptstyle f$};
	\end{tikzpicture}
	\\
	\overset{\eqref{eq:crossingIdentities}}{=}
	-\ 
	\begin{tikzpicture}[anchorbase]
	\draw[bmod] (0, 0) to[out=90, in=270] (0.75, 0.75) to[out=90, in=270] (0.5, 1) to[out=90, in=270] (0.625, 1.125);
	\draw[gmod>] (0.25, 0) to[out=90, in=270] (0.75, 0.5) to[out=90, in=270] (0.5, 0.75) to[out=90, in=270] (0.75, 1) to[out=90, in=270] (0.625, 1.125) -- (0.625, 1.375) arc(0:180:0.25) -- (0.125, 0.875) to[out=270, in=90] (0.75, 0);
	\btoken{0.625, 1.125}{west}{a};
	\node[anchor=north] at (0, 0) {$\scriptstyle f$};
	\end{tikzpicture}
	\overset{\eqref{eq:crossingIdentities}}{=}
	-\ 
	\begin{tikzpicture}[anchorbase]
	\draw[bmod] (0, 0) -- (1, 1);
	\draw[gmod>] (0.5, 0) to[out=90, in=270] (1.25, 0.75) to[out=90, in=270] (1, 1) -- (1, 1.25) arc(0:180:0.25) -- (0.5, 1) -- (0.5, 0.5) to[out=270, in=90] (1, 0);
	\btoken{1, 1}{west}{a};
	\node[anchor=north] at (0, 0) {$\scriptstyle f$};
	\end{tikzpicture}
	\overset{\eqref{eq:crossingIdentities}}{=}
	-\ 
	\begin{tikzpicture}[anchorbase]
	\draw[bmod] (0, 0) -- (0.25, 0.25);
	\draw[gmod>] (0.5, 0) -- (0.25, 0.25) -- (0.25, 0.5) to[out=90, in=270] (0.75, 1) arc(0:180:0.25) to[out=270, in=90] (0.75, 0.5) -- (0.75, 0);
	\btoken{0.25, 0.25}{east}{a};
	\node[anchor=north] at (0, 0) {$\scriptstyle f$};
	\end{tikzpicture}
	\overset{\eqref{eq:delooping}}{=}
	-\ 
	\begin{tikzpicture}[anchorbase]
	\draw[bmod] (0, 0) -- (0.25, 0.25);
	\draw[gmod>] (0.5, 0) -- (0.25, 0.25) -- (0.25, 0.5) arc(180:0:0.25) -- (0.75, 0);
	\btoken{0.25, 0.25}{east}{a};
	\node[anchor=north] at (0, 0) {$\scriptstyle f$};
	%Phantom top label to get correct vertical positioning.
	\node[anchor=south] at (0.25, 0.5) {$\vphantom{\scriptstyle f}$};
	\end{tikzpicture}\ ,
	\end{multline*}
	and hence $\begin{tikzpicture}[anchorbase]
	\draw[bmod] (0, 0) -- (0.25, 0.25);
	\draw[gmod>] (0.5, 0) -- (0.25, 0.25) -- (0.25, 0.5) arc(180:0:0.25) -- (0.75, 0);
	\btoken{0.25, 0.25}{east}{a};
	\node[anchor=north] at (0, 0) {$\scriptstyle f$};
	%Phantom top label to get correct vertical positioning.
	\node[anchor=south] at (0.25, 0.5) {$\vphantom{\scriptstyle f}$};
	\end{tikzpicture}
	\ +\
	\begin{tikzpicture}[anchorbase]
	\draw[bmod] (0, 0) -- (0.5, 0.5);
	\draw[gmod>] (0.5, 0) -- (0, 0.5) -- (0, 0.75) arc(180:0:0.25) -- (0.5, 0.5) -- (1, 0);
	\btoken{0.5, 0.5}{west}{a};
	\node[anchor=north] at (0, 0) {$\scriptstyle f$};
	%Phantom top label to get correct vertical positioning.
	\node[anchor=south] at (0.25, 0.75) {$\vphantom{\scriptstyle f}$};
	\end{tikzpicture} = 0$. Since the monoidal unit in $(\Gamma, L, A)\dashmod$ is equipped with the trivial module structure, this shows that \eqref{eq:equivariantAMapMorphism} holds, and hence $\begin{tikzpicture}[anchorbase]
	\draw[gmod>] (0, 0) -- (0, 0.25) arc(180:0:0.25) -- (0.5, 0);
	\end{tikzpicture}$ is a morphism of equivariant $(\Gamma, L, A)$-modules, as desired.
\end{proof}

\begin{lem} \label{lem:symmetricMonoidalFunctorsPreserve}
	Let $\DD$ be a symmetric monoidal $\kk$-linear category. Suppose $F \colon \CC \to \DD$ is a symmetric monoidal functor, and recall that $L$ is a $\hat{\Gamma}$-graded Lie algebra in $\CC$. Then $F(L)$ is a $\hat{\Gamma}$-graded Lie algebra in $\DD$, and there is an induced functor $F \colon (\Gamma, L, A)\dashmod \to (\Gamma, F(L), A)\dashmod$, described as follows.
	
	Let $\green{V}$ and $\red{W}$ be equivariant $(\Gamma, L, A)$-modules, and $\begin{tikzpicture}[anchorbase]
	\draw[gmod] (0, 0) -- (0, 0.25);
	\draw[rmod] (0, 0.25) -- (0, 0.5);
	\adot{(0, 0.25)};
	\end{tikzpicture}$ a morphism between them. Then $F(\green{V})$ and $F(\red{W})$ are equivariant $(\Gamma, F(L), A)$-modules, and $F\left(\begin{tikzpicture}[anchorbase]
	\draw[gmod] (0, 0) -- (0, 0.25);
	\draw[rmod] (0, 0.25) -- (0, 0.5);
	\adot{(0, 0.25)};
	\end{tikzpicture} \right)$ is a morphism of equivariant $(\Gamma, F(L), A)$-modules from $F(\green{V})$ to $F(\red{W})$. The $\hat{\Gamma}$-grading on $F(L)$ is given by $F(L)_f = F(L_f)$ for each $f \in \hat{\Gamma}$, with graded Lie bracket morphisms given by $\begin{tikzpicture}[anchorbase]
	\draw[bmod] (0, 0) -- (0.25, 0.25) -- (0.25, 0.5);
	\draw[bmod] (0.5, 0) -- (0.25, 0.25);
	\node at (-0.25, -0.25) {$\scriptstyle F(L)_f$};
	\node at (0.75, -0.25) {$\scriptstyle F(L)_g$};
	\node at (0.25, 0.75) {$\scriptstyle F(L)_{fg}$};
	\end{tikzpicture} = F\left(\begin{tikzpicture}[anchorbase]
	\draw[bmod] (0, 0) -- (0.25, 0.25) -- (0.25, 0.5);
	\draw[bmod] (0.5, 0) -- (0.25, 0.25);
	\node at (0, -0.25) {$\scriptstyle L_f$};
	\node at (0.5, -0.25) {$\scriptstyle L_g$};
	\node at (0.25, 0.75) {$\scriptstyle L_{fg}$};
	\end{tikzpicture} \right)$ for all $f, g \in \hat{\Gamma}$. The action morphisms for $F(\green{V})$ are given by $\deamapmult{a}{F(L)_f} = F \left(\deamapmult{a}{L_f} \right)$ for each $f \in \hat{\Gamma}$ and $a \in A_{f^{-1}}$, and similarly for $F(\red{W})$.
\end{lem}

\begin{proof}
	The identities needed for $F(L)$ to be a $\hat{\Gamma}$-graded Lie algebra, namely \eqref{eq:gradedSkewSymmetry} and \eqref{eq:gradedJacobi}, are the $F$-images of the corresponding identities for $L$, and hence they hold. The other claims follow similarly.
\end{proof}

In remainder of this subsection we will study various ungraded constructions, and as such we drop the assumption that $L$ is a $\hat{\Gamma}$-graded Lie algebra.

The next few definitions and lemmas will allow us to construct categorical analogues of the Lie algebra $\gl_n$, its natural module, and that module's dual.

\begin{defin} \label{def:semigroupObject}
	A \emph{semigroup} in $\CC$ is an object $S$ equipped with a morphism $\dmult \colon S \otimes S \to S$ satisfying the following associativity identity:
	\begin{equation}\label{eq:associativity}\tag{ASSOC}
	\begin{tikzpicture}[anchorbase]
	\draw[bmod] (0, 0) -- (0.25, 0.25) -- (0.5, 0.5) -- (0.5, 0.75);
	\draw[bmod] (0.5, 0) -- (0.25, 0.25);
	\draw[bmod] (1, 0) -- (0.5, 0.5);
	\end{tikzpicture} 
	\ =\
	\begin{tikzpicture}[anchorbase]
	\draw[bmod] (0, 0) -- (-0.25, 0.25) -- (-0.5, 0.5) -- (-0.5, 0.75);
	\draw[bmod] (-0.5, 0) -- (-0.25, 0.25);
	\draw[bmod] (-1, 0) -- (-0.5, 0.5);
	\end{tikzpicture}\ .
	\end{equation}
\end{defin}

\begin{lem} \label{lem:dualSemigroup}
	Suppose $\goup$ is an object with dual $\godown$. Then $\left(\goup \otimes \godown, \begin{tikzpicture}[anchorbase]
	\draw[->, thick] (0, 0) to[out=90, in=270] (0.25, 0.5) -- (0.25, 0.75);
	\draw[<-, thick] (0.25, 0) -- (0.25, 0.2) arc(180:0:0.125) -- (0.5, 0);
	\draw[<-, thick] (0.75, 0) to[out=90, in=270] (0.5, 0.5) -- (0.5, 0.75);
	\end{tikzpicture} \right)$ is a semigroup.
\end{lem}

\begin{proof}
	The left-hand side of \eqref{eq:associativity} is $\begin{tikzpicture}[anchorbase]
	\draw[->, thick] (0, 0) to[out=90, in=270] (0.25, 0.5) -- (0.25, 1.25) -- (0.25, 1.5);
	\draw[<-, thick] (0.25, 0) -- (0.25, 0.2) arc(180:0:0.125) -- (0.5, 0);
	\draw[<-, thick] (0.75, 0) to[out=90, in=270] (0.5, 0.5) -- (0.5, 0.75) arc(180:0:0.25) -- (1, 0);
	\draw[<-, thick] (1.25, 0) -- (1.25, 0.75) to[out=90, in=270] (1, 1.25) -- (1, 1.5);
	\end{tikzpicture}$, and the right-hand side is $\begin{tikzpicture}[anchorbase, xscale=-1]
	\draw[<-, thick] (0, 0) to[out=90, in=270] (0.25, 0.5) -- (0.25, 1.25) -- (0.25, 1.5);
	\draw[->, thick] (0.25, 0) -- (0.25, 0.2) arc(180:0:0.125) -- (0.5, 0);
	\draw[->, thick] (0.75, 0) to[out=90, in=270] (0.5, 0.5) -- (0.5, 0.75) arc(180:0:0.25) -- (1, 0);
	\draw[->, thick] (1.25, 0) -- (1.25, 0.75) to[out=90, in=270] (1, 1.25) -- (1, 1.5);
	\end{tikzpicture}$. These morphisms are equal due to the interchange law.
\end{proof}

\begin{lem} \label{lem:semigroupLieObject}
	Suppose $\left(S, \dmult\right)$ is a semigroup. Then $\left(S, \dmult - \begin{tikzpicture}[anchorbase]
	\draw[bmod] (0, 0.25) -- (0.5, 0.75) -- (0.25, 1) -- (0.25, 1.25);
	\draw[bmod] (0.5, 0.25) -- (0, 0.75) -- (0.25, 1);
	\end{tikzpicture} \right)$ is a Lie algebra. 
\end{lem}

\begin{proof}
	It is clear that the claimed Lie bracket satisfies \eqref{eq:skew}, and \eqref{eq:jacobi} follows straightforwardly from \eqref{eq:associativity} and \eqref{eq:crossingIdentities}.
\end{proof}

\begin{lem} \label{lem:orientedBrauerLieObject}
	Suppose $\goup$ is an object with dual $\godown$. Then $\left(\goup \otimes \godown, \begin{tikzpicture}[anchorbase]
	\draw[->, thick] (0, 0) to[out=90, in=270] (0.25, 0.5) -- (0.25, 0.75);
	\draw[<-, thick] (0.25, 0) -- (0.25, 0.2) arc(180:0:0.125) -- (0.5, 0);
	\draw[<-, thick] (0.75, 0) to[out=90, in=270] (0.5, 0.5) -- (0.5, 0.75);
	\end{tikzpicture}
	\ -\
	\begin{tikzpicture}[anchorbase]
	\draw[->, thick] (0.5, -0.75) -- (0.5, -0.5) to[out=90, in=270] (0, 0) to[out=90, in=270] (0.25, 0.5) -- (0.25, 0.75);
	\draw[<-, thick] (0.75, -0.75) -- (0.75, -0.5) to[out=90, in=270] (0.5, 0) arc(0:180:0.125) to[out=270, in=90] (0, -0.5) -- (0, -0.75);
	\draw[<-, thick] (0.25, -0.75) -- (0.25, -0.5) to[out=90, in=270] (0.75, 0) to[out=90, in=270] (0.5, 0.5) -- (0.5, 0.75);
	\end{tikzpicture} \right)$
	is a Lie algebra, and $\left(\goup,\ \begin{tikzpicture}[anchorbase]
	\draw[->, thick] (0, 0) -- (0, 0.75);
	\draw[->, thick] (0.6, 0) -- (0.6, 0.2) arc(0:180:0.2) -- (0.2, 0);
	\end{tikzpicture} \right)$ and $\left(\godown,\ -\ \begin{tikzpicture}[anchorbase]
	\draw[->, thick] (0, 0) -- (0, 0.25) arc(180:0:0.25) -- (0.5, 0);
	\draw[<-, thick] (0.25, 0) -- (0.25, 0.25) to[out=90, in=270] (0.5, 0.5) -- (0.5, 0.75);
	\end{tikzpicture} \right)$ are a dual pair of modules for this Lie algebra. We call $\goup$ the \emph{natural module} for $\goup \otimes \godown$.
\end{lem}

\begin{proof}
	Lemmas \ref{lem:dualSemigroup} and \ref{lem:semigroupLieObject} tell us that $\left(\goup \otimes \godown, \begin{tikzpicture}[anchorbase]
	\draw[->, thick] (0, 0) to[out=90, in=270] (0.25, 0.5) -- (0.25, 0.75);
	\draw[<-, thick] (0.25, 0) -- (0.25, 0.2) arc(180:0:0.125) -- (0.5, 0);
	\draw[<-, thick] (0.75, 0) to[out=90, in=270] (0.5, 0.5) -- (0.5, 0.75);
	\end{tikzpicture}
	\ -\
	\begin{tikzpicture}[anchorbase]
	\draw[->, thick] (0.5, -0.75) -- (0.5, -0.5) to[out=90, in=270] (0, 0) to[out=90, in=270] (0.25, 0.5) -- (0.25, 0.75);
	\draw[<-, thick] (0.75, -0.75) -- (0.75, -0.5) to[out=90, in=270] (0.5, 0) arc(0:180:0.125) to[out=270, in=90] (0, -0.5) -- (0, -0.75);
	\draw[<-, thick] (0.25, -0.75) -- (0.25, -0.5) to[out=90, in=270] (0.75, 0) to[out=90, in=270] (0.5, 0.5) -- (0.5, 0.75);
	\end{tikzpicture} \right)$ is a Lie algebra, using \eqref{eq:delooping} on the cap in the negated term. A direct calculation shows that $\left(\goup,\ \begin{tikzpicture}[anchorbase]
	\draw[->, thick] (0, 0) -- (0, 0.75);
	\draw[->, thick] (0.6, 0) -- (0.6, 0.2) arc(0:180:0.2) -- (0.2, 0);
	\end{tikzpicture} \right)$ is a $(\goup \otimes \godown)$-module. Using Lemma \ref{lem:dualModule}, $\left(\godown,\ -\ \begin{tikzpicture}[anchorbase]
	\draw[->, thick] (0, 0) -- (0, 0.25) to[out=90, in=270] (0.25, 0.5) arc(0:180:0.125) to[out=270, in=90] (0.5, 0);
	\draw[<-, thick] (0.25, 0) -- (0.25, 0.25) to[out=90, in=270] (0.5, 0.5) -- (0.5, 0.75) arc(180:0:0.125) -- (0.75, 0.625) arc(180:360:0.125) -- (1, 1);
	\end{tikzpicture} \right) \underset{\eqref{eq:zigzag}}{\overset{\eqref{eq:delooping}}{=}} \left(\godown,\ -\ \begin{tikzpicture}[anchorbase]
	\draw[->, thick] (0, 0) -- (0, 0.25) arc(180:0:0.25) -- (0.5, 0);
	\draw[<-, thick] (0.25, 0) -- (0.25, 0.25) to[out=90, in=270] (0.5, 0.5) -- (0.5, 0.75);
	\end{tikzpicture} \right)$ is the dual of that module.
\end{proof}

\begin{eg}\label{ex:orientedBrauerLieObject}
	When working in $\CC = \kkvec$, a semigroup is precisely an associative algebra. In this context, Lemma \ref{lem:semigroupLieObject} describes the familiar construction of a Lie algebra from an associative algebra, with Lie bracket given by the commutator: $[a, b] = ab - ba$. Taking $\goup$ to be a finite-dimensional vector space $V$, the semigroup $\goup \otimes \godown$ is $V \otimes V^* \isom \End(V)$ (with the product being composition), and the associated Lie algebra is $\gl(V)$ (with the commutator Lie bracket). In this case, the module $\goup$ discussed in Lemma \ref{lem:orientedBrauerLieObject} is simply $V$ equipped with the natural $\End(V)$-action, and $\godown$ is the dual module $V^*$.
\end{eg}

If $\g$ is a Lie algebra in $\kkvec$ and $V$ is a $\g$-module, then $V \otimes (\kk[t] / (t - a)^k)$ is naturally a $\g$-current module for each $a \in \kk$ and $k \in \N$. The following construction generalizes this family of current modules.

\begin{prop}\label{prop:truncatedPolynomialModule}
	Suppose $\CC$ is additive, and let $k \in \N$. Let $L$ be a Lie algebra and $\green{V}$ a current $L$-module. Set $\red{W} = \green{V}^{\oplus k}$, and for each $n \in \N$, define $$\damapmultred{t^n} = \bm\binom{n}{0}\damapmult[2px]{t^n} & 0 & 0 & \dotsb & 0\\ 
	\binom{n}{1} \damapmult[4px]{t^{n - 1}} & \binom{n}{0}\damapmult[2px]{t^n} & 0 & \dotsb & 0\\
	\binom{n}{2} \damapmult[4px]{t^{n - 2}} & \binom{n}{1} \damapmult[4px]{t^{n - 1}} & \binom{n}{0} \damapmult[2px]{t^n} & \dotsb & 0\\
	\vdots & \vdots & \vdots & \ddots & \vdots \\
	\binom{n}{k - 1} \damapmult[4px]{t^{n - k + 1}} & \binom{n}{k - 2} \damapmult[4px]{t^{n - k + 2}} & \binom{n}{k - 3} \damapmult[4px]{t^{n - k + 3}} & \dotsb & \binom{n}{0} \damapmult[2px]{t^n} \nm.$$
	In other words, $\left[\damapmultred{t^n}\right]_{ij} = \binom{n}{i - j} \damapmult[3px]{t^{n - i + j}}$. (Here we adopt the convention that $\damapmult{t^n} = 0$ for $n < 0$ and $\binom{n}{k} = 0$ for $k < 0$.) Equipped with these action morphisms, $\red{W}$ is a current $L$-module.
\end{prop}

\begin{proof}
	Let $i, j, m, n \in \N$. For use in the calculation below, we compute:
	\begin{commeqn}
		\summ_{a = 1}^k \binom{m}{i - a} \binom{n}{a - j} &=\summ_{b = 1 - j}^{k - j} \binom{m}{i - j -b}\binom{n}{b} &\commentary{Switching the summation variable to $b := a - j$}\\
		&= \summ_{b = 0}^{k - j} \binom{m}{i - j - b}\binom{n}{b} &\commentary{Since $\binom{n}{b} = 0$ for all $b < 0$}\\
		&= \summ_{b = 0}^{i - j} \binom{m}{i - j - b}\binom{n}{b} &\commentary{Since $\binom{m}{i - j - b} = 0$ for all $b > i - j$}\\
		&= \binom{m + n}{i - j}. &\commentary{Using Vandermonde's identity}
	\end{commeqn}
	Call this identity $(\star)$. We then have:
	\begin{multline*}
		\left(\begin{tikzpicture}[anchorbase]
		\draw[bmod] (0.25, 0) -- (0.5, 0.25);
		\draw[bmod] (0, 0) -- (0, 0.25) -- (0.25, 0.5);
		\draw[rmod] (0.75, 0) -- (0.5, 0.25) -- (0.25, 0.5) -- (0.25, 0.75);
		\btoken{0.5, 0.25}{west}{t^n};
		\btoken{0.25, 0.5}{west}{t^m};
		\end{tikzpicture}\right)_{ij} 
		= \summ_{a = 1}^k \left(\damapmultred{t^m} \right)_{ia} \circ \left(\begin{tikzpicture}[anchorbase]
		\draw[bmod] (0, 0) -- (0, 0.5);
		\end{tikzpicture}\ \damapmultred{t^n} \right)_{aj}
		=
		\summ_{a = 1}^k \binom{m}{i - a}\binom{n}{a - j} \begin{tikzpicture}[anchorbase]
		\draw[bmod] (0.25, 0) -- (0.5, 0.25);
		\draw[bmod] (0, 0) -- (0, 0.25) -- (0.25, 0.5);
		\draw[gmod] (0.75, 0) -- (0.5, 0.25) -- (0.25, 0.5) -- (0.25, 0.75);
		\btoken{0.5, 0.25}{west}{t^{n - a + j}};
		\btoken{0.25, 0.5}{west}{t^{m - i + a}};
		\end{tikzpicture}
		\\
		\underset{(\star)}{\overset{\eqref{eq:lmod}}{=}} 
		\binom{m + n}{i - j}\left(\begin{tikzpicture}[anchorbase]
		\draw[bmod] (0.5, 0) -- (0.25, 0.25);
		\draw[bmod] (0, 0) -- (0.25, 0.25) -- (0.5, 0.5);
		\draw[gmod] (0.75, 0) -- (0.75, 0.25) -- (0.5, 0.5) -- (0.5, 0.75);
		\btoken[2px]{0.5, 0.5}{west}{t^{m + n - i + j}};
		\end{tikzpicture}
		\ +\ 
		\begin{tikzpicture}[anchorbase]
		\draw[bmod] (0, -0.25) -- (0.25, 0) -- (0.5, 0.25);
		\draw[bmod] (0.25, -0.25) -- (0, 0) -- (0, 0.25) -- (0.25, 0.5);
		\draw[gmod] (0.75, -0.25) -- (0.75, 0) -- (0.5, 0.25) -- (0.25, 0.5) -- (0.25, 0.75);
		\btoken{0.5, 0.25}{west}{t^{m - i + a}};
		\btoken[2px]{0.25, 0.5}{west}{t^{n - a + j}};
		\end{tikzpicture}\right)
		\\
		\overset{(\star)}{=} \binom{m + n}{i - j}\begin{tikzpicture}[anchorbase]
		\draw[bmod] (0.5, 0) -- (0.25, 0.25);
		\draw[bmod] (0, 0) -- (0.25, 0.25) -- (0.5, 0.5);
		\draw[gmod] (0.75, 0) -- (0.75, 0.25) -- (0.5, 0.5) -- (0.5, 0.75);
		\btoken[2px]{0.5, 0.5}{west}{t^{m + n - i + j}};
		\end{tikzpicture} + \summ_{a = 1}^k \binom{n}{i - a} \binom{m}{a - j}\begin{tikzpicture}[anchorbase]
		\draw[bmod] (0, -0.25) -- (0.25, 0) -- (0.5, 0.25);
		\draw[bmod] (0.25, -0.25) -- (0, 0) -- (0, 0.25) -- (0.25, 0.5);
		\draw[gmod] (0.75, -0.25) -- (0.75, 0) -- (0.5, 0.25) -- (0.25, 0.5) -- (0.25, 0.75);
		\btoken{0.5, 0.25}{west}{t^{m - i + a}};
		\btoken[2px]{0.25, 0.5}{west}{t^{n - a + j}};
		\end{tikzpicture}
		\\
		= \left(\begin{tikzpicture}[anchorbase]
		\draw[bmod] (0.5, 0) -- (0.25, 0.25);
		\draw[bmod] (0, 0) -- (0.25, 0.25) -- (0.5, 0.5);
		\draw[rmod] (0.75, 0) -- (0.75, 0.25) -- (0.5, 0.5) -- (0.5, 0.75);
		\btoken[2px]{0.5, 0.5}{west}{t^{m + n - i + j}};
		\end{tikzpicture}\right)_{ij} 
		\ +\ 
		\summ_{a = 1}^k \left(\damapmultred{t^n} \right)_{ia} \circ \left( \begin{tikzpicture}[anchorbase]
		\draw[bmod] (0.25, -0.25) -- (0, 0) -- (0, 0.5);
		\draw[bmod] (0, -0.25) -- (0.25, 0) -- (0.5, 0.25);
		\draw[rmod] (0.75, -0.25) -- (0.75, 0) -- (0.5, 0.25) -- (0.5, 0.5);
		\btoken{0.5, 0.25}{west}{t^m};	
		\end{tikzpicture} \right)_{aj}
		= \left(\begin{tikzpicture}[anchorbase]
		\draw[bmod] (0.5, 0) -- (0.25, 0.25);
		\draw[bmod] (0, 0) -- (0.25, 0.25) -- (0.5, 0.5);
		\draw[rmod] (0.75, 0) -- (0.75, 0.25) -- (0.5, 0.5) -- (0.5, 0.75);
		\btoken[2px]{0.5, 0.5}{west}{t^{m + n - i + j}};
		\end{tikzpicture}\right)_{ij} 
		\ +\
		\left(\begin{tikzpicture}[anchorbase]
		\draw[bmod] (0, -0.25) -- (0.25, 0) -- (0.5, 0.25);
		\draw[bmod] (0.25, -0.25) -- (0, 0) -- (0, 0.25) -- (0.25, 0.5);
		\draw[rmod] (0.75, -0.25) -- (0.75, 0) -- (0.5, 0.25) -- (0.25, 0.5) -- (0.25, 0.75);
		\btoken{0.5, 0.25}{west}{t^{m}};
		\btoken[2px]{0.25, 0.5}{west}{t^n};
		\end{tikzpicture} \right)_{ij}.
	\end{multline*}
	Hence $\begin{tikzpicture}[anchorbase]
	\draw[bmod] (0.25, 0) -- (0.5, 0.25);
	\draw[bmod] (0, 0) -- (0, 0.25) -- (0.25, 0.5);
	\draw[rmod] (0.75, 0) -- (0.5, 0.25) -- (0.25, 0.5) -- (0.25, 0.75);
	\btoken{0.5, 0.25}{west}{t^n};
	\btoken{0.25, 0.5}{west}{t^m};
	\end{tikzpicture} = \begin{tikzpicture}[anchorbase]
	\draw[bmod] (0.5, 0) -- (0.25, 0.25);
	\draw[bmod] (0, 0) -- (0.25, 0.25) -- (0.5, 0.5);
	\draw[rmod] (0.75, 0) -- (0.75, 0.25) -- (0.5, 0.5) -- (0.5, 0.75);
	\btoken[2px]{0.5, 0.5}{west}{t^{m + n - i + j}};
	\end{tikzpicture} + \begin{tikzpicture}[anchorbase]
	\draw[bmod] (0, -0.25) -- (0.25, 0) -- (0.5, 0.25);
	\draw[bmod] (0.25, -0.25) -- (0, 0) -- (0, 0.25) -- (0.25, 0.5);
	\draw[rmod] (0.75, -0.25) -- (0.75, 0) -- (0.5, 0.25) -- (0.25, 0.5) -- (0.25, 0.75);
	\btoken{0.5, 0.25}{west}{t^{m}};
	\btoken[2px]{0.25, 0.5}{west}{t^n};
	\end{tikzpicture}$, showing that $\red{W}$ satisfies \eqref{eq:lmod}.
\end{proof}

\begin{eg} \label{ex:truncatedPolynomialModule}
	Let $\g$ be a Lie algebra in $\kkvec$ and $\green{V}$ a $\g$-module. Fix some $a \in \kk$. Let $\red{W} = \green{V} \oplus \green{V}$ be the current $L$-module obtained by applying the construction in Proposition \ref{prop:truncatedPolynomialModule} to the evaluation module of $\green{V}$ at $a$. Explicitly, we have $\damapmultred{1} = \bm \dmodulemult & 0 \\ 0 & \dmodulemult \nm$ and $\damapmultred{t} = \bm a \dmodulemult & 0 \\ \dmodulemult & a\dmodulemult \nm$. Then $\red{W}$ is isomorphic to the current $\g$-module $\green{V} \otimes (\kk[t] / (t - a)^2) \isom (\green{V} \otimes 1) \oplus (\green{V} \otimes (t - a))$. More generally, for each $k \in \N$ the current $L$-module $\red{W} = \green{V}^{\oplus k}$ obtained by applying the construction in Proposition \ref{prop:truncatedPolynomialModule} to the evaluation module of $\green{V}$ at $a$ is isomorphic to
	$$\green{V} \otimes (\kk[t] / (t - a)^k) \isom (\green{V} \otimes 1) \oplus (\green{V} \otimes (t - a)) \oplus \dotsb \oplus (\green{V} \otimes (t - a)^{k - 1}).$$
\end{eg}

The following proposition describes a method to construct extensions between two current $L$-modules.

\begin{prop} \label{prop:currentExtensions}
	Suppose $\CC$ is additive, and let $L$ be a Lie algebra in $\CC$. Let $\green{V}$ and $\red{W}$ be current $L$-modules, and $a \in \kk$. We write $L_a$ for the evaluation module of the adjoint at $a$. Suppose that $\begin{tikzpicture}[anchorbase]
	\draw[bmod] (0, 0) -- (0.25, 0.25);
	\draw[gmod] (0.5, 0) -- (0.25, 0.25);
	\draw[rmod] (0.25, 0.25) -- (0.25, 0.5);
	\atoken{0.25, 0.25}{west}{};
	\end{tikzpicture}$ is a morphism of current $L$-modules, from $L_a \otimes \green{V}$ to $\red{W}$. Then $\blue{U} = \green{V} \oplus \red{W}$ becomes a current $L$-module when equipped with the action morphisms $\damapmultblue{t^n} = \bm \damapmult{t^n} & 0 \\ n a^{n - 1} \begin{tikzpicture}[anchorbase]
	\draw[bmod] (0, 0) -- (0.25, 0.25);
	\draw[gmod] (0.5, 0) -- (0.25, 0.25);
	\draw[rmod] (0.25, 0.25) -- (0.25, 0.5);
	\atoken{0.25, 0.25}{west}{};
	\end{tikzpicture} & \damapmultred{t^n} \nm$.
\end{prop}

\begin{proof}
	Since $\begin{tikzpicture}[anchorbase]
	\draw[bmod] (0, 0) -- (0.25, 0.25);
	\draw[gmod] (0.5, 0) -- (0.25, 0.25);
	\draw[rmod] (0.25, 0.25) -- (0.25, 0.5);
	\atoken{0.25, 0.25}{west}{};
	\end{tikzpicture}$ is a morphism of current $L$-modules, for all $n \in \N$ we have:
	\begin{equation} \label{eq:extensionMorphismIdentity}
	\begin{tikzpicture}[anchorbase]
	\draw[bmod] (0, 0) -- (0, 0.25) -- (0.25, 0.5);
	\draw[bmod] (0.25, 0) -- (0.5, 0.25);
	\draw[gmod] (0.75, 0) -- (0.5, 0.25);
	\draw[rmod] (0.5, 0.25) -- (0.25, 0.5) -- (0.25, 0.75);
	\adot{(0.5, 0.25)};
	\btoken{0.25, 0.5}{west}{t^n};
	\end{tikzpicture}
	\ =\
	a^n \begin{tikzpicture}[anchorbase]
	\draw[bmod] (0, 0) -- (0.25, 0.25);
	\draw[bmod] (0.5, 0) -- (0.25, 0.25) -- (0.5, 0.5);
	\draw[gmod] (0.75, 0) -- (0.75, 0.25) -- (0.5, 0.5);
	\draw[rmod] (0.5, 0.5) -- (0.5, 0.75);
	\adot{(0.5, 0.5)};
	\end{tikzpicture}
	\ +\
	\begin{tikzpicture}[anchorbase]
	\draw[bmod] (0, 0) -- (0.5, 0.5);
	\draw[bmod] (0.25, 0) -- (0, 0.25) -- (0, 0.5) -- (0.25, 0.75);
	\draw[gmod] (0.75, 0) -- (0.75, 0.25) -- (0.5, 0.5) -- (0.25, 0.75);
	\draw[rmod] (0.25, 0.75) -- (0.25, 1);
	\btoken{0.5, 0.5}{west}{t^n};
	\adot{(0.25, 0.75)};
	\end{tikzpicture}.
	\end{equation} 
	Let $m, n \in \N$. Then:
	\begin{commeqn}
		\begin{tikzpicture}[anchorbase]
		\draw[bmod] (0.25, 0) -- (0.5, 0.25);
		\draw[bmod] (0, 0) -- (0, 0.25) -- (0.25, 0.5);
		\draw[blmod] (0.75, 0) -- (0.5, 0.25) -- (0.25, 0.5) -- (0.25, 0.75);
		\btoken{0.5, 0.25}{west}{t^m};
		\btoken{0.25, 0.5}{west}{t^n};
		\end{tikzpicture} &= \bm \damapmult{t^n} & 0 \\ n a^{n - 1} \begin{tikzpicture}[anchorbase]
		\draw[bmod] (0, 0) -- (0.25, 0.25);
		\draw[gmod] (0.5, 0) -- (0.25, 0.25);
		\draw[rmod] (0.25, 0.25) -- (0.25, 0.5);
		\atoken{0.25, 0.25}{west}{};
		\end{tikzpicture} & \damapmultred{t^n} \nm \bm \begin{tikzpicture}[anchorbase]
		\draw[bmod] (-0.25, 0) -- (-0.25, 0.5);
		\draw[bmod] (0, 0) -- (0.25, 0.25);
		\draw[gmod] (0.5, 0) -- (0.25, 0.25) -- (0.25, 0.5);
		\btoken{0.25, 0.25}{west}{t^m};
		\end{tikzpicture} & 0 \\ m a^{m - 1}\ \begin{tikzpicture}[anchorbase]
		\draw[bmod] (-0.25, 0) -- (-0.25, 0.5);
		\draw[bmod] (0, 0) -- (0.25, 0.25);
		\draw[gmod] (0.5, 0) -- (0.25, 0.25);
		\draw[rmod] (0.25, 0.25) -- (0.25, 0.5);
		\atoken{0.25, 0.25}{west}{};
		\end{tikzpicture} & \begin{tikzpicture}[anchorbase]
		\draw[bmod] (-0.25, 0) -- (-0.25, 0.5);
		\draw[bmod] (0, 0) -- (0.25, 0.25);
		\draw[rmod] (0.5, 0) -- (0.25, 0.25) -- (0.25, 0.5);
		\btoken{0.25, 0.25}{west}{t^m};
		\end{tikzpicture} \nm\\
		&= \bm \begin{tikzpicture}[anchorbase]
		\draw[bmod] (0.25, 0) -- (0.5, 0.25);
		\draw[bmod] (0, 0) -- (0, 0.25) -- (0.25, 0.5);
		\draw[gmod] (0.75, 0) -- (0.5, 0.25) -- (0.25, 0.5) -- (0.25, 0.75);
		\btoken{0.5, 0.25}{west}{t^m};
		\btoken{0.25, 0.5}{west}{t^n};
		\end{tikzpicture} & 0 \\ na^{n - 1} \begin{tikzpicture}[anchorbase]
		\draw[bmod] (0, 0) -- (0, 0.25) -- (0.25, 0.5);
		\draw[bmod] (0.25, 0) -- (0.5, 0.25);
		\draw[gmod] (0.75, 0) -- (0.5, 0.25) -- (0.25, 0.5);
		\draw[rmod] (0.25, 0.5) -- (0.25, 0.75);
		\btoken{0.5, 0.25}{west}{t^m};
		\adot{(0.25, 0.5)};
		\end{tikzpicture} + ma^{m - 1} \begin{tikzpicture}[anchorbase]
		\draw[bmod] (0, 0) -- (0, 0.25) -- (0.25, 0.5);
		\draw[bmod] (0.25, 0) -- (0.5, 0.25);
		\draw[gmod] (0.75, 0) -- (0.5, 0.25);
		\draw[rmod] (0.5, 0.25) -- (0.25, 0.5) -- (0.25, 0.75);
		\btoken{0.25, 0.5}{west}{t^n};
		\adot{(0.5, 0.25)};
		\end{tikzpicture} &\ \begin{tikzpicture}[anchorbase]
		\draw[bmod] (0.25, 0) -- (0.5, 0.25);
		\draw[bmod] (0, 0) -- (0, 0.25) -- (0.25, 0.5);
		\draw[rmod] (0.75, 0) -- (0.5, 0.25) -- (0.25, 0.5) -- (0.25, 0.75);
		\btoken{0.5, 0.25}{west}{t^m};
		\btoken{0.25, 0.5}{west}{t^n};
		\end{tikzpicture}\ \nm,\\
		\begin{tikzpicture}[anchorbase]
		\draw[bmod] (0.5, 0) -- (0.25, 0.25);
		\draw[bmod] (0, 0) -- (0.25, 0.25) -- (0.5, 0.5);
		\draw[blmod] (0.75, 0) -- (0.75, 0.25) -- (0.5, 0.5) -- (0.5, 0.75);
		\btoken{0.5, 0.5}{west}{t^{m + n}};
		\end{tikzpicture} &= \bm \begin{tikzpicture}[anchorbase]
		\draw[bmod] (0, 0) -- (0.25, 0.25);
		\draw[bmod] (0.5, 0) -- (0.25, 0.25) -- (0.5, 0.5);
		\draw[gmod] (0.75, 0) -- (0.75, 0.25) -- (0.5, 0.5) -- (0.5, 0.75);
		\btoken{0.5, 0.5}{west}{t^{m + n}};
		\end{tikzpicture} & 0 \\
		(m + n) a^{m + n - 1}\begin{tikzpicture}[anchorbase]
		\draw[bmod] (0, 0) -- (0.25, 0.25);
		\draw[bmod] (0.5, 0) -- (0.25, 0.25) -- (0.5, 0.5);
		\draw[gmod] (0.75, 0) -- (0.75, 0.25) -- (0.5, 0.5);
		\draw[rmod] (0.5, 0.5) -- (0.5, 0.75);
		\atoken{0.5, 0.5}{west}{};
		\end{tikzpicture} &\ \begin{tikzpicture}[anchorbase]
		\draw[bmod] (0, 0) -- (0.25, 0.25);
		\draw[bmod] (0.5, 0) -- (0.25, 0.25) -- (0.5, 0.5);
		\draw[rmod] (0.75, 0) -- (0.75, 0.25) -- (0.5, 0.5) -- (0.5, 0.75);
		\btoken{0.5, 0.5}{west}{t^{m + n}};
		\end{tikzpicture}\ \nm,\\
		\begin{tikzpicture}[anchorbase]
		\draw[bmod] (0, -0.25) -- (0.25, 0) -- (0.5, 0.25);
		\draw[bmod] (0.25, -0.25) -- (0, 0) -- (0, 0.25) -- (0.25, 0.5);
		\draw[blmod] (0.75, -0.25) -- (0.75, 0) -- (0.5, 0.25) -- (0.25, 0.5) -- (0.25, 0.75);
		\btoken{0.5, 0.25}{west}{t^n};
		\btoken{0.25, 0.5}{west}{t^m};
		\end{tikzpicture} &= \bm \begin{tikzpicture}[anchorbase]
		\draw[bmod] (0, -0.25) -- (0.25, 0) -- (0.5, 0.25);
		\draw[bmod] (0.25, -0.25) -- (0, 0) -- (0, 0.25) -- (0.25, 0.5);
		\draw[gmod] (0.75, -0.25) -- (0.75, 0) -- (0.5, 0.25) -- (0.25, 0.5) -- (0.25, 0.75);
		\btoken{0.5, 0.25}{west}{t^n};
		\btoken{0.25, 0.5}{west}{t^m};
		\end{tikzpicture} & 0 \\ ma^{m - 1} \begin{tikzpicture}[anchorbase]
		\draw[bmod] (0.25, -0.25) -- (0, 0) -- (0, 0.25) -- (0.25, 0.5);
		\draw[bmod] (0, -0.25) -- (0.25, 0) -- (0.5, 0.25);
		\draw[gmod] (0.75, -0.25) -- (0.75, 0) -- (0.5, 0.25) -- (0.25, 0.5);
		\draw[rmod] (0.25, 0.5) -- (0.25, 0.75);
		\btoken{0.5, 0.25}{west}{t^n};
		\adot{(0.25, 0.5)};
		\end{tikzpicture} + na^{n - 1} \begin{tikzpicture}[anchorbase]
		\draw[bmod] (0.25, -0.25) -- (0, 0) -- (0, 0.25) -- (0.25, 0.5);
		\draw[bmod] (0, -0.25) -- (0.25, 0) -- (0.5, 0.25);
		\draw[gmod] (0.75, -0.25) -- (0.75, 0) -- (0.5, 0.25);
		\draw[rmod] (0.5, 0.25) -- (0.25, 0.5) -- (0.25, 0.75);
		\btoken{0.25, 0.5}{west}{t^m};
		\adot{(0.5, 0.25)};
		\end{tikzpicture} &\ \begin{tikzpicture}[anchorbase]
		\draw[bmod] (0, -0.25) -- (0.25, 0) -- (0.5, 0.25);
		\draw[bmod] (0.25, -0.25) -- (0, 0) -- (0, 0.25) -- (0.25, 0.5);
		\draw[rmod] (0.75, -0.25) -- (0.75, 0) -- (0.5, 0.25) -- (0.25, 0.5) -- (0.25, 0.75);
		\btoken{0.5, 0.25}{west}{t^n};
		\btoken{0.25, 0.5}{west}{t^m};
		\end{tikzpicture}\ \nm,
	\end{commeqn}
	and we want to show $\begin{tikzpicture}[anchorbase]
	\draw[bmod] (0.25, 0) -- (0.5, 0.25);
	\draw[bmod] (0, 0) -- (0, 0.25) -- (0.25, 0.5);
	\draw[blmod] (0.75, 0) -- (0.5, 0.25) -- (0.25, 0.5) -- (0.25, 0.75);
	\btoken{0.5, 0.25}{west}{t^m};
	\btoken{0.25, 0.5}{west}{t^n};
	\end{tikzpicture} = \begin{tikzpicture}[anchorbase]
	\draw[bmod] (0.5, 0) -- (0.25, 0.25);
	\draw[bmod] (0, 0) -- (0.25, 0.25) -- (0.5, 0.5);
	\draw[blmod] (0.75, 0) -- (0.75, 0.25) -- (0.5, 0.5) -- (0.5, 0.75);
	\btoken{0.5, 0.5}{west}{t^{m + n}};
	\end{tikzpicture} + \begin{tikzpicture}[anchorbase]
	\draw[bmod] (0, -0.25) -- (0.25, 0) -- (0.5, 0.25);
	\draw[bmod] (0.25, -0.25) -- (0, 0) -- (0, 0.25) -- (0.25, 0.5);
	\draw[blmod] (0.75, -0.25) -- (0.75, 0) -- (0.5, 0.25) -- (0.25, 0.5) -- (0.25, 0.75);
	\btoken{0.5, 0.25}{west}{t^n};
	\btoken{0.25, 0.5}{west}{t^m};
	\end{tikzpicture}$. The desired equality for the entries on the main diagonal follows immediately from the fact that $\green{V}$ and $\red{W}$ are current $L$-modules. The equality for the bottom left entry follows after noting that
	\begin{commeqn}
		(m + n) a^{m + n - 1}\begin{tikzpicture}[anchorbase]
		\draw[bmod] (0, 0) -- (0.25, 0.25);
		\draw[bmod] (0.5, 0) -- (0.25, 0.25) -- (0.5, 0.5);
		\draw[gmod] (0.75, 0) -- (0.75, 0.25) -- (0.5, 0.5);
		\draw[rmod] (0.5, 0.5) -- (0.5, 0.75);
		\atoken{0.5, 0.5}{west}{};
		\end{tikzpicture} &= m a^{m + n - 1}\begin{tikzpicture}[anchorbase]
		\draw[bmod] (0, 0) -- (0.25, 0.25);
		\draw[bmod] (0.5, 0) -- (0.25, 0.25) -- (0.5, 0.5);
		\draw[gmod] (0.75, 0) -- (0.75, 0.25) -- (0.5, 0.5);
		\draw[rmod] (0.5, 0.5) -- (0.5, 0.75);
		\atoken{0.5, 0.5}{west}{};
		\end{tikzpicture} - n a^{m + n - 1}\begin{tikzpicture}[anchorbase]
		\draw[bmod] (0.5, -0.5) -- (0, 0) -- (0.25, 0.25);
		\draw[bmod] (0, -0.5) -- (0.5, 0) -- (0.25, 0.25) -- (0.5, 0.5);
		\draw[gmod] (0.75, -0.5) -- (0.75, 0) -- (0.75, 0.25) -- (0.5, 0.5);
		\draw[rmod] (0.5, 0.5) -- (0.5, 0.75);
		\atoken{0.5, 0.5}{west}{};
		\end{tikzpicture} &\commentary{Using \eqref{eq:skew}}\\
		&= ma^{m - 1}\left( \begin{tikzpicture}[anchorbase]
		\draw[bmod] (0, 0) -- (0, 0.25) -- (0.25, 0.5);
		\draw[bmod] (0.25, 0) -- (0.5, 0.25);
		\draw[gmod] (0.75, 0) -- (0.5, 0.25);
		\draw[rmod] (0.5, 0.25) -- (0.25, 0.5) -- (0.25, 0.75);
		\adot{(0.5, 0.25)};
		\btoken{0.25, 0.5}{west}{t^n};
		\end{tikzpicture} 
		- 
		\begin{tikzpicture}[anchorbase]
		\draw[bmod] (0, 0) -- (0.5, 0.5);
		\draw[bmod] (0.25, 0) -- (0, 0.25) -- (0, 0.5) -- (0.25, 0.75);
		\draw[gmod] (0.75, 0) -- (0.75, 0.25) -- (0.5, 0.5) -- (0.25, 0.75);
		\draw[rmod] (0.25, 0.75) -- (0.25, 1);
		\btoken{0.5, 0.5}{west}{t^n};
		\adot{(0.25, 0.75)};
		\end{tikzpicture} \right)
		-
		na^{n - 1}\left(\begin{tikzpicture}[anchorbase]
		\draw[bmod] (0.25, -0.25) -- (0, 0) -- (0, 0.25) -- (0.25, 0.5);
		\draw[bmod] (0, -0.25) -- (0.25, 0) -- (0.5, 0.25);
		\draw[gmod] (0.75, -0.25) -- (0.75, 0) -- (0.5, 0.25);
		\draw[rmod] (0.5, 0.25) -- (0.25, 0.5) -- (0.25, 0.75);
		\adot{(0.5, 0.25)};
		\btoken{0.25, 0.5}{west}{t^m};
		\end{tikzpicture} 
		- 
		\begin{tikzpicture}[anchorbase]
		\draw[bmod] (0.25, -0.25) -- (0, 0) -- (0.5, 0.5);
		\draw[bmod] (0, -0.25) -- (0.25, 0) -- (0, 0.25) -- (0, 0.5) -- (0.25, 0.75);
		\draw[gmod] (0.75, -0.25) -- (0.75, 0.25) -- (0.5, 0.5) -- (0.25, 0.75);
		\draw[rmod] (0.25, 0.75) -- (0.25, 1);
		\btoken{0.5, 0.5}{west}{t^m};
		\adot{(0.25, 0.75)};
		\end{tikzpicture}\right) &\commentary{Using \eqref{eq:extensionMorphismIdentity}}\\
		&= ma^{m - 1}\left( \begin{tikzpicture}[anchorbase]
		\draw[bmod] (0, 0) -- (0, 0.25) -- (0.25, 0.5);
		\draw[bmod] (0.25, 0) -- (0.5, 0.25);
		\draw[gmod] (0.75, 0) -- (0.5, 0.25);
		\draw[rmod] (0.5, 0.25) -- (0.25, 0.5) -- (0.25, 0.75);
		\adot{(0.5, 0.25)};
		\btoken{0.25, 0.5}{west}{t^n};
		\end{tikzpicture} 
		- 
		\begin{tikzpicture}[anchorbase]
		\draw[bmod] (0, 0) -- (0.5, 0.5);
		\draw[bmod] (0.25, 0) -- (0, 0.25) -- (0, 0.5) -- (0.25, 0.75);
		\draw[gmod] (0.75, 0) -- (0.75, 0.25) -- (0.5, 0.5) -- (0.25, 0.75);
		\draw[rmod] (0.25, 0.75) -- (0.25, 1);
		\btoken{0.5, 0.5}{west}{t^n};
		\adot{(0.25, 0.75)};
		\end{tikzpicture} \right)
		-
		na^{n - 1}\left(\begin{tikzpicture}[anchorbase]
		\draw[bmod] (0.25, -0.25) -- (0, 0) -- (0, 0.25) -- (0.25, 0.5);
		\draw[bmod] (0, -0.25) -- (0.25, 0) -- (0.5, 0.25);
		\draw[gmod] (0.75, -0.25) -- (0.75, 0) -- (0.5, 0.25);
		\draw[rmod] (0.5, 0.25) -- (0.25, 0.5) -- (0.25, 0.75);
		\adot{(0.5, 0.25)};
		\btoken{0.25, 0.5}{west}{t^m};
		\end{tikzpicture} 
		- 
		\begin{tikzpicture}[anchorbase]
		\draw[bmod] (0.25, 0.25) -- (0.5, 0.5);
		\draw[bmod] (0, 0.25) -- (0, 0.5) -- (0.25, 0.75);
		\draw[gmod] (0.75, 0.25) -- (0.5, 0.5) -- (0.25, 0.75);
		\draw[rmod] (0.25, 0.75) -- (0.25, 1);
		\btoken{0.5, 0.5}{west}{t^m};
		\adot{(0.25, 0.75)};
		\end{tikzpicture}\right). &\commentary{Using \eqref{eq:crossingIdentities}}
	\end{commeqn}
\end{proof}

\begin{egs} \label{ex:currentExtensions}
	Using the notation from Proposition \ref{prop:currentExtensions}, picking $\begin{tikzpicture}[anchorbase]
	\draw[bmod] (0, 0) -- (0.25, 0.25);
	\draw[gmod] (0.5, 0) -- (0.25, 0.25);
	\draw[rmod] (0.25, 0.25) -- (0.25, 0.5);
	\atoken{0.25, 0.25}{west}{};
	\end{tikzpicture} = 0$ yields the trivial extension of $\green{V}$ by $\red{W}$, i.e.\ $\blue{U} = \green{V} \oplus \red{W}$ not only as $L$-modules, but also as current $L$-modules.
	
	Taking $\green{V} = \red{W}$ to be an evaluation module at $a \in \kk$ and setting $\begin{tikzpicture}[anchorbase]
	\draw[bmod] (0, 0) -- (0.25, 0.25);
	\draw[gmod] (0.5, 0) -- (0.25, 0.25);
	\draw[rmod] (0.25, 0.25) -- (0.25, 0.5);
	\atoken{0.25, 0.25}{west}{};
	\end{tikzpicture} = \dmodulemult \colon L_a \otimes \green{V} \to \green{V}$, the resulting current $L$-module is the same as the one obtained by applying Proposition \ref{prop:truncatedPolynomialModule} to $\green{V}$ with $k = 2$. As mentioned in Example \ref{ex:truncatedPolynomialModule}, in the case $\CC = \kkvec$ this module is isomorphic to the truncated polynomial module $\green{V} \otimes (\kk[t] / (t - a)^2)$.
	
	Taking $\green{V} = \one_0$ (the trivial module evaluated at 0), $\red{W} = L_a$ for some $a \in \kk$, and $\begin{tikzpicture}[anchorbase]
	\draw[bmod] (0, 0) -- (0.25, 0.25);
	\draw[gmod] (0.5, 0) -- (0.25, 0.25);
	\draw[rmod] (0.25, 0.25) -- (0.25, 0.5);
	\atoken{0.25, 0.25}{west}{};
	\end{tikzpicture}$ to be the unitor map $L \otimes \one \to L$, one obtains an extension of $\one_0$ by $L_a$. In the case $\CC = \kkvec$, $a = 0$, this is the current module discussed at the end of \cite[\S2.6]{chariGreenstein}.
\end{egs}

\section{Lie algebras in diagrammatic categories} \label{s:diagrammaticCategories}

In this section, we will apply the content from the previous section to several diagrammatic categories: the oriented and unoriented Brauer categories $\OB$ and $\B$, their Frobenius algebra generalizations $\OB(A)$ and $\B(A)$, and the diagrammatic categories for the exceptional types $F_4$ and $G_2$. We will describe how to construct objects corresponding to the Lie algebra associated to each category (for instance, $\gl_n$ for $\OB$), its natural module, and the dual to that module. Additionally, we will use $\OB$ to define a candidate interpolating category $\Curr(\OB)$ for current $\gl_n$-modules, and outline what would need to be shown to prove that this category does indeed interpolate the categories $(\gl_n \otimes \kk[t])\dashmod$.

The oriented Brauer category $\OB$ was first defined and studied by Brundan, Comes, Nash, and Reynolds in \cite{Brundan_2017}. Its endomorphism algebras are the \emph{walled Brauer algebras} introduced independently by Turaev and Koike in \cite{turaevWalledBrauerAlgebras} and \cite{koikeWalledBrauerAlgebras}, respectively. As outlined below, $\OB$ is an interpolating category for the categories of $\gl_n$-modules, $n \in \N$. Note that the following definition of $\OB$ is not exactly the same as the one given in \cite{Brundan_2017}, and instead matches \cite[Def.~4.1]{orientedFrobeniusBrauer} (taking $A = \kk$).

\begin{defin} \label{def:orientedBrauerCategory}
	The \emph{oriented Brauer category}, denoted $\OB$, is the strict $\kk$-linear monoidal category with two generating objects, written as $\goup$ and $\godown$, generating morphisms $\dcrossup, \dcupdownup, \dcupupdown, \dcapupdown$ and $\dcapdownup,$ and the following relations:
	\begin{equation} \label{rel:orientedBasics}
	%braiding
	\begin{tikzpicture}[anchorbase]
	\draw[->, thick] (0, 0) -- (0.8, 0.8);
	\draw[->, thick] (0.4, 0) .. controls (0, 0.4) .. (0.4, 0.8);
	\draw[->, thick] (0.8, 0) -- (0, 0.8);
	\end{tikzpicture}
	\ =\
	\begin{tikzpicture}[anchorbase]
	\draw[->, thick] (0, 0) -- (0.8, 0.8);
	\draw[->, thick] (0.4, 0) .. controls (0.8, 0.4) .. (0.4, 0.8);
	\draw[->, thick] (0.8, 0) -- (0, 0.8);
	\end{tikzpicture}\ , \quad
	%double crossing
	\begin{tikzpicture}[anchorbase]
	\draw[->, thick] (0, 0) -- (0, 0.2) to[out=90,in=270] (0.4, 0.6) to[out=90,in=270] (0, 1) -- (0, 1.2);
	\draw[->, thick] (0.4, 0) -- (0.4, 0.2) to[out=90,in=270] (0, 0.6) to[out=90,in=270] (0.4, 1) -- (0.4, 1.2);
	\end{tikzpicture}
	\ =\
	\begin{tikzpicture}[anchorbase]
	\draw[->, thick] (0, 0) -- (0, 1.2);
	\draw[->, thick] (0.4, 0) -- (0.4, 1.2);
	\end{tikzpicture}\ , \quad
	\begin{tikzpicture}[anchorbase]
	\draw[->, thick] (0, 0) -- (0, 0.2) to[out=90,in=270] (0.4, 0.6) to[out=90,in=270] (0, 1) -- (0, 1.2);
	\draw[<-, thick] (0.4, 0) -- (0.4, 0.2) to[out=90,in=270] (0, 0.6) to[out=90,in=270] (0.4, 1) -- (0.4, 1.2);
	\end{tikzpicture}
	\ =\
	\begin{tikzpicture}[anchorbase]
	\draw[->, thick] (0, 0) -- (0, 1.2);
	\draw[<-, thick] (0.4, 0) -- (0.4, 1.2);
	\end{tikzpicture}\ , \quad
	\begin{tikzpicture}[anchorbase]
	\draw[<-, thick] (0, 0) -- (0, 0.2) to[out=90,in=270] (0.4, 0.6) to[out=90,in=270] (0, 1) -- (0, 1.2);
	\draw[->, thick] (0.4, 0) -- (0.4, 0.2) to[out=90,in=270] (0, 0.6) to[out=90,in=270] (0.4, 1) -- (0.4, 1.2);
	\end{tikzpicture}
	\ =\
	\begin{tikzpicture}[anchorbase]
	\draw[<-, thick] (0, 0) -- (0, 1.2);
	\draw[->, thick] (0.4, 0) -- (0.4, 1.2);
	\end{tikzpicture}\ , \quad
	\end{equation}
	\begin{equation} \label{rel:OBZigzagCurl}
	%zigzag
	\begin{tikzpicture}[anchorbase]
	\draw[->, thick] (0,0) -- (0,0.6) arc(180:0:0.2) -- (0.4,0.4) arc(180:360:0.2) -- (0.8,1);
	\end{tikzpicture}
	\ =\
	\begin{tikzpicture}[anchorbase]
	\draw[->, thick] (0,0) -- (0,1);
	\end{tikzpicture}\ , \quad
	\begin{tikzpicture}[anchorbase]
	\draw[->, thick] (0, 1) -- (0, 0.4) arc(180:360:0.2) -- (0.4, 0.6) arc(180:0:0.2) -- (0.8, 0);
	\end{tikzpicture}
	\ =\
	\begin{tikzpicture}[anchorbase]
	\draw[<-, thick] (0,0) -- (0,1);
	\end{tikzpicture}\ ,
	%curl
	\begin{tikzpicture}[centerzero]
	\draw[->, thick] (0,-0.4) to[out=up,in=180] (0.25,0.15) to[out=0,in=up] (0.4,0) to[out=down,in=0] (0.25,-0.15) to[out=180,in=down] (0,0.4);
	\end{tikzpicture}\ = \
	\begin{tikzpicture}[anchorbase]
	\draw[->, thick] (0, 0) -- (0, 0.6);
	\end{tikzpicture}\ = \
	\begin{tikzpicture}[centerzero]
	\draw[->, thick] (0,-0.4) to[out=up,in=0] (-0.25,0.15) to[out=180,in=up] (-0.4,0) to[out=down,in=180] (-0.25,-0.15) to[out=0,in=down] (0,0.4);
	\end{tikzpicture}\ .
	\end{equation}
	The sideways crossings appearing above are defined by $\dcrossright = \begin{tikzpicture}[anchorbase]
	\draw[->, thick] (0, 1) -- (0, 0.4) arc(180:360:0.2) -- (0.4, 0.6) arc(180:0:0.2) -- (0.8, 0);
	\draw[->, thick] (0.6, 0) to[out=90, in=270] (0.2, 1);
	\end{tikzpicture}$ and $\dcrossleft = \begin{tikzpicture}[anchorbase]
	\draw[<-, thick] (0,0) -- (0,0.6) arc(180:0:0.2) -- (0.4,0.4) arc(180:360:0.2) -- (0.8,1);
	\draw[->, thick] (0.2, 0) to[out=90, in=270] (0.6, 1);
	\end{tikzpicture}$\ .
	
	For each $\delta \in \kk$, we define the \emph{specialized oriented Brauer category} $\OB(\delta)$ to be the category obtained from $\OB$ by imposing one additional relation:
	\begin{equation}
	\begin{tikzpicture}[anchorbase]
	\draw[<-, thick] (0, 0) arc(0:360:0.25);
	\end{tikzpicture} = \delta.
	\end{equation}
\end{defin}

\begin{lem} \label{lem:orientedBrauerFree}
	The oriented Brauer category is the free $\kk$-linear symmetric monoidal category generated by an object $\goup$ and its dual $\godown$. That is, if $\CC$ is a $\kk$-linear symmetric monoidal category and $X \in \ob(\CC)$ is an object with dual $X^* \in \ob(\CC)$, then there is a unique monoidal functor $I \colon \OB \to \CC$ such that 
	\begin{equation} \label{eq:orientedBrauerFree}
	I(\goup) = X, \quad\quad I(\godown) = X^*, \quad\quad I\left(\dcrossup\right) = \begin{tikzpicture}[anchorbase]
	\draw[bmod] (0, 0) -- (0.5, 0.5);
	\draw[bmod] (0.5, 0) -- (0, 0.5);
	\node at (0, -0.25) {$X$};
	\node at (0.5, -0.25) {$X$};
	\node at (0, 0.75) {$X$};
	\node at (0.5, 0.75) {$X$};
	\end{tikzpicture}, \quad \text{and} \quad I\left(\dcapupdown\right) = \begin{tikzpicture}[anchorbase]
	\draw[bmod] (0, 0) -- (0, 0.2) arc(180:0:0.25) -- (0.5, 0);
	\node at (0, -0.25) {$X$};
	\node at (0.5, -0.25) {$X^*$};
	\end{tikzpicture}.
	\end{equation}
	This functor also satisfies 
	\begin{equation}\label{eq:orientedBrauerFree2}
	I\left(\dcapdownup \right) = \begin{tikzpicture}[anchorbase]
	\draw[bmod] (0, 0) -- (0, 0.2) arc(180:0:0.25) -- (0.5, 0);
	\node at (0, -0.25) {$X^*$};
	\node at (0.5, -0.25) {$X$};
	\end{tikzpicture}, \quad\quad I\left(\dcupupdown\right) = \begin{tikzpicture}[anchorbase, yscale=-1]
	\draw[bmod] (0, 0) -- (0, 0.2) arc(180:0:0.25) -- (0.5, 0);
	\node at (0, -0.25) {$X$};
	\node at (0.5, -0.25) {$X^*$};
	\end{tikzpicture}, \quad \text{and} \quad I\left(\dcupdownup \right) = \begin{tikzpicture}[anchorbase, xscale=-1, yscale=-1]
	\draw[bmod] (0, 0) -- (0, 0.2) arc(180:0:0.25) -- (0.5, 0);
	\node at (0, -0.25) {$X$};
	\node at (0.5, -0.25) {$X^*$};
	\end{tikzpicture}.
	\end{equation} We call such a functor an \emph{(oriented) incarnation functor}.
\end{lem}

\begin{proof}
	The existence of a functor $I$ satisfying \eqref{eq:orientedBrauerFree} is proved in \cite[Cor.~3.1]{Brundan_2017}. Uniqueness and the relations in \eqref{eq:orientedBrauerFree2} follow straightforwardly from the defining relations of $\OB$, the assumption that the caps and cups for $X$ and $X^*$ satisfy \eqref{eq:delooping} (see the comment below those identities), and the fact that $I$ is a symmetric monoidal functor. 
\end{proof}

For each $n \in \N$, we write $I_n$ for the oriented incarnation functor obtained by picking $\CC$ to be $\gl_n\dashmod$ and $X$ to be the natural $\gl_n$-module. We have $I_n \left(\begin{tikzpicture}[anchorbase]
\draw[<-, thick] (0, 0) arc(0:360:0.25);
\end{tikzpicture} \right) = \dim(X) = n$, and hence we may regard $I_n$ as a functor from $\OB(n)$ to $\gl_n\dashmod$. As shown in \cite[\S 8.3]{comesWilson}, each $I_n$ is full (note that the category $\underline{\text{Rep}}(GL_{m - n})$ discussed in that paper is isomorphic to $\OB(m - n)$, and their functor $F_{n \mid 0}$ corresponds to what we have called $I_n$). Thus $\OB$ is an interpolating category for the family of categories $\{\gl_n\dashmod : n \in \N \}$. In the following we use $\OB$ to construct a candidate interpolating category, which we denote $\Curr(\OB)$, for the family of current $\gl_n$-module categories, i.e.\ $\{(\gl_n \otimes \kk[t])\dashmod : n \in \N \}$. Most of the discussion applies to $(\gl_n \otimes A)$-modules more generally, but we focus on the case $A = \kk[t]$ for concreteness.

\begin{rem} \label{rem:tensorOmission}
	When working with diagrammatically-defined categories like $\OB$, it is conventional to omit the symbol $\otimes$ in tensor products of objects. For instance, we write $\goup \godown$ instead of $\goup \otimes \godown$ in what follows.
\end{rem}

Throughout the rest of this subsection, let $L$ denote the Lie algebra $\left(\goup \godown, \begin{tikzpicture}[anchorbase]
\draw[->, thick] (0, 0) to[out=90, in=270] (0.25, 0.5) -- (0.25, 0.75);
\draw[<-, thick] (0.25, 0) -- (0.25, 0.2) arc(180:0:0.125) -- (0.5, 0);
\draw[<-, thick] (0.75, 0) to[out=90, in=270] (0.5, 0.5) -- (0.5, 0.75);
\end{tikzpicture}
\ -\
\begin{tikzpicture}[anchorbase]
\draw[->, thick] (0.5, -0.75) -- (0.5, -0.5) to[out=90, in=270] (0, 0) to[out=90, in=270] (0.25, 0.5) -- (0.25, 0.75);
\draw[<-, thick] (0.75, -0.75) -- (0.75, -0.5) to[out=90, in=270] (0.5, 0) arc(0:180:0.125) to[out=270, in=90] (0, -0.5) -- (0, -0.75);
\draw[<-, thick] (0.25, -0.75) -- (0.25, -0.5) to[out=90, in=270] (0.75, 0) to[out=90, in=270] (0.5, 0.5) -- (0.5, 0.75);
\end{tikzpicture} \right)$ in $\OB$ (see Lemma \ref{lem:orientedBrauerLieObject}). For each $n \in \N$, $I_n(L)$ is isomorphic to $\gl_n$, equipped with the usual Lie bracket (see Example \ref{ex:orientedBrauerLieObject}).

\begin{rem} \label{rem:doubleModule}
	Note that $I_n(L)$ is a Lie algebra in $\gl_n\dashmod$. Precisely, the underlying object is the adjoint $\gl_n$-module, and it is equipped with the usual Lie bracket for $\gl_n$. The fact that $I_n(L)$ is itself a $\gl_n$-module does not play an important role in the context of this paper.
	
	\details{The same applies to the $I_n$-images of the modules $\goup$ and $\godown$. They're the natural module and its dual, respectively, each equipped with the same $\gl_n$-action twice; once as the $\gl_n$-action associated to any given object of $\gl_n\dashmod$, and once as the action of the Lie algebra $I_n(L)$.}
\end{rem}
Recall that $(L, \kk[t])\dashmod$ is a symmetric monoidal category (this is a special case of Proposition \ref{prop:symmetricMonoidalModules}). Hence, by Lemmas \ref{lem:orientedBrauerLieObject} and \ref{lem:orientedBrauerFree}, there is an incarnation functor $I \colon \OB \to (L, \kk[t])\dashmod$ sending $\goup \in \ob(\OB)$ to the natural $L$-module $\left(\goup,\ \begin{tikzpicture}[anchorbase]
\draw[->, thick] (0, 0) -- (0, 0.75);
\draw[->, thick] (0.6, 0) -- (0.6, 0.2) arc(0:180:0.2) -- (0.2, 0);
\end{tikzpicture} \right)$ and $\godown \in \ob(\OB)$ to the dual $L$-module $\left(\godown,\ -\ \begin{tikzpicture}[anchorbase]
\draw[->, thick] (0, 0) -- (0, 0.25) arc(180:0:0.25) -- (0.5, 0);
\draw[<-, thick] (0.25, 0) -- (0.25, 0.25) to[out=90, in=270] (0.5, 0.5) -- (0.5, 0.75);
\end{tikzpicture} \right)$.

\begin{defin}
	Let $X \in \ob(\OB)$. The \emph{canonical $L$-module structure on $X$} is the $L$-module $I(X)$, where $I \colon \OB \to (L, \kk[t])\dashmod$ is the functor discussed above.
\end{defin}

\begin{egs}
	The canonical $L$-module structures on $\goup$, $\godown$, and $\goup \godown \goup$ are, respectively, the natural $L$-module $\left(\goup,\ \begin{tikzpicture}[anchorbase]
	\draw[->, thick] (0, 0) -- (0, 0.75);
	\draw[->, thick] (0.6, 0) -- (0.6, 0.2) arc(0:180:0.2) -- (0.2, 0);
	\end{tikzpicture} \right)$, its dual $\left(\godown,\ -\ \begin{tikzpicture}[anchorbase]
	\draw[->, thick] (0, 0) -- (0, 0.25) arc(180:0:0.25) -- (0.5, 0);
	\draw[<-, thick] (0.25, 0) -- (0.25, 0.25) to[out=90, in=270] (0.5, 0.5) -- (0.5, 0.75);
	\end{tikzpicture} \right)$, and $\left(\goup\godown\goup,\ \begin{tikzpicture}[anchorbase]
	\draw[->, thick] (0, 0) -- (0, 0.75);
	\draw[->, thick] (0.6, 0) -- (0.6, 0.2) arc(0:180:0.2) -- (0.2, 0);
	\draw[<-, thick] (0.8, 0) -- (0.8, 0.75);
	\draw[->, thick] (1, 0) -- (1, 0.75);
	\end{tikzpicture}
	\ -\
	\begin{tikzpicture}[anchorbase]
	\draw[->, thick] (0, 0) -- (0, 0.2) arc(180:0:0.3) -- (0.6, 0);
	\draw[<-, thick] (0.2, 0) -- (0.2, 0.4) to[out=90, in=270] (0, 0.6) -- (0, 0.75);
	\draw[->, thick] (0.4, 0) -- (0.4, 0.4) to[out=90, in=270] (0.6, 0.6) -- (0.6, 0.75);
	\draw[->, thick] (0.9, 0) -- (0.9, 0.75);
	\end{tikzpicture}
	\ +\
	\begin{tikzpicture}[anchorbase]
	\draw[->, thick] (-0.1, 0) -- (-0.1, 0.75);
	\draw[<-, thick] (0.2, 0) -- (0.2, 0.2) arc(180:0:0.3) -- (0.8, 0);
	\draw[->, thick] (0.4, 0) -- (0.4, 0.4) to[out=90, in=270] (0.2, 0.6) -- (0.2, 0.75);
	\draw[<-, thick] (0.6, 0) -- (0.6, 0.4) to[out=90, in=270] (0.8, 0.6) -- (0.8, 0.75);
	\end{tikzpicture} \right)$.
\end{egs}

\begin{defin} \label{def:interpolatingCategories}
	We write $\Curr(\OB)$ (resp. $\Curr(\OB(n))$) for the full subcategory of $(L, \kk[t])\dashmod$ consisting of current $L$-modules in $\OB$ (resp. $\OB(n)$) whose underlying $L$-module structures are canonical.
	
	For each $n \in \N$, let $\Curr(I_n) \colon \Curr(\OB(n)) \to (\gl_n \otimes \kk[t])\dashmod$ denote the functor induced by $I_n$. Explicitly, $\Curr(I_n)$ sends a current $L$-module $\green{V}$ to the current $\gl_n$-module $\red{W} = I_n(\green{V})$, with action morphisms given by $\damapmultred[2px]{p(t)} = I_n\left(\damapmult[2px]{p(t)}\right)$, and $\Curr(I_n)$ acts on morphisms as $I_n$.
\end{defin}

\begin{conj} \label{conj:currentInterpolation}
	For each $n \in \N$, the functor $\Curr(I_n)$ is full, and hence $\Curr(\OB)$ is an interpolating category for the family of current $\gl_n$-module categories.
\end{conj}

\begin{rem}
	If one defined $\Curr(\OB(n))$ without requiring the underlying $L$-module structures to be canonical, the functors $\Curr(I_n)$ would fail to be full for $n \geq 2$. To see why, consider the current $L$-module $\goup$, with all action morphisms equal to 0. Then $\Curr(I_n)(\goup) = \kk^n$, equipped with the trivial action of $\gl_n \otimes \kk[t]$. Any linear endomorphism of $\kk^n$ is an endomorphism of this current module. But every morphism in $\End_\OB(\goup)$ gets mapped by $I_n$ to a multiple of the identity, so $\Curr(I_n)$ could not be full unless $n = 1$.
\end{rem}

\begin{lem}\label{lem:incarnationKernel}
	For each $n \in \N$, the kernel of $I_n \colon \OB(n) \to \gl_n\dashmod$ is the tensor ideal generated by the \emph{idempotent antisymmetrizer} $\begin{tikzpicture}[anchorbase]
	\draw[->, thick] (-0.125, 0) -- (-0.125, 1);
	\draw[->, thick] (0.625, 0) -- (0.625, 1);
	\node at (0.28, 0.08) {$\dotsb$};
	\node at (0.28, 0.88) {$\dotsb$};
	\filldraw[fill=lightgray,draw=black,rounded corners] (-0.25, 0.25) rectangle (0.75, 0.75);
	\node at (0.25, 0.5) {$n + 1$};
	\end{tikzpicture} := \frac{1}{(n + 1)!}\summ_{\sigma \in S_{n + 1}} \sgn(\sigma)\sigma \in \End_{\OB(n)}(\goup^{\otimes n + 1})$. Here, $S_{n + 1}$ denotes the symmetric group on $n + 1$ letters, and we identify a permutation $\sigma$ with the morphism represented by the corresponding diagrammatic permutation of $n + 1$ upwards-oriented strands.
\end{lem}

\begin{proof}
	The claim follows from \cite[Thm.~1.10]{negligibleMorphisms} and the universal property of additive Karoubi envelopes.
\end{proof}

	As noted above, each $I_n \colon \OB(n) \to \gl_n\dashmod$ is full. However, this does not immediately imply that $\Curr(I_n)$ is full. For a concrete example, consider the case $n = 2$, and let $a, b \in \kk$ be distinct constants. Let $\green{V} = \goup\goup\goup$ be the current $L$-module induced by $\begin{tikzpicture}[anchorbase]
	\draw[gmod] (0, 0) -- (0, 0.5);
	\adot{(0, 0.25)};
	\end{tikzpicture} := \begin{tikzpicture}[anchorbase]
	\draw[->, thick] (0, 0) -- (0, 1);
	\draw[->, thick] (0.25, 0) -- (0.25, 1);
	\draw[->, thick] (0.5, 0) -- (0.5, 1);
	\end{tikzpicture} + a\ \begin{tikzpicture}[anchorbase]
	\draw[->, thick] (0, 0) -- (0, 1);
	\draw[->, thick] (0.25, 0) -- (0.25, 1);
	\draw[->, thick] (0.5, 0) -- (0.5, 1);
	\filldraw[fill=lightgray,draw=black, rounded corners] (-0.25, 0.25) rectangle (0.75, 0.75);
	\node at (0.25, 0.5) {$3$};
	\end{tikzpicture}$, and $\red{W} = \goup\goup\goup$ the current $L$-module induced by $\begin{tikzpicture}[anchorbase]
	\draw[rmod] (0, 0) -- (0, 0.5);
	\bdot{(0, 0.25)};
	\end{tikzpicture} := \begin{tikzpicture}[anchorbase]
	\draw[->, thick] (0, 0) -- (0, 1);
	\draw[->, thick] (0.25, 0) -- (0.25, 1);
	\draw[->, thick] (0.5, 0) -- (0.5, 1);
	\end{tikzpicture} + b\ \begin{tikzpicture}[anchorbase]
	\draw[->, thick] (0, 0) -- (0, 1);
	\draw[->, thick] (0.25, 0) -- (0.25, 1);
	\draw[->, thick] (0.5, 0) -- (0.5, 1);
	\filldraw[fill=lightgray,draw=black, rounded corners] (-0.25, 0.25) rectangle (0.75, 0.75);
	\node at (0.25, 0.5) {$3$};
	\end{tikzpicture}\ $. Since $\begin{tikzpicture}[anchorbase]
	\draw[->, thick] (0, 0) -- (0, 1);
	\draw[->, thick] (0.25, 0) -- (0.25, 1);
	\draw[->, thick] (0.5, 0) -- (0.5, 1);
	\filldraw[fill=lightgray,draw=black, rounded corners] (-0.25, 0.25) rectangle (0.75, 0.75);
	\node at (0.25, 0.5) {$3$};
	\end{tikzpicture}$ lies in the kernel of $I_2$, we have $\Curr(I_2)(\green{V}) = \Curr(I_2)(\red{W}) = U^{\otimes 3}$ (equipped with the current $\gl_2$-module structure induced by the identity), where $U$ denotes the natural $\gl_2$-module. In order for $\Curr(I_2)$ to be full, the identity map on $U^{\otimes 3}$ must have a $\Curr(I_2)$-preimage in $\Hom_{\Curr(\OB(2))}(\green{V}, \red{W})$. It is straightforward to confirm that $\frac{1}{3}\begin{tikzpicture}[anchorbase]
	\draw[<-, thick] (0, 0) -- (0, -0.25) arc(180:360:0.125) -- (0.25, 0);
	\draw[gmod] (0.5, -0.5) -- (0.5, 0);
	\end{tikzpicture}$ is a right inverse for $\dmodulemult$, and using the fact that such a right inverse exists, one can show that an $\OB(2)$-morphism $\begin{tikzpicture}[anchorbase]
	\draw[gmod] (0, 0) -- (0, 0.25);
	\draw[rmod] (0, 0.25) -- (0, 0.5);
	\bsquare{(0, 0.25)};
	\end{tikzpicture}$ is a morphism of current modules from $\green{V}$ to $\red{W}$ if and only if $\begin{tikzpicture}[anchorbase]
	\draw[gmod] (0, 0) -- (0, 0.25);
	\draw[rmod] (0, 0.25) -- (0, 0.75);
	\bsquare{(0, 0.25)};
	\bdot{(0, 0.5)};
	\end{tikzpicture} = \begin{tikzpicture}[anchorbase]
	\draw[gmod] (0, 0) -- (0, 0.5);
	\draw[rmod] (0, 0.5) -- (0, 0.75);
	\adot{(0, 0.25)};
	\bsquare{(0, 0.5)};
	\end{tikzpicture}$. The $I_2$-preimages of $\id_{V^{\otimes 3}}$ are $\begin{tikzpicture}[anchorbase]
	\draw[gmod] (0, 0) -- (0, 0.25);
	\draw[rmod] (0, 0.25) -- (0, 0.5);
	\bsquaretoken[3px]{0, 0.25}{west}{c};
	\end{tikzpicture} := \begin{tikzpicture}[anchorbase]
	\draw[->, thick] (0, 0) -- (0, 1);
	\draw[->, thick] (0.25, 0) -- (0.25, 1);
	\draw[->, thick] (0.5, 0) -- (0.5, 1);
	\end{tikzpicture} + c\ \begin{tikzpicture}[anchorbase]
	\draw[->, thick] (0, 0) -- (0, 1);
	\draw[->, thick] (0.25, 0) -- (0.25, 1);
	\draw[->, thick] (0.5, 0) -- (0.5, 1);
	\filldraw[fill=lightgray,draw=black, rounded corners] (-0.25, 0.25) rectangle (0.75, 0.75);
	\node at (0.25, 0.5) {$3$};
	\end{tikzpicture}$\ , $c \in \kk$. Using the fact that $\begin{tikzpicture}[anchorbase]
	\draw[->, thick] (0, 0) -- (0, 1);
	\draw[->, thick] (0.25, 0) -- (0.25, 1);
	\draw[->, thick] (0.5, 0) -- (0.5, 1);
	\filldraw[fill=lightgray,draw=black, rounded corners] (-0.25, 0.25) rectangle (0.75, 0.75);
	\node at (0.25, 0.5) {$3$};
	\end{tikzpicture}$ is idempotent, we get $\begin{tikzpicture}[anchorbase]
	\draw[gmod] (0, 0) -- (0, 0.25);
	\draw[rmod] (0, 0.25) -- (0, 0.75);
	\bsquaretoken[3px]{0, 0.25}{west}{c};
	\bdot{(0, 0.5)};
	\end{tikzpicture} = \begin{tikzpicture}[anchorbase]
	\draw[->, thick] (0, 0) -- (0, 1);
	\draw[->, thick] (0.25, 0) -- (0.25, 1);
	\draw[->, thick] (0.5, 0) -- (0.5, 1);
	\end{tikzpicture} + (b + c + bc)\ \begin{tikzpicture}[anchorbase]
	\draw[->, thick] (0, 0) -- (0, 1);
	\draw[->, thick] (0.25, 0) -- (0.25, 1);
	\draw[->, thick] (0.5, 0) -- (0.5, 1);
	\filldraw[fill=lightgray,draw=black, rounded corners] (-0.25, 0.25) rectangle (0.75, 0.75);
	\node at (0.25, 0.5) {$3$};
	\end{tikzpicture}$ and $\begin{tikzpicture}[anchorbase]
	\draw[gmod] (0, 0) -- (0, 0.5);
	\draw[rmod] (0, 0.5) -- (0, 0.75);
	\adot{(0, 0.25)};
	\bsquaretoken[3px]{0, 0.5}{west}{c};
	\end{tikzpicture} = \begin{tikzpicture}[anchorbase]
	\draw[->, thick] (0, 0) -- (0, 1);
	\draw[->, thick] (0.25, 0) -- (0.25, 1);
	\draw[->, thick] (0.5, 0) -- (0.5, 1);
	\end{tikzpicture} + (a + c + ac)\ \begin{tikzpicture}[anchorbase]
	\draw[->, thick] (0, 0) -- (0, 1);
	\draw[->, thick] (0.25, 0) -- (0.25, 1);
	\draw[->, thick] (0.5, 0) -- (0.5, 1);
	\filldraw[fill=lightgray,draw=black, rounded corners] (-0.25, 0.25) rectangle (0.75, 0.75);
	\node at (0.25, 0.5) {$3$};
	\end{tikzpicture}$. Hence $\begin{tikzpicture}[anchorbase]
	\draw[gmod] (0, 0) -- (0, 0.25);
	\draw[rmod] (0, 0.25) -- (0, 0.5);
	\bsquaretoken[3px]{0, 0.25}{west}{c};
	\end{tikzpicture}$ is a morphism of current $L$-modules if and only if $b + c + bc = a + c + ac$. Given that $a \neq b$, this equality holds if and only if $c = -1$, regardless of the choice of $a$ and $b$. Thus $\begin{tikzpicture}[anchorbase]
	\draw[gmod] (0, 0) -- (0, 0.25);
	\draw[rmod] (0, 0.25) -- (0, 0.5);
	\bsquaretoken[3px]{0, 0.25}{west}{-1};
	\end{tikzpicture}$ is the unique $\Curr(I_2)$-preimage of $\id_{U^{\otimes 3}}$ in this case. (If $a = b$, then $\begin{tikzpicture}[anchorbase]
	\draw[gmod] (0, 0) -- (0, 0.25);
	\draw[rmod] (0, 0.25) -- (0, 0.5);
	\bsquaretoken[3px]{0, 0.25}{west}{c};
	\end{tikzpicture}$ is a morphism of current modules for all $c \in \kk$.) To show that $\Curr(I_2)$ is full, one would have to prove that at least one $\Curr(I_2)$-preimage always exists, for all possible choices of current $L$-modules $\green{V}$ and $\red{W}$, and all morphisms between their images in $(\gl_n \otimes \kk[t])\dashmod$.
	
	For another example, keep $n = 2$, but consider the object $\goup\goup\goup\goup$. The kernel of $I_2$ restricted to $\End_{\OB(2)}(\goup\goup\goup\goup)$ is 10-dimensional -- this result and the other claims in this paragraph follow from direct computer calculations. Let $\green{V}$ and $\red{W}$ be the current $L$-modules induced by $\begin{tikzpicture}[anchorbase]
	\draw[->, thick] (0, 0) -- (0, 1);
	\draw[->, thick] (0.25, 0) -- (0.25, 1);
	\draw[->, thick] (0.5, 0) -- (0.5, 1);
	\draw[->, thick] (0.75, 0) -- (0.75, 1);
	\end{tikzpicture}
	\ +\
	\begin{tikzpicture}[anchorbase]
	\draw[->, thick] (0, 0) -- (0, 1);
	\draw[->, thick] (0.25, 0) -- (0.25, 1);
	\draw[->, thick] (0.5, 0) -- (0.5, 1);
	\draw[->, thick] (0.75, 0) -- (0.75, 1);
	\filldraw[fill=white,draw=black] (-0.25, 0.25) rectangle (1, 0.75);
	\node at (0.375, 0.5) {$k_1$};
	\end{tikzpicture}
	$ and $\begin{tikzpicture}[anchorbase]
	\draw[->, thick] (0, 0) -- (0, 1);
	\draw[->, thick] (0.25, 0) -- (0.25, 1);
	\draw[->, thick] (0.5, 0) -- (0.5, 1);
	\draw[->, thick] (0.75, 0) -- (0.75, 1);
	\end{tikzpicture}
	\ +\
	\begin{tikzpicture}[anchorbase]
	\draw[->, thick] (0, 0) -- (0, 1);
	\draw[->, thick] (0.25, 0) -- (0.25, 1);
	\draw[->, thick] (0.5, 0) -- (0.5, 1);
	\draw[->, thick] (0.75, 0) -- (0.75, 1);
	\filldraw[fill=white,draw=black] (-0.25, 0.25) rectangle (1, 0.75);
	\node at (0.375, 0.5) {$k_2$};
	\end{tikzpicture}
	$\ , respectively, where $k_1$ and $k_2$ are morphisms in the kernel of $I_2$. The dimension of the affine space of $\Curr(I_2)$-preimages of $\id_{U^{\otimes 4}}$ in $\Hom_{\Curr(\OB(2))}(\green{V}, \red{W})$ depends on the choice of $k_1$ and $k_2$. For instance, when $\begin{tikzpicture}[anchorbase]
	\draw[->, thick] (0, 0) -- (0, 1);
	\draw[->, thick] (0.25, 0) -- (0.25, 1);
	\draw[->, thick] (0.5, 0) -- (0.5, 1);
	\draw[->, thick] (0.75, 0) -- (0.75, 1);
	\filldraw[fill=white,draw=black] (-0.25, 0.25) rectangle (1, 0.75);
	\node at (0.375, 0.5) {$k_1$};
	\end{tikzpicture} = \begin{tikzpicture}[anchorbase]
	\draw[->, thick] (0, 0) -- (0, 1);
	\draw[->, thick] (0.25, 0) -- (0.25, 1);
	\draw[->, thick] (0.5, 0) -- (0.5, 1);
	\draw[->, thick] (1, 0) -- (1, 1);
	\filldraw[fill=lightgray,draw=black, rounded corners] (-0.25, 0.25) rectangle (0.75, 0.75);
	\node at (0.25, 0.5) {$3$};
	\end{tikzpicture}$ and $\begin{tikzpicture}[anchorbase]
	\draw[->, thick] (0, 0) -- (0, 1);
	\draw[->, thick] (0.25, 0) -- (0.25, 1);
	\draw[->, thick] (0.5, 0) -- (0.5, 1);
	\draw[->, thick] (0.75, 0) -- (0.75, 1);
	\filldraw[fill=white,draw=black] (-0.25, 0.25) rectangle (1, 0.75);
	\node at (0.375, 0.5) {$k_2$};
	\end{tikzpicture} = \begin{tikzpicture}[anchorbase]
	\draw[->, thick] (0, 0) -- (0, 1);
	\draw[->, thick] (0.25, 0) -- (0.25, 1);
	\draw[->, thick] (0.5, 0) -- (0.5, 1);
	\draw[->, thick] (-0.5, 0) -- (-0.5, 1);
	\filldraw[fill=lightgray,draw=black, rounded corners] (-0.25, 0.25) rectangle (0.75, 0.75);
	\node at (0.25, 0.5) {$3$};
	\end{tikzpicture}$, the space of preimages has dimension 6. If we use the same choice of $k_1$ but define $\begin{tikzpicture}[anchorbase]
	\draw[->, thick] (0, 0) -- (0, 1);
	\draw[->, thick] (0.25, 0) -- (0.25, 1);
	\draw[->, thick] (0.5, 0) -- (0.5, 1);
	\draw[->, thick] (0.75, 0) -- (0.75, 1);
	\filldraw[fill=white,draw=black] (-0.25, 0.25) rectangle (1, 0.75);
	\node at (0.375, 0.5) {$k_2$};
	\end{tikzpicture} = \begin{tikzpicture}[anchorbase]
	\draw[->, thick] (0, -0.5) -- (0, 1);
	\draw[->, thick] (0.25, -0.5) -- (0.25, 1);
	\draw[->, thick] (1, -0.5) to[out=90, in=270] (0.5, 0) -- (0.5, 1);
	\draw[->, thick] (0.5, -0.5) to[out=90, in=270] (1, 0) -- (1, 1);
	\filldraw[fill=lightgray,draw=black, rounded corners] (-0.25, 0.25) rectangle (0.75, 0.75);
	\node at (0.25, 0.5) {$3$};
	\end{tikzpicture}$, then the space of preimages has dimension 4. Studying these low-dimensional cases further may shed light on the general problem of showing that $\Curr(I_n)$ is full.

The Brauer category $\B$ was first explicitly defined and studied by Lehrer and Zhang in \cite{LZBrauerCategory}. Its endomorphism algebras are the \emph{Brauer algebras} introduced by Richard Brauer in \cite{brauerAlgebras}, and it interpolates between the categories $\so_n\dashmod$, $n \in \N$ (as well as between the categories $\sp_{2n}\dashmod$ and, more generally, the categories $\osp_{n|2m}\dashmod$). Note that the following definition matches \cite[Def.~9.1]{diagrammaticsForRealSupergroups} (taking $A = \kk$, $\diamond = \id$, and $\sigma = 0$), and is a refinement of the presentation outlined in \cite[Thm.~2.6]{LZBrauerCategory}.

\begin{defin} \label{def:brauerCategory}
	The \emph{(unoriented) Brauer category}, denoted $\B$, is the strict $\kk$-linear monoidal category with one generating object, written as $\go$, generating morphisms $\dcross, \dcup, \dcap,$ and the following relations:
	\begin{equation} \label{rel:unorientedBasics}
	%braiding
	\begin{tikzpicture}[anchorbase]
	\draw[-, thick] (0, 0) -- (0.8, 0.8);
	\draw[-, thick] (0.4, 0) .. controls (0, 0.4) .. (0.4, 0.8);
	\draw[-, thick] (0.8, 0) -- (0, 0.8);
	\end{tikzpicture}
	\ =\
	\begin{tikzpicture}[anchorbase]
	\draw[-, thick] (0, 0) -- (0.8, 0.8);
	\draw[-, thick] (0.4, 0) .. controls (0.8, 0.4) .. (0.4, 0.8);
	\draw[-, thick] (0.8, 0) -- (0, 0.8);
	\end{tikzpicture}\ , \quad
	%double crossing
	\begin{tikzpicture}[anchorbase]
	\draw[-, thick] (0, 0) -- (0, 0.2) to[out=90,in=270] (0.4, 0.6) to[out=90,in=270] (0, 1) -- (0, 1.2);
	\draw[-, thick] (0.4, 0) -- (0.4, 0.2) to[out=90,in=270] (0, 0.6) to[out=90,in=270] (0.4, 1) -- (0.4, 1.2);
	\end{tikzpicture}
	\ =\
	\begin{tikzpicture}[anchorbase]
	\draw[-, thick] (0, 0) -- (0, 1.2);
	\draw[-, thick] (0.4, 0) -- (0.4, 1.2);
	\end{tikzpicture}\ , \quad
	%zigzag
	\begin{tikzpicture}[anchorbase]
	\draw[-, thick] (0,0) -- (0,0.6) arc(180:0:0.2) -- (0.4,0.4) arc(180:360:0.2) -- (0.8,1);
	\end{tikzpicture}
	\ =\
	\begin{tikzpicture}[anchorbase]
	\draw[-, thick] (0,0) -- (0,1);
	\end{tikzpicture}\ , \quad
	\begin{tikzpicture}[anchorbase]
	\draw[-, thick] (0, 1) -- (0, 0.4) arc(180:360:0.2) -- (0.4, 0.6) arc(180:0:0.2) -- (0.8, 0);
	\end{tikzpicture}
	\ =\
	\begin{tikzpicture}[anchorbase]
	\draw[-, thick] (0,0) -- (0,1);
	\end{tikzpicture}\ , \quad
	\begin{tikzpicture}[anchorbase]
	\draw[-, thick] (0, 0) to[out=90,in=270] (0.5, 0.5) arc(0:180:0.25) to[out=270,in=90] (0.5, 0);
	\end{tikzpicture}\ 
	\ =\
	\begin{tikzpicture}[anchorbase]
	\draw[-, thick] (0, 0) -- (0, 0.25) arc(0:180:0.25) -- (-0.5, 0);
	\end{tikzpicture}\ , \quad
	\begin{tikzpicture}[anchorbase]
	\draw[-, thick] (0, 0) -- (0, 0.25) arc(180:0:0.25) -- (0.5, 0);
	\draw[-, thick] (0.25, 0) to[out=90,in=270] (-0.25, 0.75);
	\end{tikzpicture}\
	\ =\
	\begin{tikzpicture}[anchorbase]
	\draw[-, thick] (0, 0) -- (0, 0.25) arc(180:0:0.25) -- (0.5, 0);
	\draw[-, thick] (0.25, 0) to[out=90,in=270] (0.75, 0.75);
	\end{tikzpicture}\ .
	\end{equation}
	
	For each $\delta \in \kk$, we define the \emph{specialized Brauer category} $\B(\delta)$ to be the category obtained from $\B$ by imposing one additional relation:
	\begin{equation}
	\begin{tikzpicture}[anchorbase]
	\draw[-, thick] (0, 0) arc(0:360:0.25);
	\end{tikzpicture} = \delta.
	\end{equation}
\end{defin}

\begin{lem}
	The Brauer category is the free $\kk$-linear symmetric monoidal category generated by a symmetrically self-dual object. That is, if $\CC$ is a $\kk$-linear symmetric monoidal category and $X \in \ob(\CC)$ is a self-dual object such that $$\begin{tikzpicture}[anchorbase]
	\draw[-, thick] (0, 0) to[out=90,in=270] (0.5, 0.5) arc(0:180:0.25) to[out=270,in=90] (0.5, 0);
	\node at (0, -0.2) {$X$};
	\node at (0.5, -0.2) {$X$};
	\end{tikzpicture}\ 
	\ =\
	\begin{tikzpicture}[anchorbase]
	\draw[-, thick] (0, 0) -- (0, 0.25) arc(0:180:0.25) -- (-0.5, 0);
	\node at (0, -0.2) {$X$};
	\node at (-0.5, -0.2) {$X$};
	\end{tikzpicture},$$ 
	then there is a unique monoidal functor $I \colon \B \to \CC$ such that $$I(\go) = X, \quad \quad I\left(\dcross\right) = \begin{tikzpicture}[anchorbase]
	\draw[bmod] (0, 0) -- (0.5, 0.5);
	\draw[bmod] (0.5, 0) -- (0, 0.5);
	\node at (0, -0.25) {$X$};
	\node at (0.5, -0.25) {$X$};
	\node at (0, 0.75) {$X$};
	\node at (0.5, 0.75) {$X$};
	\end{tikzpicture}, \quad \text{and} \quad I\left(\dcap\right) = \begin{tikzpicture}[anchorbase]
	\draw[bmod] (0, 0) -- (0, 0.2) arc(180:0:0.25) -- (0.5, 0);
	\node at (0, -0.25) {$X$};
	\node at (0.5, -0.25) {$X$};
	\end{tikzpicture}.$$
	This functor also satisfies $$I\left(\dcup\right) = \begin{tikzpicture}[anchorbase, yscale=-1]
	\draw[bmod] (0, 0) -- (0, 0.2) arc(180:0:0.25) -- (0.5, 0);
	\node at (0, -0.25) {$X$};
	\node at (0.5, -0.25) {$X$};
	\end{tikzpicture}.$$ We call such a functor an \emph{(unoriented) incarnation functor}.
\end{lem}

\begin{proof}
	The claim follows from arguments analogous to those for the oriented case (Lemma \ref{lem:orientedBrauerFree}). 
\end{proof}

Similarly to the oriented case considered above, for each $n \in \N$ we write $I_n'$ for the unoriented incarnation functor obtained by picking $\CC$ to be $\so_n\dashmod$ and $X$ to be the natural $\so_n$-module. (One can also work with $\sp_{2n}$ rather than $\so_n$, or, more generally, the orthosymplectic Lie algebra $\osp_{n \mid 2m}$.) We have $I_n' \left(\begin{tikzpicture}[anchorbase]
\draw[-, thick] (0, 0) arc(0:360:0.25);
\end{tikzpicture} \right) = \dim(X) = n$, and hence we may regard $I_n'$ as a functor from $\B(n)$ to $\so_n\dashmod$.

\begin{defin} \label{def:karoubiEnvelope}
	If $\CC$ is a symmetric monoidal $\kk$-linear category, we write $\Kar(\CC)$ for its additive Karoubi envelope, i.e.\ the idempotent completion of its additive envelope. Concretely, we take the objects of $\Kar(\CC)$ to be pairs $(X, f)$, where $X = X_1 \oplus \dotsb \oplus X_n$ is a finite formal sum of objects from $\CC$, and $f \colon X \to X$ is an idempotent morphism in the additive envelope of $\CC$. A morphism $h \colon (X, f) \to (Y, g)$ is a morphism $h \colon X \to Y$ in the additive envelope of $\CC$ satisfying $g \circ h = h = h \circ f$.
\end{defin}

\begin{lem} \label{lem:unorientedBrauerObject}
	Consider the object $\green{L} = \left(\go\go, \begin{tikzpicture}[anchorbase]
	\draw[gmod] (0, 0) -- (0, 0.5);
	\adot{(0, 0.25)};
	\end{tikzpicture} \right) := \left(\go \go, \frac{1}{2}\left(\begin{tikzpicture}[anchorbase]
	\draw[bmod] (0, 0) -- (0, 0.5);
	\draw[bmod] (0.25, 0) -- (0.25, 0.5);
	\end{tikzpicture}
	\ -\
	\begin{tikzpicture}[anchorbase]
	\draw[bmod] (0, 0) -- (0.5, 0.5);
	\draw[bmod] (0.5, 0) -- (0, 0.5);
	\end{tikzpicture}\right)\right)$ in $\Kar(\B)$. We have that 
	$$\left(\green{L}, \begin{tikzpicture}[anchorbase]
	\draw[gmod] (0, 0) -- (0.25, 0.25) -- (0.25, 0.5);
	\draw[gmod] (0.5, 0) -- (0.25, 0.25);
	\adot{(0.25, 0.25)};
	\end{tikzpicture} \right) := \left(\green{L}, \frac{1}{4}\left(\begin{tikzpicture}[anchorbase]
	\draw[-, thick] (0, 0) to[out=90, in=270] (0.25, 0.5) -- (0.25, 0.75);
	\draw[-, thick] (0.25, 0) -- (0.25, 0.2) arc(180:0:0.125) -- (0.5, 0);
	\draw[-, thick] (0.75, 0) to[out=90, in=270] (0.5, 0.5) -- (0.5, 0.75);
	\end{tikzpicture}
	\ -\
	\begin{tikzpicture}[anchorbase]
	\draw[-, thick] (0.5, -0.75) -- (0.5, -0.5) to[out=90, in=270] (0, 0) to[out=90, in=270] (0.25, 0.5) -- (0.25, 0.75);
	\draw[-, thick] (0.75, -0.75) -- (0.75, -0.5) to[out=90, in=270] (0.5, 0) arc(0:180:0.125) to[out=270, in=90] (0, -0.5) -- (0, -0.75);
	\draw[-, thick] (0.25, -0.75) -- (0.25, -0.5) to[out=90, in=270] (0.75, 0) to[out=90, in=270] (0.5, 0.5) -- (0.5, 0.75);
	\end{tikzpicture} 
	\ -\
	\begin{tikzpicture}[anchorbase]
	\draw[bmod] (0, 0) -- (0, 1);
	\draw[bmod] (0.25, 0) -- (0.25, 0.25) arc(180:0:0.25) -- (0.75, 0);
	\draw[bmod] (0.5, 0) to[out=90, in=270] (1, 0.5) -- (1, 1);
	\end{tikzpicture}
	\ +\
	\begin{tikzpicture}[anchorbase]
	\draw[bmod] (0, 0) -- (0, 0.25) arc(180:0:0.25) -- (0.5, 0);
	\draw[bmod] (0.25, 0) to[out=90, in=270] (0.75, 0.5) -- (0.75, 1);
	\draw[bmod] (0.75, 0) -- (0.75, 0.25) to[out=90, in=270] (0.5, 0.5) -- (0.5, 1);
	\end{tikzpicture}
	\ -\
	\begin{tikzpicture}[anchorbase, xscale=-1]
	\draw[bmod] (0, 0) -- (0, 1);
	\draw[bmod] (0.25, 0) -- (0.25, 0.25) arc(180:0:0.25) -- (0.75, 0);
	\draw[bmod] (0.5, 0) to[out=90, in=270] (1, 0.5) -- (1, 1);
	\end{tikzpicture}
	\ +\
	\begin{tikzpicture}[anchorbase, xscale=-1]
	\draw[bmod] (0, 0) -- (0, 0.25) arc(180:0:0.25) -- (0.5, 0);
	\draw[bmod] (0.25, 0) to[out=90, in=270] (0.75, 0.5) -- (0.75, 1);
	\draw[bmod] (0.75, 0) -- (0.75, 0.25) to[out=90, in=270] (0.5, 0.5) -- (0.5, 1);
	\end{tikzpicture}
	\ +\
	\begin{tikzpicture}[anchorbase]
	\draw[bmod] (0, 0) -- (0, 0.25) arc(180:0:0.375) -- (0.75, 0);
	\draw[bmod] (0.25, 0) to[out=90, in=270] (-0.25, 0.5) -- (-0.25, 1);
	\draw[bmod] (0.5, 0) to[out=90, in=270] (1, 0.5) -- (1, 1);
	\end{tikzpicture}
	\ -\
	\begin{tikzpicture}[anchorbase]
	\draw[-, thick] (0, 0) to[out=90, in=270] (0.25, 0.5) to[out=90, in=270] (0.5, 0.75) -- (0.5, 1);
	\draw[-, thick] (0.25, 0) -- (0.25, 0.2) arc(180:0:0.125) -- (0.5, 0);
	\draw[-, thick] (0.75, 0) to[out=90, in=270] (0.5, 0.5) to[out=90, in=270] (0.25, 0.75) -- (0.25, 1);
	\end{tikzpicture}\right)\right)$$
	is a Lie algebra in $\Kar(\B)$, and $\left(\left(\go, \id_{\go}\right),\ \frac{1}{2}\left(\begin{tikzpicture}[anchorbase]
	\draw[-, thick] (0, 0) -- (0, 0.75);
	\draw[-, thick] (0.6, 0) -- (0.6, 0.2) arc(0:180:0.2) -- (0.2, 0);
	\end{tikzpicture}
	\ -\ 
	\begin{tikzpicture}[anchorbase]
	\draw[bmod] (0, 0) -- (0, 0.25) arc(180:0:0.25) -- (0.5, 0);
	\draw[bmod] (0.25, 0) -- (0.25, 0.25) to[out=90, in=270] (0, 0.5) -- (0, 0.75);
	\end{tikzpicture}\right) \right)$ is a self-dual module for this Lie algebra. Moreover, for each $n \in \N$, the $I_n'$-image of this Lie algebra is $\so_n$ equipped with the usual Lie bracket, and the $I_n'$-image of the module is the natural $\so_n$-module.
\end{lem}

\begin{proof}
	A direct computation shows that
	\begin{equation} \label{eq:brauerLieIdentity}
	\begin{tikzpicture}[anchorbase]
	\draw[gmod] (0, 0) -- (0.25, 0.25) -- (0.25, 0.5);
	\draw[gmod] (0.5, 0) -- (0.25, 0.25);
	\adot{(0.25, 0.25)};
	\end{tikzpicture} = \begin{tikzpicture}[anchorbase]
	\draw[gmod] (0, 0) -- (0, 0.5);
	\adot{(0, 0.25)};
	\end{tikzpicture} \circ \left(\begin{tikzpicture}[anchorbase]
	\draw[-, thick] (0, 0) to[out=90, in=270] (0.25, 0.5) -- (0.25, 0.75);
	\draw[-, thick] (0.25, 0) -- (0.25, 0.2) arc(180:0:0.125) -- (0.5, 0);
	\draw[-, thick] (0.75, 0) to[out=90, in=270] (0.5, 0.5) -- (0.5, 0.75);
	\end{tikzpicture}
	\ -\
	\begin{tikzpicture}[anchorbase]
	\draw[-, thick] (0.5, -0.75) -- (0.5, -0.5) to[out=90, in=270] (0, 0) to[out=90, in=270] (0.25, 0.5) -- (0.25, 0.75);
	\draw[-, thick] (0.75, -0.75) -- (0.75, -0.5) to[out=90, in=270] (0.5, 0) arc(0:180:0.125) to[out=270, in=90] (0, -0.5) -- (0, -0.75);
	\draw[-, thick] (0.25, -0.75) -- (0.25, -0.5) to[out=90, in=270] (0.75, 0) to[out=90, in=270] (0.5, 0.5) -- (0.5, 0.75);
	\end{tikzpicture} \right) \circ \begin{tikzpicture}[anchorbase]
	\draw[gmod] (0, 0) -- (0, 0.5);
	\adot{(0, 0.25)};
	\end{tikzpicture}\ \begin{tikzpicture}[anchorbase]
	\draw[gmod] (0, 0) -- (0, 0.5);
	\adot{(0, 0.25)};
	\end{tikzpicture} = \left(\begin{tikzpicture}[anchorbase]
	\draw[-, thick] (0, 0) to[out=90, in=270] (0.25, 0.5) -- (0.25, 0.75);
	\draw[-, thick] (0.25, 0) -- (0.25, 0.2) arc(180:0:0.125) -- (0.5, 0);
	\draw[-, thick] (0.75, 0) to[out=90, in=270] (0.5, 0.5) -- (0.5, 0.75);
	\end{tikzpicture}
	\ -\
	\begin{tikzpicture}[anchorbase]
	\draw[-, thick] (0.5, -0.75) -- (0.5, -0.5) to[out=90, in=270] (0, 0) to[out=90, in=270] (0.25, 0.5) -- (0.25, 0.75);
	\draw[-, thick] (0.75, -0.75) -- (0.75, -0.5) to[out=90, in=270] (0.5, 0) arc(0:180:0.125) to[out=270, in=90] (0, -0.5) -- (0, -0.75);
	\draw[-, thick] (0.25, -0.75) -- (0.25, -0.5) to[out=90, in=270] (0.75, 0) to[out=90, in=270] (0.5, 0.5) -- (0.5, 0.75);
	\end{tikzpicture} \right) \circ \begin{tikzpicture}[anchorbase]
	\draw[gmod] (0, 0) -- (0, 0.5);
	\adot{(0, 0.25)};
	\end{tikzpicture}\ \begin{tikzpicture}[anchorbase]
	\draw[gmod] (0, 0) -- (0, 0.5);
	\adot{(0, 0.25)};
	\end{tikzpicture}.
	\end{equation}
	The first equality confirms that $\begin{tikzpicture}[anchorbase]
	\draw[gmod] (0, 0) -- (0.25, 0.25) -- (0.25, 0.5);
	\draw[gmod] (0.5, 0) -- (0.25, 0.25);
	\adot{(0.25, 0.25)};
	\end{tikzpicture}$ is indeed a morphism from $\green{L} \otimes \green{L}$ to $\green{L}$ in $\Kar(\B)$. Lemma \ref{lem:orientedBrauerLieObject} tells us that $\left(\go\go, \begin{tikzpicture}[anchorbase]
	\draw[gmod] (0, 0) -- (0.25, 0.25) -- (0.25, 0.5);
	\draw[gmod] (0.5, 0) -- (0.25, 0.25);
	\end{tikzpicture}\right) := \left(\go\go, \begin{tikzpicture}[anchorbase]
	\draw[-, thick] (0, 0) to[out=90, in=270] (0.25, 0.5) -- (0.25, 0.75);
	\draw[-, thick] (0.25, 0) -- (0.25, 0.2) arc(180:0:0.125) -- (0.5, 0);
	\draw[-, thick] (0.75, 0) to[out=90, in=270] (0.5, 0.5) -- (0.5, 0.75);
	\end{tikzpicture}
	\ -\
	\begin{tikzpicture}[anchorbase]
	\draw[-, thick] (0.5, -0.75) -- (0.5, -0.5) to[out=90, in=270] (0, 0) to[out=90, in=270] (0.25, 0.5) -- (0.25, 0.75);
	\draw[-, thick] (0.75, -0.75) -- (0.75, -0.5) to[out=90, in=270] (0.5, 0) arc(0:180:0.125) to[out=270, in=90] (0, -0.5) -- (0, -0.75);
	\draw[-, thick] (0.25, -0.75) -- (0.25, -0.5) to[out=90, in=270] (0.75, 0) to[out=90, in=270] (0.5, 0.5) -- (0.5, 0.75);
	\end{tikzpicture}\right)$ is a Lie algebra in $\B$. Hence: 
	\begin{multline*}
	\begin{tikzpicture}[anchorbase]
	\draw[gmod] (0.25, 0) -- (0.5, 0.25);
	\draw[gmod] (0, 0) -- (0, 0.25) -- (0.25, 0.5);
	\draw[gmod] (0.75, 0) -- (0.5, 0.25) -- (0.25, 0.5) -- (0.25, 0.75);
	\adot{(0.5, 0.25)};
	\adot{(0.25, 0.5)};
	\end{tikzpicture}
	\overset{\eqref{eq:brauerLieIdentity}}{=}
	\begin{tikzpicture}[anchorbase]
	\draw[gmod] (0.25, 0) -- (0.5, 0.25);
	\draw[gmod] (0, 0) -- (0, 0.25) -- (0.25, 0.5);
	\draw[gmod] (0.75, 0) -- (0.5, 0.25) -- (0.25, 0.5) -- (0.25, 0.75);
	\adot{(0.125, 0.375)};
	\adot{(0.375, 0.375)};
	\adot{(0.375, 0.125)};
	\adot{(0.625, 0.125)};
	\end{tikzpicture} 
	\overset{\eqref{eq:brauerLieIdentity}}{=}
	\begin{tikzpicture}[anchorbase]
	\draw[gmod] (0.25, 0) -- (0.5, 0.25);
	\draw[gmod] (0, 0) -- (0, 0.25) -- (0.25, 0.5);
	\draw[gmod] (0.75, 0) -- (0.5, 0.25) -- (0.25, 0.5) -- (0.25, 0.75);
	\adot{(0, 0.125)};
	\adot{(0.375, 0.125)};
	\adot{(0.625, 0.125)};
	\end{tikzpicture}
	\overset{\eqref{eq:jacobi}}{=}
	\begin{tikzpicture}[anchorbase]
	\draw[gmod] (0.5, 0) -- (0.25, 0.25);
	\draw[gmod] (0, 0) -- (0.25, 0.25) -- (0.5, 0.5);
	\draw[gmod] (0.75, 0) -- (0.75, 0.25) -- (0.5, 0.5) -- (0.5, 0.75);
	\adot{(0.125, 0.125)};
	\adot{(0.375, 0.125)};
	\adot{(0.75, 0.125)};
	\end{tikzpicture}
	\ +\
	\begin{tikzpicture}[anchorbase]
	\draw[gmod] (0, -0.375) -- (0, -0.25) -- (0.25, 0) -- (0.5, 0.25);
	\draw[gmod] (0.25, -0.375) -- (0.25, -0.25) -- (0, 0) -- (0, 0.25) -- (0.25, 0.5);
	\draw[gmod] (0.75, -0.375) -- (0.75, -0.25) -- (0.75, 0) -- (0.5, 0.25) -- (0.25, 0.5) -- (0.25, 0.75);
	\adot{(0, -0.25)};
	\adot{(0.25, -0.25)};
	\adot{(0.75, -0.25)};
	\end{tikzpicture}
	\overset{\eqref{eq:brauerLieIdentity}}{=}
	\begin{tikzpicture}[anchorbase]
	\draw[gmod] (0.5, 0) -- (0.25, 0.25);
	\draw[gmod] (0, 0) -- (0.25, 0.25) -- (0.5, 0.5);
	\draw[gmod] (0.75, 0) -- (0.75, 0.25) -- (0.5, 0.5) -- (0.5, 0.75);
	\adot{(0.125, 0.125)};
	\adot{(0.375, 0.125)};
	\adot{(0.625, 0.375)};
	\adot{(0.375, 0.375)};
	\end{tikzpicture}
	\ +\
	\begin{tikzpicture}[anchorbase]
	\draw[gmod] (0, -0.375) -- (0, -0.25) -- (0.25, 0) -- (0.5, 0.25);
	\draw[gmod] (0.25, -0.375) -- (0.25, -0.25) -- (0, 0) -- (0, 0.25) -- (0.25, 0.5);
	\draw[gmod] (0.75, -0.375) -- (0.75, -0.25) -- (0.75, 0) -- (0.5, 0.25) -- (0.25, 0.5) -- (0.25, 0.75);
	\adot{(0.125, 0.375)};
	\adot{(0.375, 0.375)};
	\adot{(0.375, 0.125)};
	\adot{(0.625, 0.125)};
	\end{tikzpicture}
	\overset{\eqref{eq:brauerLieIdentity}}{=}
	\begin{tikzpicture}[anchorbase]
	\draw[gmod] (0.5, 0) -- (0.25, 0.25);
	\draw[gmod] (0, 0) -- (0.25, 0.25) -- (0.5, 0.5);
	\draw[gmod] (0.75, 0) -- (0.75, 0.25) -- (0.5, 0.5) -- (0.5, 0.75);
	\adot{(0.25, 0.25)};
	\adot{(0.5, 0.5)};
	\end{tikzpicture}
	\ +\
	\begin{tikzpicture}[anchorbase]
	\draw[gmod] (0, -0.375) -- (0, -0.25) -- (0.25, 0) -- (0.5, 0.25);
	\draw[gmod] (0.25, -0.375) -- (0.25, -0.25) -- (0, 0) -- (0, 0.25) -- (0.25, 0.5);
	\draw[gmod] (0.75, -0.375) -- (0.75, -0.25) -- (0.75, 0) -- (0.5, 0.25) -- (0.25, 0.5) -- (0.25, 0.75);
	\adot{(0.5, 0.25)};
	\adot{(0.25, 0.5)};
	\end{tikzpicture},
	\end{multline*}
	so $\begin{tikzpicture}[anchorbase]
	\draw[gmod] (0, 0) -- (0.25, 0.25) -- (0.25, 0.5);
	\draw[gmod] (0.5, 0) -- (0.25, 0.25);
	\adot{(0.25, 0.25)};
	\end{tikzpicture}$ satisfies \eqref{eq:jacobi}. A similar calculation shows that it satisfies \eqref{eq:skew}, and thus $\green{L}$ is indeed a Lie algebra.
	\details{\begin{equation*}
		\begin{tikzpicture}[anchorbase]
		\draw[gmod] (0.5, -0.5) -- (0, 0) -- (0.25, 0.25) -- (0.25, 0.5);
		\draw[gmod] (0, -0.5) -- (0.5, 0) -- (0.25, 0.25);
		\adot{(0.25, 0.25)};
		\end{tikzpicture}
		\overset{\eqref{eq:brauerLieIdentity}}{=}
		\begin{tikzpicture}[anchorbase]
		\draw[gmod] (0.5, -0.5) -- (0, 0) -- (0.25, 0.25) -- (0.25, 0.5);
		\draw[gmod] (0, -0.5) -- (0.5, 0) -- (0.25, 0.25);
		\adot{(0, 0)};
		\adot{(0.5, 0)};
		\end{tikzpicture}
		=
		\begin{tikzpicture}[anchorbase]
		\draw[gmod] (0.5, -0.5) -- (0, 0) -- (0.25, 0.25) -- (0.25, 0.5);
		\draw[gmod] (0, -0.5) -- (0.5, 0) -- (0.25, 0.25);
		\adot{(0.125, -0.375)};
		\adot{(0.375, -0.375)};
		\end{tikzpicture}
		\overset{\eqref{eq:skew}}{=}
		-\ \begin{tikzpicture}[anchorbase]
		\draw[gmod] (0, 0) -- (0.25, 0.25) -- (0.25, 0.5);
		\draw[gmod] (0.5, 0) -- (0.25, 0.25);
		\adot{(0.125, 0.125)};
		\adot{(0.375, 0.125)};
		\end{tikzpicture}
		\overset{\eqref{eq:brauerLieIdentity}}{=}
		-\ \begin{tikzpicture}[anchorbase]
		\draw[gmod] (0, 0) -- (0.25, 0.25) -- (0.25, 0.5);
		\draw[gmod] (0.5, 0) -- (0.25, 0.25);
		\adot{(0.25, 0.25)};
		\end{tikzpicture}.
		\end{equation*}}
	Next, we have that $\frac{1}{2}\left(\begin{tikzpicture}[anchorbase]
	\draw[-, thick] (0, 0) -- (0, 0.75);
	\draw[-, thick] (0.6, 0) -- (0.6, 0.2) arc(0:180:0.2) -- (0.2, 0);
	\end{tikzpicture}
	\ -\ 
	\begin{tikzpicture}[anchorbase]
	\draw[bmod] (0, 0) -- (0, 0.25) arc(180:0:0.25) -- (0.5, 0);
	\draw[bmod] (0.25, 0) -- (0.25, 0.25) to[out=90, in=270] (0, 0.5) -- (0, 0.75);
	\end{tikzpicture}\right) = \begin{tikzpicture}[anchorbase]
	\draw[-, thick] (0, 0) -- (0, 0.75);
	\draw[-, thick] (0.6, 0) -- (0.6, 0.2) arc(0:180:0.2) -- (0.2, 0);
	\end{tikzpicture} \circ \begin{tikzpicture}[anchorbase]
	\draw[gmod] (0, 0) -- (0, 0.5);
	\draw[bmod] (0.25, 0) -- (0.25, 0.5);
	\adot{(0, 0.25)};
	\end{tikzpicture}$\ , and hence it is a morphism from $\green{L} \otimes (\go, \id_{\go})$ to $(\go, \id_{\go})$. Lemma \ref{lem:orientedBrauerLieObject} tells us that $\left(\go, \begin{tikzpicture}[anchorbase]
	\draw[-, thick] (0, 0) -- (0, 0.75);
	\draw[-, thick] (0.6, 0) -- (0.6, 0.2) arc(0:180:0.2) -- (0.2, 0);
	\end{tikzpicture}\right)$ is a module for $\left(\go\go, \begin{tikzpicture}[anchorbase]
	\draw[gmod] (0, 0) -- (0.25, 0.25) -- (0.25, 0.5);
	\draw[gmod] (0.5, 0) -- (0.25, 0.25);
	\end{tikzpicture}\right)$. A straightforward calculation similar to the one above then shows that $\left(\left(\go, \id_{\go}\right),\ \frac{1}{2}\left(\begin{tikzpicture}[anchorbase]
	\draw[-, thick] (0, 0) -- (0, 0.75);
	\draw[-, thick] (0.6, 0) -- (0.6, 0.2) arc(0:180:0.2) -- (0.2, 0);
	\end{tikzpicture}
	\ -\ 
	\begin{tikzpicture}[anchorbase]
	\draw[bmod] (0, 0) -- (0, 0.25) arc(180:0:0.25) -- (0.5, 0);
	\draw[bmod] (0.25, 0) -- (0.25, 0.25) to[out=90, in=270] (0, 0.5) -- (0, 0.75);
	\end{tikzpicture}\right) \right)$ is an $\green{L}$-module. Applying the definition from Lemma \ref{lem:dualModule} immediately shows that this $\green{L}$-module is self-dual.
	\details{Write $\begin{tikzpicture}[anchorbase]
		\draw[gmod] (0, 0) -- (0.25, 0.25);
		\draw[bmod] (0.5, 0) -- (0.25, 0.25) -- (0.25, 0.5);	
		\end{tikzpicture} = \begin{tikzpicture}[anchorbase]
		\draw[bmod] (0, 0) -- (0, 0.75);
		\draw[bmod] (0.6, 0) -- (0.6, 0.2) arc(0:180:0.2) -- (0.2, 0);
		\end{tikzpicture}$ for the $\go\go$-action on $\go$ in $\B$, and $\begin{tikzpicture}[anchorbase]
		\draw[gmod] (0, 0) -- (0.25, 0.25);
		\draw[bmod] (0.5, 0) -- (0.25, 0.25) -- (0.25, 0.5);
		\adot{(0.25, 0.25)};
		\end{tikzpicture} = \begin{tikzpicture}[anchorbase]
		\draw[gmod] (0, 0) -- (0.25, 0.25);
		\draw[bmod] (0.5, 0) -- (0.25, 0.25) -- (0.25, 0.5);
		\adot{(0.125, 0.125)};
		\end{tikzpicture}$ for the $L$-action on $(\go, \id_{\go})$ in $\Kar(\B)$. We have:
		\begin{equation*}
		\begin{tikzpicture}[anchorbase]
		\draw[gmod] (0.25, 0) -- (0.5, 0.25);
		\draw[gmod] (0, 0) -- (0, 0.25) -- (0.25, 0.5);
		\draw[bmod] (0.75, 0) -- (0.5, 0.25) -- (0.25, 0.5) -- (0.25, 0.75);
		\adot{(0.5, 0.25)};
		\adot{(0.25, 0.5)};
		\end{tikzpicture}
		=
		\begin{tikzpicture}[anchorbase]
		\draw[gmod] (0.25, 0) -- (0.5, 0.25);
		\draw[gmod] (0, 0) -- (0, 0.25) -- (0.25, 0.5);
		\draw[bmod] (0.75, 0) -- (0.5, 0.25) -- (0.25, 0.5) -- (0.25, 0.75);
		\adot{(0, 0.125)};
		\adot{(0.375, 0.125)};
		\end{tikzpicture}
		\overset{\eqref{eq:lmod}}{=}
		\begin{tikzpicture}[anchorbase]
		\draw[gmod] (0.5, 0) -- (0.25, 0.25);
		\draw[gmod] (0, 0) -- (0.25, 0.25) -- (0.5, 0.5);
		\draw[bmod] (0.75, 0) -- (0.75, 0.25) -- (0.5, 0.5) -- (0.5, 0.75);
		\adot{(0.125, 0.125)};
		\adot{(0.375, 0.125)};
		\end{tikzpicture}
		\ +\
		\begin{tikzpicture}[anchorbase]
		\draw[gmod] (0, -0.375) -- (0, -0.25) -- (0.25, 0) -- (0.5, 0.25);
		\draw[gmod] (0.25, -0.375) -- (0.25, -0.25) -- (0, 0) -- (0, 0.25) -- (0.25, 0.5);
		\draw[bmod] (0.75, -0.375) -- (0.75, 0) -- (0.5, 0.25) -- (0.25, 0.5) -- (0.25, 0.75);
		\adot{(0, -0.25)};
		\adot{(0.25, -0.25)};
		\end{tikzpicture} 
		\overset{\eqref{eq:brauerLieIdentity}}{=} 
		\begin{tikzpicture}[anchorbase]
		\draw[gmod] (0.5, 0) -- (0.25, 0.25);
		\draw[gmod] (0, 0) -- (0.25, 0.25) -- (0.5, 0.5);
		\draw[bmod] (0.75, 0) -- (0.75, 0.25) -- (0.5, 0.5) -- (0.5, 0.75);
		\adot{(0.125, 0.125)};
		\adot{(0.375, 0.125)};
		\adot{(0.375, 0.375)};
		\end{tikzpicture}
		\ +\
		\begin{tikzpicture}[anchorbase]
		\draw[gmod] (0, -0.375) -- (0, -0.25) -- (0.25, 0) -- (0.5, 0.25);
		\draw[gmod] (0.25, -0.375) -- (0.25, -0.25) -- (0, 0) -- (0, 0.25) -- (0.25, 0.5);
		\draw[bmod] (0.75, -0.375) -- (0.75, 0) -- (0.5, 0.25) -- (0.25, 0.5) -- (0.25, 0.75);
		\adot{(0, 0.25)};
		\adot{(0.25, 0)};
		\end{tikzpicture}
		=
		\begin{tikzpicture}[anchorbase]
		\draw[gmod] (0.5, 0) -- (0.25, 0.25);
		\draw[gmod] (0, 0) -- (0.25, 0.25) -- (0.5, 0.5);
		\draw[bmod] (0.75, 0) -- (0.75, 0.25) -- (0.5, 0.5) -- (0.5, 0.75);
		\adot{(0.25, 0.25)};
		\adot{(0.5, 0.5)};
		\end{tikzpicture}
		\ +\
		\begin{tikzpicture}[anchorbase]
		\draw[gmod] (0, -0.375) -- (0, -0.25) -- (0.25, 0) -- (0.5, 0.25);
		\draw[gmod] (0.25, -0.375) -- (0.25, -0.25) -- (0, 0) -- (0, 0.25) -- (0.25, 0.5);
		\draw[bmod] (0.75, -0.375) -- (0.75, 0) -- (0.5, 0.25) -- (0.25, 0.5) -- (0.25, 0.75);
		\adot{(0.5, 0.25)};
		\adot{(0.25, 0.5)};
		\end{tikzpicture},
		\end{equation*}
		as desired.}
	The claims about $I_n'$-images follow from direct calculations; note that $I_n'\left(\go\go, \begin{tikzpicture}[anchorbase]
	\draw[gmod] (0, 0) -- (0.25, 0.25) -- (0.25, 0.5);
	\draw[gmod] (0.5, 0) -- (0.25, 0.25);
	\end{tikzpicture}\right)$ is isomorphic to $\gl_n$ equipped with the usual Lie bracket, and $I_n'\left(\begin{tikzpicture}[anchorbase]
	\draw[gmod] (0, 0) -- (0, 0.5);
	\adot{(0, 0.25)};
	\end{tikzpicture}\right)$ is projection onto $\so_n$.
	\details{Recall that $I_n'(\go) = \kk^n$, and that $I_n'(\go) \otimes I_n'(\go) \isom \gl_n$ via the map $v \otimes w \mapsto v w^\top$ (the outer product). We have $I_n'\left(\dcross\right)(v \otimes w) = w \otimes v$, so applying this isomorphism and extending linearly, we see that $I_n'\left(\dcross\right)$ is the transpose map $X \mapsto X^\top$ on $\gl_n$. We know that any given matrix $X \in \gl_n$ decomposes into the sum of a symmetric matrix and a skew symmetric matrix, say $X = A + B$. Then:
		\begin{equation*}
		I_n'\left(\begin{tikzpicture}[anchorbase]
		\draw[gmod] (0, 0) -- (0, 0.5);
		\adot{(0, 0.25)};
		\end{tikzpicture}\right)(X) 
		=
		\frac{1}{2}\left(X - X^\top \right)
		=
		\frac{1}{2}\left(A + B - A^\top - B^\top \right)
		=
		\frac{1}{2}(A + B - A + B)\\
		=
		B,
		\end{equation*}
		which is the skew-symmetric part of $X$, as desired.}
\end{proof}

Constructions similar to those in Lemma \ref{lem:unorientedBrauerObject} can be applied to many diagrammatic categories that are related to the representations of Lie algebras. Let $\g$ be a finite-dimensional Lie algebra in $\kkvec$ and $\rho \colon \g \to \gl(V)$ a faithful representation (for some finite-dimensional vector space $V$). Suppose that $\CC$ is a symmetric monoidal $\kk$-linear category containing dual objects $\goup$ and $\godown$, and that $F \colon \CC \to \g\dashmod$ is a functor such that $F(\goup) = V$. Suppose further that $e \colon \goup\godown \to \goup\godown$ is an idempotent endomorphism such that the extension of $F$ to $\Kar(\CC)$ maps $(\goup\godown, e)$ to $\im(\rho) \isom \g$ , i.e.\ the adjoint $\g$-module. Then, as in Example \ref{ex:orientedBrauerLieObject}, $F$ maps the Lie algebra
	$$\left(\left(\goup\godown, e \right), e \circ \left(\begin{tikzpicture}[anchorbase]
	\draw[->, thick] (0, 0) to[out=90, in=270] (0.25, 0.5) -- (0.25, 0.75);
	\draw[<-, thick] (0.25, 0) -- (0.25, 0.2) arc(180:0:0.125) -- (0.5, 0);
	\draw[<-, thick] (0.75, 0) to[out=90, in=270] (0.5, 0.5) -- (0.5, 0.75);
	\end{tikzpicture}
	\ -\
	\begin{tikzpicture}[anchorbase]
	\draw[->, thick] (0.5, -0.75) -- (0.5, -0.5) to[out=90, in=270] (0, 0) to[out=90, in=270] (0.25, 0.5) -- (0.25, 0.75);
	\draw[<-, thick] (0.75, -0.75) -- (0.75, -0.5) to[out=90, in=270] (0.5, 0) arc(0:180:0.125) to[out=270, in=90] (0, -0.5) -- (0, -0.75);
	\draw[<-, thick] (0.25, -0.75) -- (0.25, -0.5) to[out=90, in=270] (0.75, 0) to[out=90, in=270] (0.5, 0.5) -- (0.5, 0.75);
	\end{tikzpicture} \right) \circ (e \otimes e)\right)$$
	in $\Kar(\CC)$ to $\im(\rho)$ equipped with the Lie bracket inherited from $\gl(V)$, and $F$ maps the modules $\left(\left(\goup, \id_{\go}\right), \begin{tikzpicture}[anchorbase]
	\draw[->, thick] (0, 0) -- (0, 0.75);
	\draw[->, thick] (0.6, 0) -- (0.6, 0.2) arc(0:180:0.2) -- (0.2, 0);
	\end{tikzpicture} \circ (e \otimes \id_{\go}) \right)$ and $\left(\left(\godown, \id_{\go}\right), -\ \begin{tikzpicture}[anchorbase]
	\draw[->, thick] (0, 0) -- (0, 0.25) arc(180:0:0.25) -- (0.5, 0);
	\draw[<-, thick] (0.25, 0) -- (0.25, 0.25) to[out=90, in=270] (0.5, 0.5) -- (0.5, 0.75);
	\end{tikzpicture} \circ (e \otimes \id_{\go}) \right)$ to $V$ (the natural module for $\im(\rho)$) and its dual $V^*$, respectively. (Note that one needs to check that these definitions do in fact yield a Lie algebra and a pair of modules in $\Kar(\CC)$. For the cases discussed below, this follows easily from essentially the same arguments as in Lemma \ref{lem:unorientedBrauerObject}.)
	
For instance, one could take $\g$ to be the simple Lie algebra of type $F_4$, $\CC$ the $F_4$ diagrammatic category defined and studied in \cite{f4diagrammatics} (with parameters $\alpha = \frac{7}{3}, \delta = 26$), $\goup$ and $\godown$ to both be the generating object $\go$ of $\CC$, $F$ the functor $\Phi \colon \CC \to \g\dashmod$ discussed in \cite[\S5]{f4diagrammatics}, and $e$ to be the idempotent $e_1$ defined at the start of Section 6 of that paper (see \cite[Thm.~6.4]{f4diagrammatics}, and note that $V_{\omega_1}$ denotes the adjoint module). The constructions outlined above would then yield objects in $\CC$ corresponding to $\g$ and its (self-dual) natural module.
	
One could similarly apply these definitions to the $G_2$ diagrammatic category (with quantum parameter set to $q = 1$) studied in \cite{kuperbergG2-2} and \cite{kuperbergG2-1} to obtain objects corresponding to the simple Lie algebra of type $G_2$ and its natural module.
	
For a final family of examples, consider the oriented and unoriented Frobenius Brauer supercategories $\OB(A)$ and $\B(A)$ (see \cite{orientedFrobeniusBrauer} and \cite{diagrammaticsForRealSupergroups}), where $A$ is a Frobenius superalgebra -- the content of the present paper generalizes straightforwardly to the case of supercategories. The constructions outlined above would yield objects corresponding to the Frobenius superalgebras $\gl(m|n, A)$ and $\osp(m|2n, A)$, and their natural supermodules.

\bibliographystyle{alphaurl}
\bibliography{LieObjects}

\end{document}